%% file: num_sm.tex
\documentclass[final]{siamltex}
\usepackage {epsfig,amsmath,euscript}
\usepackage[latin1]{inputenc}
\usepackage[english]{babel}
\usepackage{color}
\usepackage{amsmath}
\usepackage{amssymb}
\usepackage{ae}
\usepackage{subfigure}
\usepackage[T1]{fontenc}
\usepackage{graphicx}
\usepackage{epstopdf}
\usepackage{color}
\definecolor{rred}{rgb}{0.7,0.0,0.2}
\definecolor{bblue}{rgb}{0.2,0.0,0.7}

\include{texdef}

\begin{document}

\title {Computation of saddle type slow manifolds using iterative methods}

\author {K. Uldall Kristiansen\thanks{The author was funded by a H. C. {\O}rsted post doc grant.}} 
\date {}
\maketitle

% \vspace* {-2em}
\begin{center}
\begin{tabular}{c}
Department of Applied Mathematics and Computer Science, \\
Technical University of Denmark, \\
2800 Kgs. Lyngby, \\
DK
% % \begin{lrbox}
% %  
% %  \end{lrbox}

% and\\
% Department of Mathematics, \\
% University of Surrey, \\
% Guildford, GU2 7XH\\
% DK \\
\end{tabular}
\end{center}

\begin{abstract}
This paper presents an alternative approach for the computation of trajectory segments on slow manifolds of saddle type. This approach is based on iterative methods rather than collocation-type methods. Compared to collocation methods, that require mesh refinements to ensure uniform convergence with respect to $\epsilon$, appropriate estimates are directly attainable using the method of this paper. The method is applied to several examples including: A model for a pair of neurons coupled by reciprocal inhibition with two slow and two fast variables and to the computation of homoclinic connections in the FitzHugh-Nagumo system. %\ed{The method is also applied to an example where Fenichel's theory does not apply but where there exists an almost invariant slow manifold. }
% In this paper we present some results on the numerical implementation of some methods developed by the authors elsewhere for the computation of slow manifolds. A particular focus will be on demonstrating how the approximation of fibers can be used to compute fairly accurate approximations of solutions initially approaching and finally excaping canard solutions. 
\end{abstract}

%  
% \tableofcontents

% \baselineskip 24pt % Double spacingnow
% \newpage
 \input{intro}

\input{theIterativeMethods}
\input{theSOSMSTMethod}

\input{numericalImplementationSOAndSOF}
\input{examples1}
\input{conclusion}
% \newpage
\appendix
\input{errorApp}
\bibliography{refs}
\bibliographystyle{plain}
% \newpage
% \input{reply}
 \end{document}

%% file: texdef.tex
\newcommand{\be}{\begin{eqnarray}}
\newcommand{\ee}{\end{eqnarray}}
\newcommand{\ba}{\begin{align}}
\newcommand{\ea}{\end{align}}
\newcommand{\bi}{\begin{itemize}}
\newcommand{\ei}{\end{itemize}}

\newcommand{\R}{\mathbb R}

\newcommand{\beq}[1]{\begin{equation} \label{#1}}
\newcommand{\eeq}{\end{equation}}
\newcommand{\beqa}{\begin{eqnarray}}
\newcommand{\eeqa}{\end{eqnarray}}
\newcommand{\bal}{\begin{align}}
\newcommand{\eal}{\end{align}}
\newcommand{\bsub}{\begin{subequations}}
\newcommand{\esub}{\end{subequations}}
\newcommand{\eqlab}[1]{\label{eq:#1}}
\renewcommand{\eqref}[1]{(\ref{eq:#1})}
\newcommand{\eqsref}[2]{(\ref{eq:#1}) and~(\ref{eq:#2})}

\newcommand{\figref}[1]{Fig.~\ref{fig:#1}}
\newcommand{\figlab}[1]{\label{fig:#1}}

\newcommand{\secref}[1]{section~\ref{sec:#1}}
\newcommand{\Secref}[1]{Section~\ref{sec:#1}}
\newcommand{\seclab}[1]{\label{sec:#1}}

\newcommand{\remref}[1]{Remark~\ref{remark:#1}}
\newcommand{\remlab}[1]{\label{remark:#1}}
\newcommand{\corref}[1]{Corollary~\ref{Corollary:#1}}
\newcommand{\corlab}[1]{\label{Corollary:#1}}

\newcommand{\propref}[1]{Proposition~\ref{proposition:#1}}
\newcommand{\proplab}[1]{\label{proposition:#1}}
\newcommand{\appref}[1]{App.~\ref{app:#1}}
\newcommand{\applab}[1]{\label{app:#1}}

%% file: intro.tex
\section{Introduction}
Slow-fast systems of the form
\begin{align}
 \dot x & = \epsilon X(x,y),\quad \dot y =Y(x,y),\eqlab{fs}
\end{align}
or equivalently
\begin{align}
 x' & = X(x,y),\quad y' =\epsilon^{-1}Y(x,y),\quad X,Y\in C^r,\,C^\infty\,\mbox{or}\,C^\omega,\eqlab{ss}
\end{align}
with $x\in \mathbb R^{n_s}$ and $y\in \mathbb R^{n_f}$ being the slow and fast variables, respectively, arise in a wide variety of scientific problems. Here $\dot{()}$ denotes the derivative with respect to the fast time $t$ whereas $()^\prime$ denotes differentiation with respect to the slow time $\tau=\epsilon t$. The vector-fields $X$ and $Y$ may in
general also depend upon the constant $\epsilon$ that measures the time-scale separation. For simplicity, however, the $\epsilon$-dependency shall in this paper always suppressed. Slow-fast systems appear 
in neuroscience \cite{domijan2006,rubin2002,rubin2007,row1,ski1,rin2}, chemical reaction dynamics \cite{olsen1983}, laser systems \cite{braza1989,dubbeldam1999,erneux1988,erneux1981a,erneux1981b}, meteorology and short-term weather forecasting \cite{lor1,lor2,lor3,tem1,van1}, molecular physics and the Born-Oppenheimer approximation \cite{McQ1}, the evolution and stability of the solar system \cite{las1,las2}, modeling of water waves in the presence of surface tension \cite{amick1989}, and the modeling of tethered satellites \cite{kri1,kri2}. The identification
of slow and fast variables is extremely useful because of dimension reduction. Indeed, the two limit systems \eqref{fs}$_{\epsilon=0}$ and \eqref{ss}$_{\epsilon=0}$ enable in many cases a description of the system with $\epsilon>0$ but sufficiently small. The actual identification of a time-scale separation parameter $\epsilon$ in a particular problem can, however, be a challenging task, even in planar problems see e.g. \cite{bro2}.

% by which all the fast variables are 
% ``slaved'' to the slow ones through the \textit{slow manifold}. Dimension reduction is one of the main 
% aims and tools for a dynamicist and the elimination of fast variables is very useful in for example 
% numerical computations. The identification of a time-scale separation parameter $\epsilon$ in particular problem can be an extremely challenging task, even in planar problems \cite{bro2}, but it is powerful because it induces a natural splitting of the variables  

Although all of the problems mentioned above can be written in the form of \eqref{fs} or \eqref{ss}, they are typically dynamically very different. Some are dissipative and all the interesting dynamics takes place on a lower dimensional manifold \cite{lam2,micmen1}. Others are conservative and oscillatory \cite{alfi,amick1989,lor1,kri2}. In this case there is no complete theory (except for the case with only one slow and one fast degree of freedom \cite{arnold2007mathematical,geller2}) that relates the two limit systems \eqref{fs}$_{\epsilon=0}$ and \eqref{ss}$_{\epsilon=0}$ to $\epsilon>0$ but small. Finally, there are cases where different lower dimensional objects interact through stable and unstable manifolds to form very non-trivial dynamics, see e.g. \cite{desroches2014,guck4,olsen1983}. In dynamical systems, numerical computations can often offer great insight. However, in slow-fast systems with both attracting and repelling lower dimensional manifolds the time scale separation makes the computation of such dynamics a challenging task \cite{guc3}.  

% \textbf{Numerical methods for slow-fast systems}.
% Important examples include: Meteorology and short-term weather forecasting \cite{lor2,lor1,tem1}, molecular physics and the Born-Oppenheimer approximation \cite{McQ1}, chemical enzyme kinetics and the Michaelis-Menten mechanism \cite{micmen1}, predator-prey and reaction-diffusion models \cite{Mur1}, the evolution and stability of the solar system \cite{las1,las2} and the modeling of tethered satellites \cite{kri1,kri2}. These systems can also be ``artificially constructed'' by a partial scaling of variables near a bifurcation \cite{romtur1}. The main advantage of identifying slow and fast variables is dimension reduction by which all the fast variables are ``slaved'' to the slow ones through the \textit{slow 
% manifold}. 

%Fenichel's theory \cite{fen1,feb2} establishes the existence of normally hyperbolic slow manifolds. 
\textbf{Slow-fast theory}. Consider a compact set of constrained equilibria $M_0=\{(x,y)\vert Y(x,y)=0\}$ with the spectrum $\text{spec}\,(\partial_y Y\vert_{M_0})$ satisfying
\begin{align}
\text{dist}\,(\text{spec}\,(\partial_y Y\vert_{M_0}),i\mathbb R)\ge c>0,\quad  \mbox{$c$ independent of $\epsilon$}.\eqlab{M0h}
\end{align}
Here $\partial_y Y$ is the Jacobian of $Y(x,\cdot)$. 
%Then $M_0$ is said to be normally hyperbolic. 
Condition \eqref{M0h} implies, by the implicit function theorem, that $M_0$ is a graph of some function 
\begin{align}
y=\eta_0(x),\eqlab{eta0}
\end{align}
that is $M_0=\{(x,y)\vert y=\eta_0(x)\}$. For $\epsilon=0$ this manifold $M_0$ is a fixed point set for \eqref{fs} which is normally hyperbolic. It is referred to as the critical manifold. Fenichel's theory \cite{fen1,fen2} then applies to $M_0$ so that there exists an invariant manifold $M_h=\{y=\eta(x)\}$, with $\eta$ smooth, which is $\mathcal O(\epsilon)$-close to $M_0$. The slow manifold $M_h$ is attracting if $\mbox{spec}\,(\partial_y Y\vert_{M_0})\subset \{z\in \mathbb C\vert \text{Re}z<0\}$ or repelling if $\mbox{spec}\,(\partial_y Y\vert_{M_0})\subset \{z\in \mathbb C\vert \text{Re}z>0\}$. Otherwise it is of saddle type. In this case there are both a stable 
manifold $W^s(M_h)$, on which 
trajectories are attracted 
exponentially fast towards $M_h$ forward in time, and an unstable manifold $W^u(M_h)$, on 
which trajectories are 
attracted exponentially fast towards $M_h$ backwards in time \cite{jon1}. Fenichel's theory also says that $W^s(M_h)$ and $W^u(M_h)$ are $\mathcal O(\epsilon)$-close to the stable and unstable manifolds of the fix point set $M_0\vert_{\epsilon=0}$ of \eqref{fs}$_{\epsilon=0}$. The normally hyperbolic slow manifolds are like center manifolds \cite{car1} but as opposed to center manifolds, slow manifolds are only local in the fast variables. {Slow manifolds are "global" in the slow variables in the sense that Fenichel's description of these objects only fails locally where \eqref{M0h} is violated.}%Slow manifolds are ``global'' in the slow variables in the sense that they are only (potentially) destroyed near points where \eqref{M0h} is violated. 

If on the other hand $\mbox{spec}\,(\partial_y Y\vert_{M_0})$ is not disjoint from the imaginary axis, but instead only satisfies
\begin{align}
\text{dist}\,(\mbox{spec}\,(\partial_y Y\vert_{M_0}),0)\ge c>0,\eqlab{fastcond}
\end{align}
$c$ independent of $\epsilon$, then the motion normal to $M_0$ is still fast but there is in general no invariant slow manifold nearby \cite{mac1}. However, if the vector-field 
\begin{align}
U=\begin{pmatrix}\epsilon X\\Y\end{pmatrix} \eqlab{Ufld}
\end{align}
is analytic then there is in this case some $M_e$ on which the restriction of the vector-field has exponentially small angle $\mathcal O(e^{-c/\epsilon})$ with the tangent space \cite{geller1,kri3}. The slow manifold $M_e$ is therefore exponentially close to being invariant. This holds even in the normally elliptic case where $\mbox{spec}\,(\partial_y Y\vert_{M_0})\subset i\mathbb R$ which is relevant for Hamiltonian systems. Only in the case of one fast degree of freedom does there exist a theory for the description of the fast dynamics off the slow manifold \cite{geller1}.

\textbf{Numerical methods}. There are traditionally two numerical approaches for the computation of slow manifolds. The first approach is to use collocation in the solution of an associated boundary value problem. 
The advantages of using a collocation based approach are many. One advantage is nonlinear differential equations are effectively replaced with nonlinear algebraic ones and the method therefore circumvents issues related to dynamic stability. This enables the computation of highly unstable orbit segments. The nonlinear algebraic equations can be solved by Newton's method provided a good initial guess is known. Collocation based approaches are also highly adaptable and can be directly integrated within the AUTO bifurcation analysis software \cite{doedel2000} to perform bifurcation analysis. The second approach for the computation of slow manifolds is simply to use direct integration (also called ``the sweeping method'' \cite{desroches2014}). {Direct integration is easy to use.} Also whereas a collocation method requires an accurate initial guess to converge, direct integration can be used to explore the phase space. In fact, an initial guess for a collocation approach is often obtained using direct integration. Direct integration, however, has some documented disadvantages, see e.g. \cite{england2004computing}. In particular this approach is limited to the computation of attracting slow manifolds (by forward integration) and repelling slow manifolds (by backward integration). The computation of trajectories following saddle type slow manifolds $M_h$ for a long time, $t= \mathcal O(\epsilon^{-1})$ or $\tau = \mathcal O(1)$, cannot be achieved by any ``stiff'' integration method. Even an exact initial value solver in the presence of round-off errors of magnitude $\delta$ will amplify this error to unit size in a time of order $\mathcal O(\epsilon \log \delta^{-1})$ \cite{guc3}. Such highly unstable orbit segments will be referred to as canards or more accurately canard segments.

There are many examples (e.g. Van der Pol system \cite{guck5}, model for reciprocal inhibition \cite{guc3,guck4}, FitzHugh-Nagumo \cite{guc2,guc3,guck6,jon1,kru1}) where important orbits have {canard segments}. Such orbits are referred to as canard orbits and these were first analyzed in planar slow-fast planar systems by Beno{\^i}t et al \cite{benoit1981}. They found canard orbits as stable limit cycles that only existed in an exponentially small parameter regime. They appeared as the intersections of attracting and repelling slow manifolds. In $\mathbb R^3$ with two slow variables and only one fast, canard orbits appear persistently. Collocation based methods have in general proven very useful for the computational analysis of such canards, see e.g. \cite{desroches2010}. However, it is also possible to compute these orbits in $\mathbb R^3$ by a simpler approach using direct integration combined with shooting to a section by applying forward integration on the attracting slow manifold and backwards integration on the repelling one \cite{guc7,wechselberger2005}. For canard segments on saddle-type slow manifolds there exists to date, to the author's best knowledge, no alternative to collocation methods.

% Many slow-fast systems can be analyzed by direct numerical integration and 

% \note{SOMETHING HERE. Use \cite{desroches2014} as reference. Explain sweeping and collocation. and then move on to following section for a detailed description of SMST . Recall Referee 3}

\textbf{SMST algorithm}. Guckenheimer and Kuehn in \cite{guc3} developed an algorithm \textsc{SMST} (\textit{Slow Manifolds} of \textit{Saddle Type}) based on collocation for the computation of trajectories near a saddle-type slow manifold. 
The \textsc{SMST} method starts from an initial guess provided by the reduced system:
\begin{align}
 x' = X(x,\eta_0(x)),\eqlab{sss}
\end{align}
with $x(0)=x_0$ and $\tau\in [0,T]$. Here $T=\mathcal O(1)$ with respect to $\epsilon$. Set $z=(x,y)$ and let $x_T=x(T)$. The \textsc{SMST} algorithm then solves for a solution $z=z(\tau)$ that approaches the slow manifold near $z_0\equiv (x_0,\eta_0(x_0))$ and exits it near $z_T\equiv (x_T,\eta_0(x_T))$.  For this time is discretized $0=\tau_0<\tau_1<\cdots <\tau_N=T$ and on each mesh $\tau_{i}\le \tau\le \tau_{i+1}$ the $z=z(\tau)$ is replaced by a cubic interpolation based on the values $z_i\equiv z(\tau_i)$, $z_{i+1}\equiv z(\tau_{i+1})$ and the tangent vectors $z_i'\equiv V(z_i)$, $z_{i+1}' \equiv V(z_{i+1})$. Here $V=\epsilon^{-1} U$ with $U$ given in \eqref{Ufld}. The dynamical constraint $z' = V(z)$ is then enforced at the mid-points $\tau_{i+1/2} \equiv \frac12 (\tau_i + \tau_{i+1})$ using this cubic interpolation of $z=z(\tau)$. See also Eq. (2.1) in \cite{guc3}. This gives $n\times N$ equations for the $n\times (N+1)$ unknowns $z_0,\,z_1,\,\ldots,z_N$. The remaining $n$ equations are obtained from the boundary conditions which may be included in the following way. By 
assumption the matrix $\partial_y Y(z)$ introduces 
a splitting of the form $E_s^{z}\oplus E_u^{z}= \mathbb R^{n_f}$ where $E_s^{z}=E_s^z(x)$ and $E_u^{z}=E_u^z(x)$ can be interpreted as the stable and unstable eigenspaces of the \textit{constrained hyperbolic equilibria} $y=\eta(x)$ of $\dot y = Y(x,y)$, $x$ here being constrained as a parameter. Fenichel's theory guarantees that $E_s^z$ and $E_u^z$ are transverse to $W^u(M_h)$ and $W^s(M_h)$, respectively. %Set $\text{dim}\,E_s^z = n_f^s$ and $\text{dim}\,E_u^z=n_f^u=n-n_f^s$. The numbers $n_f^s$ and $n_f^u$ do not depend upon $z$. 
Let 
\begin{align}
\pi_s^z:\quad &\mbox{\text{The projection onto $E_s^z$}},\eqlab{pis}\\
\pi_u^z:\quad &\mbox{\text{The projection onto $E_u^z$}}.\nonumber
\end{align}
 Then at $\tau=0$ one specifies $x(0)=x_0$ and ``the stable components'' of $y(0)=\eta_0(x_0)+y_{s0}+y_{u0}$ by fixing the value of
\begin{align}
y_{s0}=\pi_s^{z_0} (y(0)-\eta_0(x_0)).\eqlab{scomp}
\end{align}
In \cite{guc3} the value is fixed to $0$. At $\tau=T$, on the other hand, one specifies the ``unstable components'' of $y(T)=\eta_0(x_{T})+ y_{sT}+y_{uT}$ by fixing the value of
\begin{align}
y_{uT}=\pi_u^{z_T} (y(T)-\eta_0(x_T)).\eqlab{ucomp} 
\end{align}
The value is set to $0$ in \cite{guc3}. % it is fixed to $0$.
Here $\eta_0(x_T)$ is the value of the fast variables, when using \eqref{sss} for the propagation of the slow variables, at $\tau=T$. From the cubic interpolation the sparse Jacobian can be computed explicitly and a Newton method can be used to obtain an accurate solution. %Note that the associated linear problems that appear in the Newton method are ill-conditioned with coefficient matrices having a block decomposition with blocks having the following form 
% \begin{align}
% \begin{pmatrix}
%                                                      \mathcal O(1) & \mathcal O(1)\\
%                                                      \mathcal O(\epsilon^{-1}) & \mathcal O(\epsilon^{-1} )                                                                                                                                                                                                                                                                                                                                                            \end{pmatrix}.\eqlab{ill}
%                                                      \end{align}
                                                      Note that Fenichel's theory implies that $y_{u0}$ and $y_{sT}$ are each $\mathcal O(\epsilon)$ since the stable and unstable fibers are $\mathcal O(\epsilon)$ close to the unperturbed ones. The time $T$ can be included as separate variable upon inclusion of a further boundary condition. 
                                                      
                                                       As the \textsc{SMST} method is formulated in \cite{guc3}, it cannot be used to approach trajectories on the slow manifold directly. {Trajectories will always include transitions at the ends. }In \cite[Section 4.2]{langfield2014} a related collocation based method is used to compute trajectories on an $1D$ attracting slow manifold using a continuation mechanism to \textit{push out} the fast part at the ends. It may be possible to extend this approach to saddle-type slow manifolds. 
                                                       
                                                       For the computation of a full orbit the \textsc{SMST} algorithm will in general have to be combined with a separate part that computes the remaining trajectory segments (e.g. via direct integration of \eqref{ss}). 
                                                      
                                                      Assume that \textsc{SMST} method converges to a solution $\sigma=\sigma(t)$ and that $z=z(t)$ is a true solution of $z'=V(z)$ that satisfies the $n(N+1)$ conditions. Then by Taylor's formula
\begin{align*}
 \Vert \sigma -z \Vert &\le \frac{1}{24}\max_{\tau \in [0,T]} \Vert z^{(4)}(\tau) \Vert \max_i \vert \tau_{i+1}-\tau_i\vert^4 \\
 &\le \mathcal O(\epsilon^{-4} \max_i \vert \tau_{i+1}-\tau_i\vert^4).
\end{align*}
The factor $\epsilon^{-4}$ appears from estimating $\Vert z^{(4)}\Vert$. This is too pessimistic on the slow manifold since there $z^{(4)}=\mathcal O(1)$ (by definition of being slow) but it is appropriate if the connections at the ends are fast.  If a mesh and boundary conditions are fixed, then based on this estimate, one will expect the error to grow as $\epsilon$ goes to zero. For example, the reference \cite{shi}, describes the use of collocation to solve the boundary value problem
 \begin{align}
  \epsilon u''(\tau) + u'(\tau) = 1,\quad u(0)=1=u(1),\eqlab{bctoy}
 \end{align}
  and it is shown that in order to ensure convergence estimates that are uniform with respect to $\epsilon$ for this problem, a fixed mesh must be replaced by an adaptive Shishkin mesh \cite{shi}. A Shiskin mesh is basically a piecewise uniform mesh that places more points at ends where the fast transitions occur. %In conclusion, the \textsc{SMST} method has some (potential) issues with accurately describing the limit $\epsilon\rightarrow 0$. This is somewhat unsatisfactory, because the limit is well-defined and simple according to Fenichel's theory. 
  {It is the main aim of this article to establish an alternative to collocation for the computation of saddle type slow manifolds, that accurately resolves both the slow motion along the slow manifold and the fast transitions, by splitting the computation into two sub-problems. The splitting will be obtained by the application of two iterative reduction methods: \textsc{SO} and \textsc{SOF}.}

 \textbf{Reduction methods}. The \textsc{SO} method (the method of \textit{Straightening Out}, also referred to as the iterative method of Fraser and Roussel \cite{kap3}) is an example of a reduction method that enables the computation of slow manifolds without direct reference to a small parameter as e.g. it is required when using asymptotic expansions. There are several alternative methods: The intrinsic low-dimensional manifold (ILDM) method of Maas and Pope \cite{maa1}, the zero-derivative principle (ZDP) \cite{kap1,kap2}, and the computational singular perturbation (CSP) method initially due to Lam and Goussis \cite{lam1,lam2}, and later thoroughly analyzed by Zagaris and co-
workers \cite{kap4}. The \textsc{SO} method has the following interesting and numerically advantageous features:%is unique in that it is the only method, as far as I am aware of, that possesses all of the following properties. 
\begin{itemize}
 \item[(i)] It leads to exponential accurate slow manifolds.
 \item[(ii)] It can written in a form (see \eqref{fra2} below) that only involves the vector-field and its Jacobian matrix.
 \item[(iii)] It does not require smoothness of $X$ and $Y$ in $\epsilon$.
 \item[(iv)] The slow manifold approximation includes nearby equilibria.
\end{itemize}
For the purpose of this work (ii) is an important property. It means that the approach is easy to implement. In comparison with the other methods, where the number of partial derivatives required depends on the desired accuracy, the \textsc{SO} method only requires the vector-field and its first partial derivatives. Property (iii) might seem rather academic, but it highlights the methods potential in $\epsilon$-free systems (see \cite[Section 8]{kri4} and \cite{bro1,bro2,brokri1,eqn1}): The proof of statement (i) is not based on comparisons with asymptotic expansions in $\epsilon$. In the forthcoming paper \cite{brokri1} the authors apply the \textsc{SO} method in $\epsilon$-free systems. 

The \textsc{SOF} method \cite{kri4} (\textit{Straightening Out Fibers}) is also an iterative method, built as an extension to the \textsc{SO} method, that enables approximation of fibers in slow-fast systems. %To the authors knowledge, only the CSP method \cite{kap4} is capable of establishing similar approximations. 
However, only the \textsc{SOF} method enjoys all the properties listed above. In \cite[Section 8]{kri4} it was furthermore demonstrated that the \textsc{SOF} method performed far better on a problem where the slow and fast variables had not been properly identified.

% There has been quite a lot of development of reduction methods. 
%However, to the author's opinion, very little attention has actually been devoted to the applications of these methods. An exception is \cite{kap1} where the application of the ZDP method to equation-free problems is investigated. However, \cite{van2011} compared this approach with a method based on solving a fix-point equation using direct integration and the Newton method, an approach also used in \cite{marschler2014}, and concluded that the latter was vastly superior in terms of convergence and performance. 
% I believe that iterative methods may have a potential in applications with: 
% \begin{itemize}
% \item (a) normally elliptic slow manifolds, where neither Fenichel's theory nor the collocation methods apply and direct integration by implicit methods can be a challenge too;
% \item (b) saddle-type slow manifolds where direct integration does not provide a viable alternative to collocation methods. 
% \end{itemize}
This paper aims to demonstrate that the iterative methods, \textsc{SO} and \text{SOF}, can be used to compute saddle-type slow manifolds where direct integration does not provide a viable simple alternative to collocation methods. 

\textbf{Aims of paper}. %This paper presents an alternative approach for the computation of trajectories near saddle type slow manifolds. 
The idea behind the presented method is simply to split the computations into two non-stiff sub-problems: A computation on the slow manifold and a computation for the connection to and from the slow manifold. This approach is well-known. In fact it is at the very foundation of the theory of singular perturbation theory, geometric \cite{fen1,fen2,jon1} or non-geometric \cite[Chapter 10]{ascher1987}, \cite{tikhonov1952systems}, and its aim to connect $\epsilon\ne 0$ with $\epsilon=0$ of \eqsref{fs}{ss}. The novelty here, however, is to obtain the splitting using the \textsc{SO} method and the \textsc{SOF} method. In particular, the \textsc{SO} method will be used in a quadrature scheme for the propagation on the slow manifold. This procedure also applies to attracting or repelling slow manifolds (where direct integration of the full system probably offers a better approach) and even normally elliptic ones. Although, normally elliptic slow manifolds have not received as much attention as their hyperbolic counterparts, they do appear in a wide range problems in science \cite{alfi,amick1989,lor1,kri2}. %In meteorology and short-term weather forecasting \cite{lor1,lor2,lor3,van1}, for example, these objects are particularly important for e.g. the understanding of spontaneous wave generation in geophysical fluid dynamics. In particular, the slow manifold allows for an approximation of the full dynamics while filtering out fast behaviour. This (formal) reduction can potentially bridge the gap between tractable and in-tractable calculations as fast oscillations may require unfeasible many function evaluation \cite{las1,las2}. %In the example in \secref{nume} it will also be demonstrated how to use the \textsc{SOF} method to perform an accurate projection to the normally elliptic slow manifold. %molecular physics and the Born-Oppenheimer approximation \cite{McQ1}, the evolution and stability of the solar system \cite{las1,las2}, modeling of water waves in the presence of surface tension \cite{alfimov1998}, and the modeling of tethered satellites \cite{kri1,kri2}. 

In \cite{nip1} an alternative numerical scheme is suggested for the propagation on the slow manifold. {This is based on asymptotic expansions which require several partial derivatives of the vector-field $U$ with respect to the slow and fast variables but also with respect to the small parameter. The \textsc{SO} method only requires $U$ and the Jacobian $\partial_z U$ (see (ii) above). }

The \textsc{SOF} method enables, through an accurate projection onto the slow manifold, the computation of connections to and from a trajectory on the slow manifold. This computation will involve collocation but it is performed on the fast space only, using $\mathcal O(1)$-many time intervals of the fast time $t$, and will therefore not involve any $\epsilon^{-1}$-factors (as opposed to collocation on the full space). % ill-conditioned Jacobians of the form \eqref{ill}. % once the motion on it has been established by applying the \textsc{SO} method and the numerical quadrature. 
%   collocation algorithm, that is inspired by the \textsc{SMST} method, on the fast space using only $\mathcal O(1)$ equations. 
  The full method, which will be named \text{SO-SMST} method (\textit{Straightening Out} for \textit{Slow Manifolds}
 of \textit{Saddle Types}), will be described in full details in \secref{nonstiff}. %The method is applied to five examples in \secref{examples}. 
It is among the main aims to demonstrate the use of the \textsc{SO-SMST} method and describe its performance. {This will include an analysis of discretized \textsc{SO} and \textsc{SOF} methods used in the implementation.} %with the result of applying the \textsc{SMST} method. 
% focusses on (b) through the development of the \textsc{SO-SMST} method.% but th current paper also partially aims to demonstrate the use of the \textsc{SO} and \textsc{SOF} methods in both of these contexts.
The \textsc{SO-SMST} method will be applied to sevaral examples and comparisons will be made with the \textsc{SMST} method. A thorough comparison with the \textsc{SMST} method is, however, not among the aims of the paper. This must be a topic for future research. Nevertheless, some potential advantages of the iterative method will be highlighted. For one thing, it will be stressed that the method presented here, does not have any issues with $\epsilon\rightarrow 0$. {This is for example documented by the inclusion of a linear test problem \eqref{bctoy} in \secref{SOSMSTExact} where the \textsc{SO-SMST} method captures the limit $\epsilon\rightarrow 0$ accurately. %To successfully apply the classical \textsc{SMST} algorithm to this problem, the reference \cite{shi} shows that one has to make the mesh $\epsilon$-dependent. 
Hence, for certain specially structured
systems, there could be potentially interesting applications for the \textsc{SO-SMST}
method.} On the other hand, it should be pointed out that a certain disadvantage with the iterative approach taking here, is that for \textit{larger} values of $\epsilon$ \textsc{SO} and \textsc{SOF} may take longer time before reaching a specified tolerance. Worse yet, this tolerance may not reach at all since the iterative methods may require an $\epsilon$ that is smaller than what is required by Fenichel's theory. In these cases, it is very likely that the \text{SMST} method will perform far better.

\textbf{Outline of paper}. In \secref{iterativeMethods} the two different iterative methods are presented. This includes a modification of the \textsc{SO} method which is due to Neishtadt \cite{nei87}. \Secref{nonstiff} presents the \textsc{SO}-\textsc{SMST} method (\textit{Straightening Out} for \textit{Slow Manifolds}
 of \textit{Saddle Types}) 
for the computation of canard segments and their transients. \appref{Error} includes some error estimates. In \secref{numsom} some results on the numerical implementation of the iterative methods via finite differences is presented. This section furthermore covers the use of the \textsc{SO} method in a Runge-Kutta scheme. Finally, in \secref{examples} the \textsc{SO}-\textsc{SMST} method is applied to five different examples, including a nonlinear model of reciprocal inhibition with two slow and two fast variables. The results will be compared with trajectories computed using the \textsc{SMST} algorithm. As a further proof of concept homoclinic connections for the FitzHugh-Nagumo model are computed.

\textbf{Main results}. The main theoretical results of the paper are collected in the following:
\begin{itemize}
%  \item In \corref{etah2} a discrete version of the \textsc{SO} method for the point-wise approximation of a slow manifold is presented. The differential operator is replaced by a finite difference operator and \corref{etah2} shows that the error is of order $\mathtrhcal O(e^{-c/\epsilon}+\epsilon^2 h^p)$, $h$ a step-size and $p$ the order of the finite difference operation.  
\item \Secref{nonstiff} contains the most important result by demonstrating the \textsc{SO-SMST} method and how the iterative methods can be applied for the approximation of canard segments for saddle-type slow manifolds. \propref{est} and \propref{est2} describe the errors associated with this approximation.
 \item In \corref{etah2} a discrete version of the \textsc{SO} method is presented. It is this discretized version which will be used in \textsc{SO-SMST}. It is shown that the discretized method approximates the slow manifold of \eqref{fs} up to an error of order $\mathcal O(e^{-c/\epsilon}+\epsilon^2 h^p)$. Here $h$ describes the grid size in an order $p$ finite difference operator.
 \item In \propref{sofmod} a discrete version of the \textsc{SOF} method is presented. This discretized version is used in \textsc{SO-SMST}. It is shown that this discretized method approximates the tangent spaces of the fibers of \eqref{fs} up to an error of order $\mathcal O(e^{-c/\epsilon}+\epsilon^2 h^p)$.
%  \item  
%  \item In \secref{num} and \secref{fhn} I apply the iterative methods to two important examples. Most importantly, 
\end{itemize}

\textbf{Notation}. All norms will be denoted by $\Vert \cdot \Vert$ including operator norms. This should not cause unnecessary confusion. If $\mathcal U\subset \mathbb R^{n}$, $n\in \mathbb N$, then $\mathcal U+i\chi$ will denote its complex $\chi$-neighborhood:
\begin{align*}
 \mathcal U+i\chi = \{x\in \mathbb C^{n} \vert \sup_{y\in \mathcal U} \Vert x-y\Vert <\chi\}.
\end{align*}
Consider $f:\mathcal U+i\chi\rightarrow \mathbb C^m$, $m\in \mathbb N$, being analytic and bounded. Then Cauchy-estimates apply to $f$ in the following sense
\begin{align*}
\sup_{x_0\in \mathcal U+(\chi-\xi)}\Vert \partial_x f(x_0)\Vert \le\frac{\sup_{x\in \mathcal U+\chi}\Vert f(x)\Vert}{\xi},
\end{align*}
which will be written as
\begin{align}
\Vert \partial_x f\Vert_{\chi-\xi} \le\frac{\Vert f(x)\Vert_{\chi}}{\xi}.\eqlab{cauchy}
\end{align} 
% In the presentation of the iterative methods, it is useful to reserveNote the use of superscript to indicate summation over the subscripts. 
Superscripts with $n\in \mathbb N_0$ will be used to denote partial sums such as:
\begin{eqnarray}
 \eta^n = \sum_{i=0}^n \eta_i,\quad n\ge 0,\eqlab{etaNN}
\end{eqnarray}
with each of the terms in the sum being enumerated through subscripts. Following this convention means that $\eta^0=\eta_0$. 
%The method uses the \textsc{SO} and \textsc{SOF} methods 

% ill present some numerical results of a model of reciprocal inhibition. We will compare our solutions with solutions obtained using the \textsc{SMST} algorithm.
%The remaining components $y_{u0}=\pi_f(y_0-\eta_0(x_0))$, will be part of the solution of the boundary value problem. 

% . The complement to this space $E_u$ 
% 
% The authors present a simple example that shows that the time step $\delta \tau$ must satisfy $\delta \tau \ll \epsilon$ to archive a reasonable accuracy. When computing canards for $\Delta t =\mathcal O(\epsilon^{-1})$ this then introduces $\mathcal O(\epsilon^{-1})$ mesh-points. Computationally this becomes infeasible when having many variables $n=n_s+n_f$ and $\epsilon$ very small. For example, the inversion of the Jacobian via LU-factorisation in the Newton step of the \textsc{SMST} algorithm then introduces a complexity of $\mathcal O(n^{3/2} \epsilon^{-3/2})$. The exponent is $3/2$ rather than $3$ since the Jacobian is sparse. 

%% file: theIterativeMethods.tex
\section{The iterative methods}\seclab{iterativeMethods}
In this section, the different iterative methods used in this paper are presented. \Secref{som} presents the \textsc{SO} method. This section also includes the modification due to Neishstadt \cite{nei87}. \Secref{sofm} presents the \textsc{SOF} method. 
\subsection{The \textsc{SO} method: Approximation of the slow manifold}\seclab{som}
% \begin{itemize}
%  \item FENICHEL'S THEORY: $A_0$ invertible. 
%  \item WE NEED SOMETHING HERE ON the canards in \cite{desroches2010}.
% \end{itemize}
% If the condition in \eqref{M0h} holds true for a compact set of $x$-values, then the critical manifold $M_0=\{y=\eta_0(x)\}$ for $\epsilon=0$, perturbs to a potentially non-unique invariant manifold, which is also a graph over the same $x$-values: $M_h = \{y=\eta(x)\}$ where $\eta$ is $\mathcal O(\epsilon)$-close to $\eta_0$.  
The \textsc{SO} method is an iterative approach to approximating an invariant slow manifold $M=\{y=\eta(x)\}$. The point of departure is the invariance equation:
\begin{align}
 0=-\epsilon \partial_x \eta(x) X(x,\eta(x))+Y(x,\eta(x)),\eqlab{inveqn}
\end{align}
which is obtained from \eqref{fs} by enforcing the invariance of the graph $y=\eta(x)$.
Basically, the \textsc{SO} method aims to solve this equation iteratively by considering the following equations:
\begin{align}
 -\epsilon \partial_x \eta^{n-1}(x) X(x,\eta^n(x))+Y(x,\eta^n(x))=0,\eqlab{fra2}
\end{align}
for $n\ge 1$ starting from $\eta^{0}(x)=\eta_0(x)$ \eqref{eta0} for $n=1$. Here superscripts are used because the functions $\eta^n$ will be obtained as partial sums. Recall \eqref{etaNN}. 
In the form \eqref{inveqn} the \textsc{SO} method is also known as the iterative method of Fraser and Roussel \cite{kap3}. Each step of the method involves the solution of a non-linear equation. There are some simple alterations to the method which makes the method computationally simpler. To present these, it is, however, advantageous to take a different view-point which will highlight the following:
\begin{itemize}
 \item[$1^\circ$] The method can be initiated from any initial guess;
  \item[$2^\circ$] The method leads to exponentially accurate approximations;
  \item[$3^\circ$] The method applies to any $M_0$ satisfying \eqref{fastcond};
  \item[$4^\circ$] The method can be altered so that if $\eta_0$ is known then the method only involves the solution of linear equations.
  \end{itemize}
The last point $4^\circ$ is perhaps not surprising because asymptotic expansions possess this property. However, most reduction methods are posed as fully non-linear algebraic equations. %I will aim to show each of these points below. I will also try to allude to the fact that 
The \textsc{SO} method does not require the slow-fast system to be written in the canonical form \eqref{fs}. In \secref{Lindemann} we consider the Lindemann mechanism 
\begin{align}
 \dot x &=-x(x-y),\eqlab{lm0}\\
 \dot y &=x(x-y)-\epsilon y,\nonumber
\end{align}
which is a slow-fast system where the slow and fast variables have not been properly identified. The iterative methods \textsc{SO} and \textsc{SOF} still apply to such systems \cite[Section 8]{kri4} but it is unclear how to apply these methods if one is presenting these using $\epsilon X$ for $\dot x$. Clearly $\dot x$ is not small throughout for \eqref{lm0}.
Therefore in this section we replace $\epsilon X$ by
\begin{align*}
X^\epsilon\equiv \epsilon X
\end{align*}
and consider the equations
\begin{align*}
 \dot x &=X^\epsilon(x,y),\\
 \dot y &=Y(x,y),
\end{align*}
instead of \eqref{fs}. Having said that, the main focus below will still be on the case where $X^\epsilon$ is $\epsilon X$ and small throughout. All the proofs of the statements are based on the canonical slow-fast form \eqref{fs}. The reason why the iterative methods still apply when $X^\epsilon$ is not small throughout is that one in fact only needs 
\begin{align}
X^\epsilon(x,\eta_0(x))=\mathcal O(\epsilon),\eqlab{XepsSmall}
\end{align}
to be small with respect to $\epsilon$. 

Now suppose that \eqref{M0h} holds true so that $M_0$ is normally hyperbolic. Suppose furthermore that $y=\zeta_0(x)$ is an approximation to the slow manifold in the sense that it satisfies \eqref{inveqn} up to a small error $\delta_0  = \sup_x \Vert \rho_0(x)\Vert$:
\begin{align}
  \rho_0(x)=-\partial_x \zeta_0(x) X^\epsilon(x,\zeta_0(x))+Y(x,\zeta_0(x)).\eqlab{rho01}
\end{align}
The function $\rho_0=\rho_0(x)$ is the \textit{obstacle} to invariance of the slow manifold: If $\rho_0\equiv 0$ then $\zeta_0=\zeta_0(x)$ satisfies the invariance equation \eqref{M0h} and defines an invariant slow manifold. 
 The approximation $\zeta_0$ could be $\eta_0$ from \eqref{eta0}. Then introduce $y_0$ by 
 \begin{align}
 y=\zeta_0(x)+y_0.\eqlab{so0}
 \end{align}
 The transformation \eqref{so0} \textit{straightens out} the approximation of the slow manifold $y=\zeta_0(x)$ to $y_0=0$. The new equations for $y_0$ are:
 \begin{align}
 \dot{y}_0 =Y_0(x,y_0)&\equiv-\partial_x \zeta_0 X^\epsilon (x,\zeta_0(x)+y_0) + Y(x,\zeta_0(x)+y_0)\nonumber\\
 &=\rho_0(x) + A_0(x)y_0+R_0(x,y_0),\eqlab{y0}
\end{align}
with $\rho_0$ as in \eqref{rho01},
\begin{align}
 A_0(x) &= -\partial_x \zeta_0(x) \partial_y X^\epsilon (x,\zeta_0(x)) + \partial_y Y(x,\zeta_0(x)), \eqlab{A0}
 \end{align}
and $R_0=\mathcal O(y_0^2)$. The equality in \eqref{y0} is due to the Taylor expansion of $Y_0$ about $y_0=0$.  
% \begin{lemma}
 The condition \eqref{M0h} implies that $\Vert (\partial_y Y\vert_{M_0})^{-1}\Vert\ll \epsilon^{-1}$. The matrix-valued function $A_0=A_0(x)$ in \eqref{A0} is therefore invertible for $\epsilon$ sufficiently small. Also since $\delta_0$ is assumed to be small, the contraction mapping theorem implies that there exists a solution $\eta_1=\eta_1(x)\approx -A_0(x)^{-1}\rho_0(x)$ of $Y_0(x,\eta_1) = 0$:
\begin{align}
 0=\rho_0(x)+A_0(x)\eta_1+R_0(x,\eta_1),\eqlab{fra0}
\end{align}
satisfying
\begin{align}
 \sup_x \Vert \eta_1(x)\Vert =\mathcal O(\delta_0).\eqlab{deltap}
\end{align}
The solution $\eta_1$ is analytic if $X^\epsilon,\,Y\in C^\omega$. 
Note also that \eqref{fra0} cf. \eqref{y0} can be written as
\begin{align}
 0=-\partial_x \zeta_0(x) X^\epsilon(x,\zeta_0(x)+\eta_1(x))+Y_0(x,\zeta_0(x)+\eta_1(x)).\eqlab{fra1}
\end{align}
Now, straighten out this new approximation $y_0=\eta_1(x)$ of the slow manifold to $y_1=0$ by setting $y_0=\eta_1(x)+y_1$ so that
 \begin{align*}
  \dot y_1 &=Y_1(x,y_1) = \rho_1(x) + A_1(x)y_1+R_1(x,y_1),
 \end{align*}
with
\begin{align}
 \rho_1(x) &=-\partial_x\eta_1(x) X^\epsilon (x,\zeta_0(x)+\eta_1(x)).\eqlab{rho1}
\end{align}
If the vector-fields $X^\epsilon$ and $Y$ are analytic then one can apply Cauchy estimates \eqref{cauchy} to estimate $\sup_x \Vert \partial_x \eta_1(x)\Vert$ in terms of $\sup_x \Vert \eta_1(x)\Vert$ on a smaller domain so that
\begin{align*}
 \delta_1\equiv\sup_x \Vert \rho_1(x)\Vert=\mathcal O(\epsilon \delta_0).
\end{align*}
Hence the new error is of the order of $\epsilon$ times the previous error. If one starts with $\zeta_0=\eta_0$ then the error $\rho_0$ is $\mathcal O(\epsilon)$ and applying the procedure successively therefore directly leads to formal error estimates of the form $\mathcal O(\epsilon^{n+1})$, even when the vector-field $U$ is only $C^r$, $r\ge n+1$ \cite{kap3}. In terms of the original variables the approximation takes the form $$y=\eta^{n}(x) = \zeta_0(x)+\eta_1(x)+\cdots +\eta_{n}(x).$$ 
The form in \eqref{fra1} immediately implies that the procedure can be written compactly as \eqref{fra2}
% \begin{align}
%  -\epsilon \partial_x \eta^{n-1}(x) X(x,\eta^n(x))+Y(x,\eta^n(x))=0.\eqlab{fra2}
% \end{align}
for the approximation $y=\eta^n(x)$ of $y=\eta(x)$ satisfying \eqref{inveqn}. %When starting from $\eta^{-1}\equiv 0$, then $\eta^0\vert_{\epsilon=0}=\eta_0\vert_{\epsilon=0}$ solving $Y(x,\eta^0)\vert_{\epsilon=0}=0$ is just the critical manifold. This is the usually representation of the \textsc{SO} method. In this form it is also known as the iterative method of Fraser and Roussel \cite{kap3}. 
The alternative presentation of the \textsc{SO} method above, which is due to MacKay \cite{mac1}, has the advantage that it shows that one does not need to start the procedure from $\eta_0$. One could also just start from a guess $y=\zeta_0(x)$. 
% and solve
% \begin{align}
%  -\epsilon \partial_x (\zeta_0(x)+\eta^{n-1}(x)) X(x,\zeta_0(x)+\eta^n(x))+Y(x,\zeta_0 +\eta^n(x))=0,\eqlab{fra3}
% \end{align}
% for $\eta^n$ starting from $\eta^{-1}\equiv 0$. 
The new error will still be $\epsilon$ times a $C^1$ estimate of the previous error cf. \eqref{deltap}. This explains $1^\circ$. 

The $\mathcal O(\epsilon^{k+1})$-estimate is not uniform in $k$: In the analytic case the domain of definition will eventually vanish when iteratively applying the Cauchy estimates. Using Neishstadt-type estimates it was, however, shown in \cite{kri3} that the error can be made exponentially small. This explains $2^\circ$. From this presentation, it is also clear that condition \eqref{M0h} is not needed. The importance is just that $A_0$ can be inverted and for this \eqref{fastcond} suffices. The results of \cite{kri3} does therefore not only apply to normally hyperbolic $M_0$'s. It also holds for normally elliptic slow manifolds, which confirmed a conjecture by MacKay \cite{mac1}. This shows $3^\circ$. A remarkable property of the \textsc{SO} method is in fact that it does not require $A_0$ to be bounded. Only $A_0^{-1}$ is measured, making the method potentially useful in the analysis of slow ODE - fast PDE systems like the one in \cite{kri1}.   %The only requirement is that $\Vert (\partial_y Y)^{-1}\Vert\le \frac{K}{2}\ll 

%epsilon^{-1}$ and that the vector-fields are analytic.

Computationally the \textsc{SO} method involves solving a nonlinear equation at each step $n$. In practice, the method therefore involves two loops: An outer loop updating $n$ and an inner loop using e.g. a Newton method for the solution $\eta^n$ of the nonlinear equation \eqref{fra2}. A result of Neishstadt in \cite[Lemma 1]{nei87} shows, however, that this inner loop is actually not necessary. Furthermore, the matrix-valued functions $A_i=A_i(x)$, that appear by the procedure outlined above, does not need to be updated. 
 \begin{proposition}\proplab{modSOprop}
  (Modified \textsc{SO} method, Lemma 1 in \cite{nei87}) Consider the slow-fast system \eqref{fs} with $X^\epsilon$ and $Y$ analytic on some complex $(\chi,\nu)$-neighborhood $(x,y)\in (\mathcal U+i\chi)\times (\mathcal V+i\nu)$ of $\mathcal U\times \mathcal V$ for some compact sets $\mathcal U$ and $\mathcal V$ in $\mathbb R^{n_s}$ and $\mathbb R^{n_f}$, respectively. Assume furthermore that $y=\zeta_0(x)$ is an approximation to the slow manifold so that $A_0=A_0(x)$ \eqref{A0} is invertible on $\mathcal U$ and that the error $\delta_0= \Vert \rho_0\Vert_\chi$ in \eqref{rho01} is sufficiently small. Then for $\epsilon$ and $\delta_0$ sufficiently small the transformation $$y=\eta(x)+\tilde y,$$ $\eta\equiv \eta^{N(\epsilon)}=\zeta_0+ \sum_{n=1}^{N(\epsilon)} \eta_n$ with $N(\epsilon) = \mathcal O(\epsilon^{-1})$ and $\eta_n=\eta_n(x)$ satisfying
  \begin{align}
  \eta_n (x)&= -A_{0}(x)^{-1}\rho_{n-1}(x),\eqlab{etan}\\
  \rho_{n-1}(x)&=-\partial_x \eta_{n-1}(x) X^\epsilon(x,\eta^{n-1}(x))+Y(x,\eta^{n-1}(x)),\nonumber
%   A_{0} (x)&=-\epsilon \partial_x \zeta_{0}(x) \partial_y X(x,\zeta_{0}(x))+\partial_y Y(x,\zeta_{0}(x)).\nonumber
  \end{align}
  with $A_0$ as in \eqref{A0}, 
%  \end{proposition}
 will on $\mathcal U\times \mathcal V$ transform \eqref{fs} into
 \begin{align*}
  \dot{\tilde y} = \tilde \rho(x) + \tilde A(x) \tilde y +\tilde R(x,\tilde y)
 \end{align*}
with $\tilde R=\mathcal O(\tilde y^2)$ and
\begin{align*}
 \Vert \tilde \rho\Vert_0 = \mathcal O(e^{-c/\epsilon}),
\end{align*}
with $c>0$ independent of $\epsilon$ and $\delta$. 
\end{proposition}
%  \note{Can be improved further - take $A_n=A_0$ constant and equal to $\partial_y Y_0(x,0)$. If $h=\mathcal O(\epsilon)$ then $A_0$ also constant on $h$-grid; its inverse can be stored.}
\begin{proof}
The proof is only sketched. For all the details see \cite{nei87}.  
At the $n$th-step the equations take the following form
\begin{align*}
 \dot x &=X^\epsilon_{n-1}(x,y_{n_1}),\\
 \dot y_{n-1} &=Y_{n-1}(x,y_{n-1})\rho_{n-1}(x)+(A_0(x)+a_{n-1}(x))y_{n-1}+R_{n-1}(x,y_{n-1}),
\end{align*}
with $\delta_{n-1} = \Vert \rho_{n-1} \Vert_{\nu_{n-1}}$ and $\Vert a_{n-1}\Vert_{\nu_{n-1}}=\mathcal O(\epsilon\delta_0)$.  
The variables $y_n$ is then introduced in accordance with \eqref{etan}:
\begin{align*}
 y_{n-1}  =\eta_n(x) + y_{n},
\end{align*}
giving 
\begin{align*}
\dot x &=X^\epsilon_n(x,y_n),\\
 \dot y_{n} &=\rho_{n}(x)+(A_0(x)+a_{n}(x))y_n+R_{n}(x,y_{n}).
\end{align*}
with
\begin{align}
 \rho_n &= - \partial_x \eta_n(x) X_{n-1}^\epsilon(x,\eta_n)+a_{n-1}\eta_n + R_{n-1}(x,\eta_n),\eqlab{rhon}\\
 a_{n} &=-\partial_x \eta_n \partial_y X^\epsilon_{n-1}(x,\eta_n)+a_{n-1}(x)+\partial_y R(x,\eta_n),\eqlab{ean}
\end{align}
and $X_{n}^\epsilon(x,y_n)=  X_{n-1}^\epsilon(x,\eta_n+y_n)$. This gives
\begin{align*}
 \Vert \rho_n \Vert_{\nu_{n}} &\le c_n\left( \epsilon \xi_{n}^{-1} \delta_{n-1} +\epsilon\delta_0\delta_{n-1}+\delta_{n-1}^2\right), \nonumber\\
  \nu_n &=\nu_{n-1}-\xi_{n},
\end{align*}
for some $c_n>0$, upon applying a Cauchy estimate. The last two terms are subordinate to the first term and one can therefore take $\xi_n = 4c_n\epsilon$ so that for $\delta_{n-1}$ sufficiently small
\begin{align*}
 \Vert \rho_n \Vert_{\nu_n} \le \frac{1}{2}\delta_{n-1}.
\end{align*}
 Also $\Vert a_n\Vert_{\nu_n}=\mathcal O(\epsilon+\delta_0)$. One can then easily bound $c_n\le 2c_0$ for $\epsilon$ sufficiently small and therefore uniformly bound the $\xi_n$'s and take $\mathcal O(\epsilon^{-1})$ steps before the domain vanishes. This gives the exponential estimate. 

% I refer to \cite{nei87} for the proof. 
% for the original proof. Consider again \eqref{y0} but instead of introducing $\zeta_1$ as the solution of $Y_0(x,\zeta_1)=0$ let 
% \begin{align}
% \zeta_1=-A_0^{-1}\rho_0.\eqlab{zeta1}
% \end{align}
% That is the solution of the linear equation obtained from \eqref{fra0} with $R_0$ neglected. 
% Then from $y_0=\zeta_1+y_1$ I obtain
% \begin{align*}
%  \dot y_1& = \rho_1(x)+A_1(x)y_1 + R_1(x,y_1),
% \end{align*}
% where
% \begin{align}
%  \rho_1(w) &= -\epsilon \partial_x \zeta_1 X(x,\zeta_0+\zeta_1) + Y_0(x,\zeta_1) \nonumber\\
%  &=-\epsilon \partial_x \zeta_1 X(x,\zeta_0+\zeta_1) +R_0(x,\zeta_1).\eqlab{rho12}
% \end{align}
% The new error is therefore
% \begin{align*}
%  \delta_1 = \Vert \rho_1 \Vert &= \mathcal O(\epsilon\delta_0+\delta_0^2),
% \end{align*}
% by applying Cauchy estimates. Suppose that $\zeta_0$ comes from $Y_0(x,\zeta_0(x))=0$ then $\delta_0=\mathcal O(\epsilon)$ and so $\delta_1=\mathcal O(\epsilon^2)$. I can then use this principle successively to conclude
% \begin{align*}
%  \delta_{n+1} =\mathcal O(\epsilon \delta_n+\delta_n^2).
% \end{align*}
% The quadratic term $\delta_n^2=\mathcal O(\epsilon^{2(n+1)})$ will be subordinate to the first term $\epsilon \delta_n=\mathcal O(\epsilon^{n+2})$ for $n\ge 1$. One can then re-use the estimates from \cite{kri3} to conclude that after $\mathcal O(\epsilon^{-1})$ applications the error will be exponentially small. %I plan to report on the details of this in a separate paper.
\end{proof}

This shows $4^\circ$. 

\begin{remark}
If $X^\epsilon=\epsilon X$ is small then one could just replace $A_0$ in \propref{modSOprop} by $\partial_y Y(x,\eta_0(x))$ since their difference is $\mathcal O(\epsilon)$. In the more general case, where the slow variables have not been properly identified, then $-\partial_x \eta_0 \partial_y X^\epsilon(x,\eta_0)=\mathcal O(1)$ and replacing $A_0$ by $\partial_y Y(x,\eta_0(x))$ will not work. We will focus on this in greater details in our forthcoming paper \cite{brokri1}. 
\end{remark}

\begin{remark}\remlab{modSO}
%  I could start the procedure from $\zeta_0=\eta_0$ solving $Y(x,\eta_0)=0$. 
 The error $\rho_0$ is given by
 \begin{align*}
 \rho_0(x) &=-\partial_x \zeta_0(x) X^\epsilon(x,\zeta_0(x))+Y(x,\zeta_0(x)).%\eqlab{rho0}
 \end{align*}
 If $\zeta_0=\eta_0$ then $\rho_0$ vanishes at any equilibrium of the form $(x,y)=(x^e,\eta_0(x^e))$ where $X^\epsilon (x^e,\eta_0(x^e))=0$ and $Y(x^e,\eta_0(x^e))=0$. %Therefore $\eta_1(x^e)=0$ cf. \eqref{etan} in this case and it follows, from \eqsref{rhon}{ean} with $n=1$, $a_0\equiv 0$ and the fact that the Taylor series of $R_0$ starts with quadratic terms in the second argument, that $\rho_1(x^e)=0$. T
 Proceeding by inducion on $n$ using \eqref{etan}, it easily follows that the modified \textsc{SO} method will preserve this property so that $\rho_{n-1}(x^e)=0$ and hence all of the approximations $\eta^n$ will include equilibria.
% I could start the procedure from $\zeta_0=\eta_0$ so that $\rho_0$ is given by \eqref{rho0}. 
% Notice that $\rho_0$ given by \eqref{rho0} vanishes at an equilibrium where $X(x^e,\zeta_0(x^e))=0$. Therefore $\zeta_1(x^e)=0$ cf. \eqref{eta1} and this $\rho_1(x^e)=0$  
% $\delta_0=\mathcal O(\epsilon)$ in which case $\delta_1=\mathcal O(\epsilon^2)$. ??As opposed to the traditional \textsc{SO} method this procedure does not guarantee that the slow manifold includes all true equilibria (compare \eqsref{rho1}{rho12}).?? 
% Note also that the two methods coincide when the vector-field $U$ is linear in the fast variables.
\end{remark}
\begin{remark}\remlab{noniter}
 The iterative method cannot be used to compute the canard orbits as those in \cite{desroches2010} that appear as the intersection of an attracting slow manifold with a repelling one. This is because near the intersection the condition \eqref{fastcond} is violated.  
\end{remark}

% In the following section I will consider the numerical implementation of the \textsc{SO} method.
% Now I estimate the new error
% \begin{align*}
%  \delta_1 = \Vert \rho_1\Vert &\le  \frac{\epsilon K\delta}{2\xi}+\frac{C_R}{\kappa^2}\left(\frac{K\delta}{2}\right)^2.
% \end{align*}
% Here I have used the Cauchy estimates:
% \begin{align}
%  \Vert \partial_x \zeta \Vert_{\nu-\xi}& \le \frac{\Vert \zeta \Vert_{\nu}}{\xi},\eqlab{etax_2} 
% \end{align}
% for $\nu-\xi>0$ and $\xi>0$. Assume $$\frac{C_R K\delta}{2\kappa^{2}}  \le \frac{\epsilon }{\xi},$$ so that
% \begin{align}
%  \delta_1 = \Vert \rho_1\Vert_{\chi-\xi} &\le \frac{\epsilon K\delta}{2\xi}+\frac{C_RK\delta }{2\kappa^2}\frac{K\delta}{2} \le \frac{\epsilon K\delta}{\xi}.\eqlab{delta1}
% \end{align}
% I can reuse the estimates for $A_1$ and $R_1$ in \cite{kri3}. Applying the result once with $\xi$ large gives $\delta_1 = \mathcal O(\epsilon^2)$. Then I am in a position to apply the result successively taking $\xi \ge 2K\epsilon$ so that $\delta_{i+1} \le 2^{-i}\delta_1$. I can take $N_\epsilon=\mathcal O(\epsilon^{-1})$ steps before the domain vanishes, giving $\delta_{N_\epsilon}=\mathcal O(e^{-c/\epsilon})$.

% \subsection{``Twisted'' Van der Pol}
The following section describes the \textsc{SOF} method which will be used to approximate the fibers.
 \subsection{The \textsc{SOF} method: Approximation of fibers}\seclab{sofm}
 Let $M_h$ be a slow manifold of saddle type, with a stable manifold $W^s(M_h)$ of dimension $n_s+n_f^s$ and an unstable manifold $W^u(M_h)$ of dimension $n_s+n_f^u$ ($n_f=n_f^s+n_f^u$). Then Fenichel's theory shows that there exists a local transformation $(u,v,w)\mapsto (x,y)$, with $\text{dim}\,\{v\}={n_f^s}$ and $\text{dim}\,\{w\}={n_f^u}$, mapping \eqref{fs} into the \textit{Fenichel normal form} \cite{jon1}:
\begin{align}
         \dot u &=\epsilon (U_0(u)+U_1(u,v,w)vw),\nonumber\\
\dot v&=V(u,v,w)v,\eqlab{fnf2}\\
\dot w&=W(u,v,w)w.\nonumber
\end{align}
Here $U_1(u,v,w):\{v\}\times \{w\}\rightarrow \mathbb R^{n_s}$ is a
bilinear function of $v$ and $w$. The slow manifold is then given by
$\{v=0,\,w=0\}$ with stable manifold $\{w=0\}$ and unstable manifold
$\{v=0\}$. %The transformation may only exist in a small neighborhood of the
%slow manifold so in general I need $\Vert v\Vert \le \Delta_v$ and $\Vert
%w\Vert\le \Delta_w$ for some small $\Delta_v>0$ and $\Delta_w>0$ independent of $\epsilon$.
Note in particular, that the slow vector-field is independent of the fast variables to linear order. The \textsc{SOF} method approaches this ideal. To explain this first assume that the \textsc{SO} method has been applied for an approximation of the slow manifold $y=\eta(x)$. Then introduce $y_0$ by $y=\eta(x)+y_0$ so that 
\begin{align*}
 \dot x &=\Lambda^\epsilon(x)+\mu_0(x)y_0+T(x,y_0),\\
 \dot y_0&=A(x_0)y_0+R(x_0,y_0),
\end{align*}
neglecting the exponentially small terms. Here 
\begin{align}
\Lambda^\epsilon(x) &= X^\epsilon(x,\eta(x))=\mathcal O(\epsilon), \eqlab{Lambda}\\
\mu_0(x) &= \partial_y X^\epsilon(x,\eta(x)),\eqlab{mu0}
\end{align}
and
\begin{align}
 A(x) = - \partial_x \eta \partial_y X^\epsilon(x,\eta(x))+\partial_y Y(x,\eta(x)),\eqlab{A}
\end{align}
while $R=\mathcal O(y_0^2),T=\mathcal O(\epsilon y_0^2)$. We then seek a transformation of the slow variables of the form 
\begin{align}
x=x_0+\phi^\epsilon_0 (x_0) y_0,\eqlab{x0}
\end{align}
pushing the error $\gamma_0 = \Vert \mu_0 \Vert=\mathcal O(\epsilon)$ to higher order in $\epsilon$. Here $\phi^\epsilon_0\in \mathbb R^{n_s\times n_f}$ and the superscript $\epsilon$ is as above used to highlight that $\phi_0^\epsilon$ will be $\mathcal O(\epsilon)$ if the slow-fast system is written in the canonical slow-fast form \eqref{fs}. Applying the transformation in \eqref{x0} gives
\begin{align}
 \dot x_0 &=\Lambda^\epsilon(x_0) + \left\{   \partial_x \Lambda^\epsilon  (x_0) \phi^\epsilon_0(x_0) +\mu_0(x_0)-\phi^\epsilon_0(x_0) A(x_0) +\mu_1(x_0)\right\} y_0 +\mathcal O(\epsilon y_0^2).\eqlab{curly}
\end{align}
where $\mu_1$ is
\begin{align}
 \mu_1(x_0) = -\partial_x \phi^\epsilon_0(x_0) \Lambda^\epsilon(x_0).\eqlab{mu1}
\end{align}
Here $\partial_x \phi^\epsilon_0$ times $\Lambda$ is understood column-wise. In the \textsc{SOF} method one is looking for a solution $\phi^\epsilon_0=\phi^\epsilon$ that makes the curly brackets \eqref{curly} vanish:
\begin{align}
 \partial_x \Lambda^\epsilon  \phi^\epsilon +\mu_0-\phi^\epsilon A +\mu_1 = 0.\eqlab{phieqn}
\end{align}
As with the \textsc{SO} method this is then approached iteratively, letting first $\phi^\epsilon_0(x)\approx \mu_0(x) A(x)^{-1}=\mathcal O(\gamma_0)$ solve the linear equation 
\begin{align}
 \partial_x \Lambda^\epsilon   \phi^\epsilon_0+\mu_0-\phi^\epsilon_0 A = 0.\eqlab{phi0}
\end{align}
Then the new error is 
\begin{align*}
 \gamma_1\equiv \Vert \mu_1\Vert=\mathcal O(\epsilon \gamma_0).
\end{align*}
 using that $\Lambda^\epsilon=\mathcal O(\epsilon)$ cf. \eqref{Lambda} in \eqref{mu1}. This error is smaller than the previous one $\gamma_0=\mathcal O(\epsilon)$. 
Iterating this procedure one obtains the full \textsc{SOF} method. 
\begin{proposition}
(The \textsc{SOF} method \cite{kri4}) Provided $\epsilon$ is sufficiently small, then the function ${\phi^\epsilon}=\sum_{n=1}^{N(\epsilon)}\phi^\epsilon_n$, $N(\epsilon)=\mathcal O(\epsilon^{-1})$, where 
% the partial sums $\phi^\epsilon^n=\sum_{i=0}^n\phi^\epsilon_i(x_0)$ satisfy the linear equations
% \begin{align}
%   \epsilon (\partial_x X_0+\partial_y X_0 \partial_{x} {\eta})\phi^\epsilon^{n} -\epsilon \partial_x \phi^\epsilon^{n-1} X_0+\partial_y X_0 - \phi^\epsilon^{n} &(-\epsilon \partial_{x} {\eta} \partial_y X_0+\partial_y Y_0) = 0,\eqlab{phii}\\
% \mbox{for $0\le n\le N(\epsilon)$, using the convention that}\quad  &\phi^\epsilon^{-1}\equiv 0,\nonumber 
% \end{align}
the $\phi^\epsilon_n$'s satisfy the linear equations
\begin{align*}
 \partial_x \Lambda^\epsilon &\phi^\epsilon_n +\mu_{n}-\phi^\epsilon_n A= 0,\eqlab{SOFEqns}\\
 \mu_n &=-\partial_x \phi^\epsilon_{n-1} \Lambda^\epsilon,\nonumber\\
%  \partial_x \Lambda^\epsilon & = \epsilon \partial_x X_0(x,\eta)+\epsilon \partial_y X_0(x,\eta) \partial_{x} {\eta},\nonumber\\
%  A(x) &= -\epsilon \partial_{x} {\eta} \partial_y X_0(x,\eta)+\partial_y Y_0(x,\eta),\\
 \phi^\epsilon_{-1}&\equiv 0,\nonumber
\end{align*}
with $\Lambda^\epsilon$ and $A$ as in \eqsref{Lambda}{A}, respectively,  
solves \eqref{phieqn} up to exponentially small error
\begin{align}
 \partial_x \Lambda^\epsilon  \phi^\epsilon +\mu_0-\phi^\epsilon A y_0 -\partial_x \phi^\epsilon \Lambda^\epsilon = \mathcal O(e^{-c/\epsilon}),
\end{align}
for some constant $c>0$ independent of $\epsilon$.  
\end{proposition}

As for the modified \textsc{SO} method it also here suffices to replace $A=A_0+a$ in \eqref{phi0} by $A_0$ since $\Vert a\Vert = \mathcal O(\epsilon\delta_0)$ is small and can therefore along with $\partial_x \Lambda^\epsilon \phi^\epsilon_n$ be combined into the error at the following step $\mu_{n+1} = -\partial_x \phi^\epsilon_{n} \Lambda^\epsilon+\partial_x \Lambda^\epsilon  \phi^\epsilon_n-\phi^\epsilon_n a$. Indeed, the last two error terms are by Cauchy estimates subordinate to the first error and the exponential estimates can therefore also be obtained in this case.
\begin{proposition}
(The modified \textsc{SOF} method) Provided $\epsilon$ is sufficiently small, then the function ${\phi^\epsilon}=\phi^\epsilon_0+\sum_{n=1}^{N(\epsilon)}\phi^\epsilon_n$, $N(\epsilon)=\mathcal O(\epsilon^{-1})$, where 
% the partial sums $\phi^\epsilon^n=\sum_{i=0}^n\phi^\epsilon_i(x_0)$ satisfy the linear equations
% \begin{align}
%   \epsilon (\partial_x X_0+\partial_y X_0 \partial_{x} {\eta})\phi^\epsilon^{n} -\epsilon \partial_x \phi^\epsilon^{n-1} X_0+\partial_y X_0 - \phi^\epsilon^{n} &(-\epsilon \partial_{x} {\eta} \partial_y X_0+\partial_y Y_0) = 0,\eqlab{phii}\\
% \mbox{for $0\le n\le N(\epsilon)$, using the convention that}\quad  &\phi^\epsilon^{-1}\equiv 0,\nonumber 
% \end{align}
% $
%  \phi^\epsilon_0 = \mu_0 A_0^{-1}$
% and where the $\phi^\epsilon_n$'s satisfy the linear equations
% \begin{align*}
%  \partial_x \Lambda &\epsilon \phi^\epsilon_n +\mu_{n}-\phi^\epsilon_n A= 0,\\
%  \mu_n &=-\epsilon \partial_x \phi^\epsilon_{n-1} \Lambda,\\
%  \partial_x \Lambda(x)& = \partial_x X_0(x,\eta)+\partial_y X_0(x,\eta) \partial_{x} {\eta},\\
%  A(x) &= -\epsilon \partial_{x} {\eta} \partial_y X_0(x,\eta)+\partial_y Y_0(x,\eta)
% \end{align*}
% and
\begin{align}
\phi^\epsilon_n&=\mu_n A_0^{-1},\eqlab{SOFphin}\\
%   \mu_{n}(x)-\phi^\epsilon_n(x) A_0(x)= 0,\\
  \mu_n &=-\partial_x \phi^\epsilon_{n-1}\Lambda^\epsilon
%   &+\left(\epsilon  \partial_x X_0(x,\eta(x))+\epsilon \partial_y X_0(x,\eta(x)) \partial_x {\eta}(x)\right)\epsilon \phi^\epsilon_{n-1}(x)\nonumber\\
+\partial_x \Lambda^\epsilon \phi^\epsilon_{n-1}-\epsilon \phi^\epsilon_{n-1} a,\nonumber\\
%   \Lambda(x) &= X_0(x,\eta^h(x)),\\
%   \delta_x^h \Lambda(x) &= \partial_x X_0(x,\eta^h(x))+\partial_y X_0(x,\eta^h(x)) (\partial_{x} {\eta_0}(x)+\delta_x^h(\eta^h(x)-\eta_0(x))),\\
%  A_0(x) &= -\epsilon \partial_{x} {\eta_0}(x) \partial_y X_0(x,\eta_0(x))+\partial_y Y_0(x,\eta_0(x)),\nonumber\\
 a(x) &= -\partial_x {\eta}\partial_y X^\epsilon(x,\eta)+\partial_y Y_0(x,\eta)-A_0(x),\nonumber
%  \phi^\epsilon_{-1}&\equiv 0,
 \end{align}
 for $n\ge 1$ and $A_0$ and $\mu_0$ as in \eqsref{A0}{mu0} respectively, solves \eqref{phieqn} up to exponentially small error:
\begin{align}
 \partial_x \Lambda^\epsilon \phi^\epsilon +\mu_0-\phi^\epsilon A y_0 -\partial_x \phi^\epsilon \Lambda^\epsilon = \mathcal O(e^{-c/\epsilon}),
\end{align}
for some constant $c>0$ independent of $\epsilon$.  
\end{proposition}
% \begin{proof}
% \begin{proof}
%  I can proceed as in \propref{etah}. I just note that $\partial_x \Lambda=\partial_x X_0(x,\eta^h(x))+\partial_y X_0(x,\eta^h(x)) \partial_x {\eta}^h(x)$ has been replaced by
%  \begin{align*}
%   \partial_x X_0(x,\eta^h(x))+\partial_y X_0(x,\eta^h(x)) \delta_{x}^h {\eta}^h(x).
%  \end{align*}
% %  Similarly, I have in the expression for $a$ replaced $A$ by $-\epsilon \delta_{x}^h {\eta^h} \partial_y X_0(x,\eta^h)+\partial_y Y_0(x,\eta^h)$. 
% \end{proof}
%  I refer the reader to \cite{kri4}. 
%  I only note that 
%  \begin{align*}
%   \partial_x \Lambda &= \partial_x X_0+\partial_y X_0 \partial_{x} {\eta},\\
% %   A &= -\epsilon \partial_{x} {\eta} \partial_y X_0+\partial_y Y_0.
%  \end{align*}
% \end{proof}

% \begin{remark}
% 
% \end{remark}

Geometrically, the function $\phi^\epsilon$ gives through 
% \begin{align*}
 \begin{eqnarray}
 \textnormal{Rg}\,\left(
\begin{pmatrix}
 {\phi^\epsilon}(x)\\
I_f+\partial_x {\eta}(x){\phi^\epsilon}(x)
\end{pmatrix}+\mathcal O(e^{-c_2/\epsilon})\right),\eqlab{tsp}
\end{eqnarray}
% \end{align*}
an exponentially accurate approximation of the tangent spaces to the fibers at $y=\eta(x)$ \cite{kri4}. Here $I_f=\text{identity}\in \mathbb R^{n_f\times
n_f}$.

% MORE HERE: COULD WE DESCRIBE THE COMPLEXITY: FOR $2X2$ CASE? To evaluate vector-field we use 
% \begin{align*}
% \text{cost}A_0^{-1}+3^{n_s}\times ITER (\text{cost}A_0^{-1}\rho +\text{cost}\rho_+)\\
% &=n_f^3+3^{n_s}\times ITER(n_{f}^2 +\text{cost}V+n_{f} n_s).
% \end{align*}
% To use an implicit solver the cost is dominated by the cost it takes to perform the newton step: The complexity is $(n_f+n_s)^3+ITER (n_f+n_s)^2$.  Consider a few number of slow variables, then the order is comparable. However, the implicit solver still needs to resolve the fast time scale in oscillatory problems. The quadrature with \textsc{SO}, on the other hand, only needs to resolve the slow time scale hence there is expected to be an improvement. 

%% file: theSOSMSTMethod.tex
% \section{The SO-SMST method}
The following section combines the two iterative methods to obtain the \textsc{SO}-\textsc{SMST} method for the approximation of trajectories near a saddle-type slow manifold.
\section{The \textsc{SO}-\textsc{SMST}  method}\seclab{nonstiff}
The outcome of \textsc{SO} and \textsc{SOF} are the functions $\eta$ and $\phi^\epsilon$ respectively. The properties of these functions are such that if the following transformation:
\begin{align}
 (x_0,y_0) \mapsto (x=x_0+\phi^\epsilon(x_0)y_0,y=\eta(x)+y_0),\eqlab{sofTransformation}
\end{align}
is applied to \eqref{ss} then one obtains the following equations of motion
\begin{align}
 \dot x_0 &=\Lambda^\epsilon (x_0)+\mathcal O(\epsilon y_0^2),\eqlab{x0eqn}\\
 \dot y_0 &=A(x_0) y_0+\mathcal O(y_0^2).\nonumber
\end{align}
Recall that $\Lambda^\epsilon(x_0)=X^\epsilon(x_0,\eta(x_0))=\mathcal O(\epsilon)$ cf. \eqref{Lambda}. The case where $X^\epsilon=\epsilon X$ will again be the primary focus. %I will again primarily focus on the case where $X^\epsilon=\epsilon X$. 
In \eqref{x0eqn} the exponentially small terms have been ignored. The remainder $\mathcal O(\epsilon y_0^2)$ in \eqref{x0eqn} shall also be ignored and we will here just consider
\begin{align}
 \dot x_0= \Lambda^\epsilon (x_0),\eqlab{x0eqn2}
\end{align}
or
\begin{align}
 x_0'=\Lambda(x_0),\eqlab{x0eqn3}
\end{align}
in terms of the slow time $\tau=\epsilon t$ and where $\Lambda(x_0) = X(x_0,\eta(x_0))$ . This formally decouples the slow variables from the fast ones. Cf. \eqref{tsp} it corresponds to projecting along the tangent space of the fibers based at $(x,\eta(x))$. A simpler but less accurate approach to obtain a formal decoupling of the equations would be to base the projection on \eqref{fs}$_{\epsilon=0}$ and the tangent spaces of the fibers at $\epsilon=0$. This corresponds to ignoring $\phi^\epsilon$ above in the transformation above and instead just consider 
\begin{align}
&(x_0,y_0) \mapsto (x=x_0,y=\eta(x)+y_0) \text{ and decouple the } \eqlab{Naive}\\
&\text{equations by ignoring a remainder of the form $\mathcal O(\epsilon y_0)$}.\nonumber
\end{align}
This approach is, as highlighted by the orders $\mathcal O(\epsilon y_0^2)$ and $\mathcal O(\epsilon y_0)$ in \eqref{x0eqn} and \eqref{Naive} respectively, less accurate. It assumes that the fibers are vertical. See also Fig. 4.1 in \cite{kri4}. The approach \eqref{Naive} is therefore particularly inaccurate in comparison with \eqref{sofTransformation} if the slow and fast variables have not been properly identified. The error in \eqref{Naive} is then $\mathcal O(y_0)$ rather than $\mathcal O(\epsilon y_0)$. See \cite[Section 8]{kri4} and \secref{Lindemann} below.

The error from replacing \eqref{x0eqn} with \eqref{x0eqn2} will be further quantified in \appref{Error}. However, within this approximation, the fast variables can be solved for using
\begin{align}
\dot y &=Y(x,y),\eqlab{y0eqn2}\\
x &= x_0+\phi^\epsilon(x_0)(I_f+\partial_x \eta(x_0) \phi^\epsilon(x_0))^{-1} (y-\eta(x_0)),\eqlab{xx0}
\end{align}
which is a non-autonomous system once $x_0=x_0(\tau)$ has been obtained from \eqref{x0eqn3}. 
% Here $I_f=\text{id}\in \mathbb R^{n_f\times n_f}$.
The equation for $x=x(x_0,y)$ \eqref{xx0} has been obtained by inserting $x=x_0+\phi^\epsilon(x_0)y_0$ into $y_0=y-\eta(x)$ and Taylor expanding about $y_0=0$. This introduces an error of $\mathcal O(\epsilon^3 y_0^2)$ which is subordinate to the remainder $\mathcal O(\epsilon y_0^2)$ which was ignored in \eqref{x0eqn}.

It will also be useful to invert $x=x_0+\phi^\epsilon(x_0)y_0$ for $x_0$ (see \secref{fhn}). By Taylor expansion about $y_0=0$ the following approximation
\begin{align}
 x_0 = x-\phi^\epsilon(x) y_0+\mathcal O(\epsilon^2 y_0^2), \eqlab{xx0invert}
\end{align}
is obtained where $\phi^\epsilon=\mathcal O(\epsilon)$.

The main purpose of this paper, is to use this principle near a saddle type slow manifold to construct the type of trajectories that are computed by the \textsc{SMST} algorithm. Consider e.g. a base trajectory $x_0=x_0(\tau)$ solving \eqref{x0eqn3} with $x_{0}(0)=x_{00}$ and $x_{0}(T)=x_{0T}$. This will be obtained by applying a quadrature to \eqref{x0eqn3}. We will return to this in \secref{rungeKutta}. The aim is then to compute an approximation of a trajectory connecting to such base trajectory, in the sense that it decays to the base trajectory exponentially fast at one end and escapes from it exponentially fast at the other end. This is done as it is done in the \textsc{SMST} algorithm \cite{guc3} by specifying the stable components $y_{s0}=\pi_s^{z}(y-\eta(x(0)))$ at $t=0$ and unstable ones $y_{uT}=\pi_u^{z}(y-\eta(x(T)))$ at the other end $t=T/\epsilon$. In particular, the approximation
$$ y-\eta(x)=(I_f+\partial_x \eta(x_0) \phi^\epsilon(x_0))^{-1} (y-\eta(x_{0}))+\mathcal O(\epsilon^2 y_0^2),$$
also used in \eqref{xx0}, is used to write these components as 
\begin{align*}
 y_{s0}=\pi_s^z(I_f+\partial_x \eta(x_{00}) \phi^\epsilon(x_{00}))^{-1} (y-\eta(x_{00})),
\end{align*}
and
\begin{align*}
 y_{uT}=\pi_u^z(I_f+\partial_x \eta(x_{0T}) \phi^\epsilon(x_{0T}))^{-1} (y-\eta(x_{0T})),
\end{align*}
respectively. Recall here the definitions of $\pi_{s,u}^z$ in \eqref{pis}. 
In contrast to the \textsc{SMST}  algorithm, however, collocation is only performed on the fast $y$-space as the base trajectory $x_0=x_0(\tau)$ solving \eqref{x0eqn3} has been obtained by direct integration. Moreover, one only needs to consider time intervals of order $t=\mathcal O(1)$ in each end. This means that the vector-field in this collocation problem has no $\epsilon^{-1}$ factor and hence the Jacobian will be well-conditioned. The length of the time intervals can be estimated through the eigenvalues of $\partial_y Y(x,\eta(x))$. Suppose that $r_0=\Vert y_{s0}\Vert$ is small and that $\lambda_s>0$ is a lower estimate of the absolute values of the real parts of the eigenvalues of $\partial_y Y(x,\eta(x))$ with negative real parts, then
\begin{align*}
 t_0 = - \lambda_s^{-1} \log\left(\frac{\mbox{tol}}{r_0}\right),
\end{align*}
is an estimate for how long it takes $y_0$ to decrease below a given tolerance $\mbox{tol}$ ($\ll \epsilon r_0^2$ cf. \eqref{epsr2} below). At $t=t_0$ we then enforce the condition that the ``unstable components'' of $y=y(t_0)$ vanish. That is $$\pi_u^z(y-\eta(x_0(\epsilon t_0)))=0.$$  %Here I use the approximation also used in \eqref{xx0} 
% \begin{align*}
%  y-\eta(x) = 
% \end{align*}
At the other end, we then let $r_T=\Vert y_{uT}\Vert$ and suppose that $\lambda_u>0$ is a lower estimate of the real parts of the eigenvalues of $\partial_y Y(x,\eta(x))$ with positive real part. Then 
\begin{align*}
 t_1 = - \lambda_u^{-1} \log\left(\frac{\mbox{tol}}{r_T}\right),
\end{align*}
is an estimate for how long it takes $y_0$ to decrease below, now in backward time, a given tolerance $\mbox{tol}$ ($\ll \epsilon r_T^2$). At $t=T/\epsilon-t_1$ one therefore enforce the condition that the ``stable components'' of $y=y(T/\epsilon-t_1)$ vanish. That is $$\pi_s^{z}(y-\eta(x_0(T-\epsilon t_1))).$$ This defines two boundary value problems on the fast space. They are solved by the same collocation principle as used in the \textsc{SMST}  algorithm on the full space using divisions of the fast time intervals $t\in [0,t_0]$ and $t\in [T/\epsilon-t_1,T/\epsilon]$ by a time step $\Delta t$. For $t\in (t_0,T/\epsilon-t_1)$, one sets $y(t) = \eta(x_0(t))$, that is $y_0=0$. Finally, $x=x(\tau)$ is obtained from \eqref{xx0}. %I refer to this method as the \textsc{SO}-\textsc{SMST} method. 

\begin{remark}\remlab{eps3}
 A consequence of the analysis in \appref{Error} (see also \remref{eps3app}) is that for the trajectories computed in \cite{guc3} where the stable and unstable components are taken from the critical manifold, setting \eqref{scomp}$_{\epsilon=0}$ and \eqref{ucomp}$_{\epsilon=0}$, respectively, to $0$, the error of the \textsc{SO-SMST} method is of order $\mathcal O(\epsilon^3)$. 
\end{remark}
%  By assumption:
%  \begin{align*}
%  \frac{d}{dt} \left(e^{\lambda t} \Delta y_0\right)&= \lambda_s e^{\lambda t}\Delta y_0 +  e^{\lambda t} \int_0^1 \partial_x Y(x+s \Delta x_0,y_0+s \Delta y_0) \Delta x_0ds \\
%  &+ e^{\lambda t} \int_0^1 \partial_y Y(x+s \Delta x_0,y_0+s \Delta y_0) \Delta y_0ds,
%  \end{align*}
%  so that
%  \begin{align*}
%     e^{\lambda t} \Vert \Delta y_0(t)\Vert & \le \lambda_s rt \Vert\Delta y_0(0)\Vert +  C_1e^{\lambda t}\epsilon r^2 + C_2 rt,
%  \end{align*}
% for some $C_1>0$ and $C_2>0$. Therefore
%  \begin{align*}
%     \Vert \Delta y_0(t)\Vert & \le  \lambda_s r t e^{-\lambda t}\Vert\Delta y_0(0)\Vert +  C_1\epsilon r^2 + C_2  re^{-\lambda t}t,
%  \end{align*}
% 
% \end{proof}

% bolig.dk christian winthersvej1b stv
% langelandsvej 6 
% along with
% \begin{align*}
%  \dot y &= Y(x,y).
% \end{align*}

% The

%  Once an approximation to the slow manifold has been obtained $y=\eta(x)$ the \textsc{SOF} method seeks a transformation that 
%  
%  Errors in the approximation of derivatives therefore first appear in the second iteration, and I can therefore take $\eta^h=\mathcal O(\epsilon)$ in \eqref{dsomr} so that the error is in fact $\mathcal O(e^{-c_p/\epsilon},\epsilon^2 \Vert h^2)$. 
% 
% CONNECT THIS WITH MODIFIED Fraser

%% file: numericalImplementationSOAndSOF.tex
\section{Numerical implementation of the iterative methods}\seclab{numsom}
If the non-linear equation for the critical manifold $y=\eta_0(x)$ can be solved explicitly, then \textsc{SO} and \textsc{SOF} can be implemented into a computer algebra system (CAS), such as Maple or Mathematica, to obtain very accurate closed-form approximations of the slow manifold and the tangent spaces to the fibers. There are other methods that could also be used to achieve this. If an explicit small parameter $\epsilon$ can be identified, then such accurate closed-form approximations are even obtainable using direct asymptotic expansions. But whereas closed-form approximations could potentially be useful in some specific cases, they have clear disadvantages in general. Firstly, the number of terms to include to obtain a desired accuracy depends in a non-trivial way on the position in phase space. Secondly, the expressions are typically very lengthy and just the evaluation of such expressions will involve many operations, which if combined with numerical integration could be costly. Finally, it is highly inflexible: If the model is slightly modified then one needs to redo the CAS-computations. 

A numerical implementation of the \textsc{SO} and \textsc{SOF} methods circumvents the highlighted issues of a CAS implementation. The only obstacle is the fact that one needs to approximate derivatives of the approximations: see $\partial_x \eta_{n-1}$ in \eqref{etan} and $\partial_x \phi^\epsilon_{n-1}$ in \eqref{SOFphin}, to obtain improved approximations. For this, the differential operator $\partial_x$ that appears in these expressions can be replaced by a finite difference operator $\delta_x^h$ satisfying
\begin{align}
 (\partial_x - \delta_x^h) f = \mathcal O(h^p),\eqlab{cond1}
\end{align}
for all smooth $f$. As an example, one could take 
\begin{align}
 \delta_{x_i}^h f(x) = \frac{f(x+he_i)-f(x-he_i)}{2h},\,i=1,\ldots,n_s,\eqlab{dxhns1}
\end{align}
with $(e_i)_j=\delta_{ij}$, $\delta_{ij}$ being Kronecker's delta, and set
\begin{align*}
 \delta_x^h f = (\delta_{x_1}^h f\,\cdots\, \delta_{x_{n_s}}^h f).
\end{align*}
Then $p=2$ since:
\begin{align}
 \Vert (\partial_{x_i} -\delta_{x_i}^h) f(x)\Vert \le \frac{\sup_{t\in [0,1]} \Vert \partial_{x_i}^3 f(x+th e_i)\Vert }{6}h^2. \eqlab{fd2}
\end{align}
Cauchy-type estimates also apply to $\delta_x^h$ in the sense that
\begin{align}
 \Vert \delta_x^h f\Vert_{\chi-\xi}\le \frac{C_{p}\Vert f \Vert_{\chi}}{\xi},\quad C_p\ge 1,\eqlab{cond2}
 \end{align}
 provided $h<\xi$ is sufficiently small and that $f$ is analytic. For \eqref{dxhns1}, for example, with $h\le \xi/2$ we have
 \begin{align*}
 \Vert \delta_{x_i}^h f\Vert_{\chi-\xi} &\le \frac{3\Vert f\Vert_\chi}{\xi},
 \end{align*}
 using \eqref{fd2} and a Cauchy estimate of $\Vert\partial_{x_i}^3 f\Vert_{\chi-\xi+h}$.
%  \note{Details:  \begin{align*} \Vert \delta_{x_i}^h f\Vert_{\chi-\xi} &\le \Vert \partial_{x_i} f\Vert_{\chi-\xi}+\frac{\Vert \partial_{x_i}^3 f\Vert_{\chi-\xi+h}}{6}h^2\\
% &\le \frac{\Vert f\Vert_\chi}{\xi}+\frac{\Vert f\Vert_\chi}{(\xi-h)^3}h^2\\
% &\le \frac{\Vert f\Vert_\chi}{\xi}+\frac{8\Vert f\Vert_\chi}{\xi}\frac{h^2}{\xi^2}\\
% &\le \frac{\Vert f\Vert_\chi}{\xi}+\frac{2\Vert f\Vert_\chi}{\xi}
% \end{align*}} 
Therefore
 \begin{align*}
  \Vert \delta_x^h f\Vert_{\chi-\xi}\le \frac{3n_s \Vert f\Vert_\chi}{\xi},
 \end{align*}
 and $C_p=3n_s>1$ in this case. The discretized version of the invariance equation
 \begin{align}
  -\epsilon \delta_x^h \eta^h X(x,\eta^h)+Y(x,\eta^h)=0,\eqlab{dsom}
 \end{align}
can then be solved by the \textsc{SO} principle to obtain an approximate solution $\eta^h$ with exponential small error:
\begin{align}
  -\epsilon \delta_x^h \eta^h(x) X(x,\eta^h(x))+Y(x,\eta^h(x))=\mathcal O(e^{-c_p/\epsilon}).\eqlab{dxheta}
 \end{align}
 Here $c_p$ is independent of $\epsilon$ and $h$.
 \begin{proposition}\proplab{etah}
  Consider a finite difference operator $\delta_x^h$ satisfying \eqsref{cond1}{cond2}. Provided $\epsilon$ is sufficiently small, then applying the \textsc{SO} method to \eqref{dsom} gives an approximate solution $\eta^h$ that satisfies
  \begin{align}
   -\partial_x \eta^h X^\epsilon(x,\eta^h(x)) + Y(x,\eta^h(x)) = \mathcal O(\epsilon \Vert (\partial_x -\delta_x^h) \eta^h \Vert+e^{-c_p/\epsilon})=\mathcal O(\epsilon h^p+e^{-c_p/\epsilon}).\eqlab{dsomr}
  \end{align}
 \end{proposition}
 \begin{proof}
%   The error comes from \eqref{dxheta} and $-\epsilon \partial_x \
We write
\begin{align*}
 \partial_x \eta^h = \delta_x^h \eta^h-(\partial_x^h -\delta_x^h )\eta^h
\end{align*}
in \eqref{dsomr} and use \eqsref{cond1}{dxheta} to estimate the error.
 \end{proof}

% \begin{remark}
 From $Y(x,\eta_0(x))=0$ one can obtain $\partial_x \eta_0 = -(\partial_y Y)^{-1} \partial_x Y$ in the first step of the iteration. The error from replacing $\partial_x$ with $\delta_x^h$ does then not appear before the second step. This gives rise to the improved order $\mathcal O(\epsilon^2 h^p)$ in \eqref{dsomr}. 
%  I could then continue with the modified \textsc{SO} method:
 \begin{corollary}\corlab{etah2}
 Suppose $\eta_0=\eta_0(x)$ is known. Then, provided $\epsilon$ is sufficiently small, applying the following procedure:
 \begin{align}
  \eta_n^h &= -A_{0}^{-1}\rho_{n-1},\eqlab{etanh}\\
  \rho_{n-1}&=-(\partial_x \eta_0(x)+\delta_x^h (\eta^{n-1,h}(x)-\eta_0(x))) X^\epsilon(x,\eta^{n-1,h})+Y(x,\eta^{n-1,h}),\eqlab{split}\\
  A_{0}(x) &=-\partial_x \eta_0\partial_y X^\epsilon (x,\eta_0)+\partial_y Y(x,\eta_0),\nonumber\\
    \eta^{n,h}&=\eta_0+\sum_{k=1}^{n}\eta_k^{h},\nonumber
  \end{align}
  generates an approximate solution $\eta^{h} = \eta_0+\sum_{n=1}^{N(\epsilon)}\eta_n^{h}$, $N(\epsilon)= \mathcal O(\epsilon^{-1})$, satisfying
  \begin{align}
   -\partial_x \eta^{h} X^\epsilon(x,\eta^{h}) + Y(x,\eta^{h}) = \mathcal O(\epsilon \Vert (\partial_x -\delta_x^h) (\eta^h-\eta_0) \Vert+e^{-c_p/\epsilon})=\mathcal O(\epsilon^2 h^p +e^{-c_p/\epsilon}).\eqlab{dsomr2}
  \end{align}
    The derivative $\partial_x \eta_0$ is obtained through $Y(x,\eta_0)=0$:
  \begin{align*}
  \partial_x \eta_0&=-(\partial_y Y(x,\eta_0))^{-1} \partial_x Y(x,\eta_0).\nonumber
  \end{align*}
 \end{corollary}
 Note how $\partial_x \eta^{n-1,h}$ is approximated as $\partial_x \eta_0+\delta_x^h (\eta^{n-1,h}-\eta_0)$ in \eqref{split}. We further stress the simplicity of this method: It only requires the first partial derivatives of the vector-field.

 It is easy to obtain a similar result for the discretization of \textsc{SOF} method:
\begin{proposition}\proplab{sofmod}
Assume that the conditions \eqsref{cond1}{cond2} hold true and let $\eta^h$ be given as in \propref{etah}. Then provided $\epsilon$ is sufficiently small, the function ${\phi}^{\epsilon,h}=\sum_{n=0}^{N(\epsilon)}\phi_n^{\epsilon,h}$, $N(\epsilon)=\mathcal O(\epsilon^{-1})$, where 
% the partial sums $\phi^{n,h}=\sum_{i=0}^n\phi_i^h(x_0)$ satisfy
% \begin{align*}
%   \epsilon (\partial_x X_0+\partial_y X_0 \delta_{x}^h {\eta})\phi^{n,h} -\epsilon \delta_x^h \phi^{n-1,h} X_0+\partial_y X_0 - \phi^{n,h}& (-\epsilon \delta_{x}^h {\eta} \partial_y X_0+\partial_y Y_0) = 0,\\
% $$ and
\begin{align*}
\phi_n^{\epsilon,h}&=\mu_n A_0^{-1},\\
%   \mu_{n}(x)-\phi_n(x) A_0(x)= 0,\\
  \mu_n &=-\delta_x^h \phi_{n-1}^\epsilon X^\epsilon(x,\eta^h)\\
  &+\left(\partial_x X^\epsilon(x,\eta^h)+\partial_y X^\epsilon(x,\eta^h(x)) \delta_{x}^h {\eta}^h(x)\right)\phi_{n-1}^{\epsilon,h}\\
  &-\phi_{n-1}^{\epsilon,h} a,\\
%   \Lambda(x) &= X_0(x,\eta^h(x)),\\
%   \delta_x^h \Lambda(x) &= \partial_x X_0(x,\eta^h(x))+\partial_y X_0(x,\eta^h(x)) (\partial_{x} {\eta_0}(x)+\delta_x^h(\eta^h(x)-\eta_0(x))),\\
 A_0(x) &= -\partial_{x} {\eta_0} \partial_y X^\epsilon(x,\eta_0)+\partial_y Y_0(x,\eta_0),\\
 a(x) &= -\delta_{x}^h {\eta^h} (x)\partial_y X_0^\epsilon(x,\eta^h(x))+\partial_y Y_0(x,\eta^h)-A_0(x).
 \end{align*}
 %for $n\ge 1$ and with $\mu_0$ given by \eqref{mu0},
solves \eqref{phieqn} up to the error 
\begin{align}
\mathcal O(\epsilon \Vert (\partial_x - \delta_x^h)\phi^{\epsilon,h} \Vert +e^{-c_p/\epsilon}).\eqlab{errorPhinh}
\end{align}
\end{proposition}
\begin{proof}
 One can proceed as in \propref{etah}. Note that $\partial_x \Lambda^\epsilon=\partial_x X^\epsilon(x,\eta^h)+\partial_y X^\epsilon(x,\eta^h) \partial_x {\eta}^h(x)$ has been replaced by
 \begin{align*}
  \partial_x X^\epsilon(x,\eta^h)+\partial_y X_0(x,\eta^h) \delta_{x}^h {\eta}^h.
 \end{align*}
%  and that $\phi = ^\epsilon
%  Similarly, I have in the expression for $a$ replaced $A$ by $-\epsilon \delta_{x}^h {\eta^h} \partial_y X_0(x,\eta^h)+\partial_y Y_0(x,\eta^h)$. 
\end{proof}
% \note{I can also improve this method further by taking $A$ to be $A_0$.}

% Since the function $\phi^h$ enters the transformation \eqref{x0} as $\epsilon \phi^h$ this error appears as $\mathcal O(\epsilon^2 h^p+e^{-c/\epsilon})$ as in \eqref{dsomr2} and subordinate to the error in \eqref{dsomr}. This is also why I in \propref{sofmod} settle with approximating $\partial_x \phi^h$ by $\delta_x^h\phi ^h$ instead of $\partial_x \phi_0+\delta_x^h(\phi^h-\phi_0)$ (as it was done in \corref{etah2}).
\begin{remark}\remlab{errorPhinhRemark}
% In the expression for \eqref{errorPhinh} I have used the fact that $\epsilon \phi$ itself is $\mathcal O(\epsilon)$. If the slow variables have not been properly identified then $\epsilon \phi$ is not necessarily small (see \secref{Lindemann} for an example), one should just think of the $\epsilon$ ``glued'' to the $\phi$, and then the error in \eqref{errorPhinh} will only be $\mathcal O(\epsilon h^p+e^{-c/\epsilon})$. To improve the order to \eqref{errorPhinh} in this case one could do as in \corref{etah2} and replace $\delta_x^h (\epsilon \phi ^h)$ by $\partial_x (\epsilon \phi_0)+\delta_x^h(\epsilon \phi^h-\epsilon \phi_0)$. 
If $X^\epsilon=\epsilon X$ and the slow and fast variables have been properly identified, then $\Vert \phi^\epsilon \Vert=\mathcal O(\epsilon)$ cf. \eqref{phi0} and the order in \eqref{errorPhinh} will be $\mathcal O(\epsilon^2 h^p+e^{-c/\epsilon})$ as in \corref{etah2}. If $X^\epsilon = \mathcal O(1)$ and only satisfies \eqref{XepsSmall} then $\phi^\epsilon=\mathcal O(1)$ and the error in \eqref{errorPhinh} is therefore only $O(\epsilon h^p+e^{-c/\epsilon})$ which is slightly less accurate. To improve it by a factor of $\epsilon$ one could do as in \corref{etah2} and replace $\delta_x^h \phi^{\epsilon,h}$ by $\partial_x \phi_0^\epsilon+\delta_x^h(\phi^{\epsilon,h}-\phi_0^{\epsilon})$ and use that $\partial_x \phi_0^\epsilon$ can be obtained analytically from \eqref{phi0}.
\end{remark}

% By applying the \textsc{SO} method one can accurately approximate the slow manifold $y=\eta(x)$. 
The \textsc{SO-SMST} method requires the propagation of the slow variables on the slow manifold. For this the discretized \textsc{SO} method will be integrated into a Runge-Kutta quadrature scheme as explained in the following section.
\subsection{Modified Runge-Kutta scheme and $h$-grid}\seclab{rungeKutta}
% Suppose that an accurate approximation to the slow manifold $y=\eta(x)$ has been obtained. 
On the slow manifold, the motion of the slow variables is given in terms of the reduced system:
% \eqref{x0eqn3}:
% reduced slow subsystem:
\begin{align}
 x' =\Lambda(x)\equiv X(x,\eta(x)).\eqlab{reducedEqn}
\end{align}
Recall that $()'$ denotes differentiation with respect to the slow time $\tau=\epsilon t$. 
% In the numerics that I will present I will be using a classical $4$th order Runge-Kutta scheme on this reduced problem. 
The solution of this reduced system can be approximated by applying a quadrature scheme. A classical $4$th order Runge-Kutta scheme will be used. The modifications from one scheme to another is straightforward and not important for what will be presenting. %I will therefore for simplicity just focus on this particular choice.

Starting from $x(\tau)=x_0$ the $4$th order Runge-Kutta method approximates $x(\tau+\Delta \tau)=x_1$ as
\begin{align*}
 x_1 = x_0+\frac{1}{6}\left(\kappa_1 + 2\kappa_2 + 2\kappa_3 +\kappa_4\right),
\end{align*}
where
\begin{align*}
 \kappa_1 &= \Delta \tau X(x_0,\eta(x_0)),\\
 \kappa_2 &= \Delta \tau X(x_0+0.5 \kappa_1,\eta(x_0+0.5 \kappa_1)),\\
 \kappa_3&=\Delta \tau X(x_0+0.5 \kappa_2,\eta(x_0+0.5\kappa_2)),\\
 \kappa_4&=\Delta \tau X(x_0+\kappa_3,\eta(x_0+\kappa_3)).
\end{align*}
See e.g. \cite{but1}. Here $\Delta \tau$ is the time step on the slow time scale. The local error is $\mathcal O(\Delta \tau^5)$ while the accumulated error is $\mathcal O(\Delta \tau^4)$. 
%For this one could obtain the function $\eta$ through the \textsc{SO} method on a grid in $x$ and then interpolate to find $\eta$ at points where 
%The \textsc{SO} method will be used to determine $\eta=\eta(x)$ at the four different points 
The Runge-Kutta scheme will therefore require the determination of $\eta(x)$ at the following different $x$-values: 
\begin{align}
x=x_0,\quad x_0+0.5 \kappa_1,\quad x_0+0.5 \kappa_2,\quad \mbox{and}\quad x_0+\kappa_3.\eqlab{rkpts}
\end{align} 
This is where the discretized \text{SO} method will be used. To explain the construction of the finite difference operator $\delta_x^h$ \eqref{cond1}, consider for example the determination of $\eta(x_0)$. A grid is introduced around $x_0$, and $\delta_x^h$ is then determined by Lagrange interpolation polynomials derived from function values at the $3^{n_s}$ points:
\begin{align}
 x_0+e_h\quad \mbox{with}\quad (e_h)_i = 0,\,\mbox{or}\, \pm h\quad \text{for}\quad i=1,\ldots,n_s.\eqlab{discretized}
\end{align}
This gives $p=2$ in \eqsref{dsomr}{dsomr2}. The following is important: Since $h$ is small $A_0=A_0(x)$ can be taken to be a constant on the $h$-grid. The error from this can be collected into $a_n=\mathcal O(\epsilon\delta_0)$ cf. \eqref{ean} and does therefore not change the result. The LU-decomposition of $A_0$ can therefore be stored outside the iteration in $n$. Cf. \eqref{dsomr} one can by letting $\epsilon h^2 \sim \Delta \tau^5$ or $h^2 \sim \epsilon^{-1} \Delta \tau^5$ match the order of the Runge-Kutta scheme with the order of the approximation of the slow manifold $\eta$. If $\eta_0$ is used explicitly as described in \corref{etah2} then one can instead let 
\begin{align}
h^2 \sim \epsilon^{-2} \Delta \tau^5.\eqlab{hgrid}
\end{align}

This quadrature scheme for the propagation will be referred to as the \textit{modified Runge-Kutta scheme}. For moderate values of $n_s$, say $1$, $2$ or $3$, the cost involved in each time step is comparable to the cost of a single step in an implicit method of the same order applied to the full system. Indeed, for both methods the computational cost is expected to be dominated by the cost required to obtain a solution of a linear equation. The modified Runge-Kutta scheme requires the solution \eqref{etanh} while an implicit method requires the solution of another linear equation on the full space in the application of the Newton method. For larger values of $n_s$ the reduced quadrature suffers from having to resolve $\partial_x$ using $3^{n_s}$ number of points in the $h$-grid. 

\begin{remark}\remlab{slowManifoldGrid}
 Alternative to the method outlined above, one could compute the slow manifold on a larger grid and then interpolate to obtain the values of $\kappa_i$, $i=1,2,3,4$, needed in the Runge-Kutta scheme. This, however, involves unnecessarily many computations. The direct use of the \textsc{SO} method in the forward integration only involves computations of the slow manifold where it is needed for the propagation of the variables.
\end{remark}

The following section combines several examples for testing and demonstrating the \textsc{SO-SMST} method. 

%% file: examples1.tex
\section{Examples}\seclab{examples}
This section includes five different examples. %I collect the main points in the following:
\begin{itemize}
 \item In \secref{toy} a toy model is considered in order to test the iterative methods and demonstrate their stated properties;
 \item \Secref{SOSMSTExact} includes the boundary value problem \eqref{bctoy} where the \textsc{SO-SMST} methods gives the desired solution up to exponentially small terms;
 \item \Secref{num} considers a nonlinear model of reciprocal inhibition. A boundary value problem with fixed boundaries is considered. The results from applying the \textsc{SO-SMST} method to this problem are compared to results obtained from the \textsc{SMST} method. It is demonstrated that there is no issues with obtaining a solution using the \textsc{SO-SMST} method for $\epsilon\rightarrow 0$. %Values as low as $10^{-7}$ are considered.
  \item In \secref{fhn} the FitzHugh-Nagumo model is considered. A homoclinic solution is computed and it is shown how the \textsc{SO-SMST} can be combined with other methods to compute a full orbit. The projection based on the tangent spaces to the fibers through the function $\phi^\epsilon$ is also compared with the result of just using the tangent spaces with $\epsilon=0$ as explained in \eqref{Naive}. An improvement in accuracy by a factor of $10^{-3}$ is observed when the projection is based on $\phi^\epsilon$ without any detectable difference in computational time. % The errors are  $5\times 10^{-9}$ using $\phi^\epsilon$ and $2\times 10^{-6}$ using \eqref{Naive} without any detectable difference in computational time. 
 \item In \secref{Lindemann} the \textsc{SO-SMST} method is successfully applied to the Lindemann mechanism \cite{calder2011properties,gou1,kri4}, which is an example of a system where the slow and fast variables have not been properly identified. Applying the ``naive'' projection as described in \eqref{Naive} gives rise to $\mathcal O(1)$-errors.
\end{itemize}

\subsection{Testing the iterative methods: a toy example with a saddle-type slow manifold}\seclab{toy}
 Consider the following toy-problem:
 \begin{align}
  \dot x &= \epsilon \begin{pmatrix}
             \cos x_1 +y_1 +y_2\cos x_2 \\
             - \sin x_2+y_2 +  y_1\sin x_1
            \end{pmatrix},\eqlab{toy}\\
            \dot y &=\begin{pmatrix}
            \cos x_2-y_1\\
           -\sin x_1 +y_2
            \end{pmatrix}.
            \end{align}
            From the Jacobian matrix $\partial_y Y=\text{diag}\,(-1,1)$ it follows that the slow manifold is of saddle type. Since the problem \eqref{toy} is linear in the fast variables the \textsc{SO} method can then be used to compute $\eta=(\eta_1,\eta_2)$ explicitly using Maple. Terms up to and including order $\epsilon^{8}$ will be used in the following. In \figref{etalinh} such a reference CAS-solution is compared with a numerical solution $\eta^h$ obtained using the discretized \textsc{SO} method, see \propref{etah} and \corref{etah2}, at $x=(-0.5,-0.7)$. The finite difference operator $\delta_x^h$ was second order ($p=2$) and based on Lagrange interpolation, as explained after \eqref{discretized}. In both figures $h=\epsilon$. Figure (a) is obtained using \propref{etah} whereas figure (b) is obtained using \corref{etah2} and
\begin{align*}
 \eta_0  =\begin{pmatrix}
           \cos x_2\\
           \sin x_1      
          \end{pmatrix},
\end{align*}
explicitly. From the log-log scale in \figref{etalinh} we numerically determine the orders of the approximations to be $\mathcal O(\epsilon^3)$ and $\mathcal O(\epsilon^4)$. This is in agreement with the analysis above, see \eqref{dsomr}$_{h=\epsilon}$ and \eqref{dsomr2}$_{h=\epsilon}$, respectively, with $p=2$.

\begin{figure}[h!]
\begin{center}
\subfigure[]{\includegraphics[width=.475\textwidth]{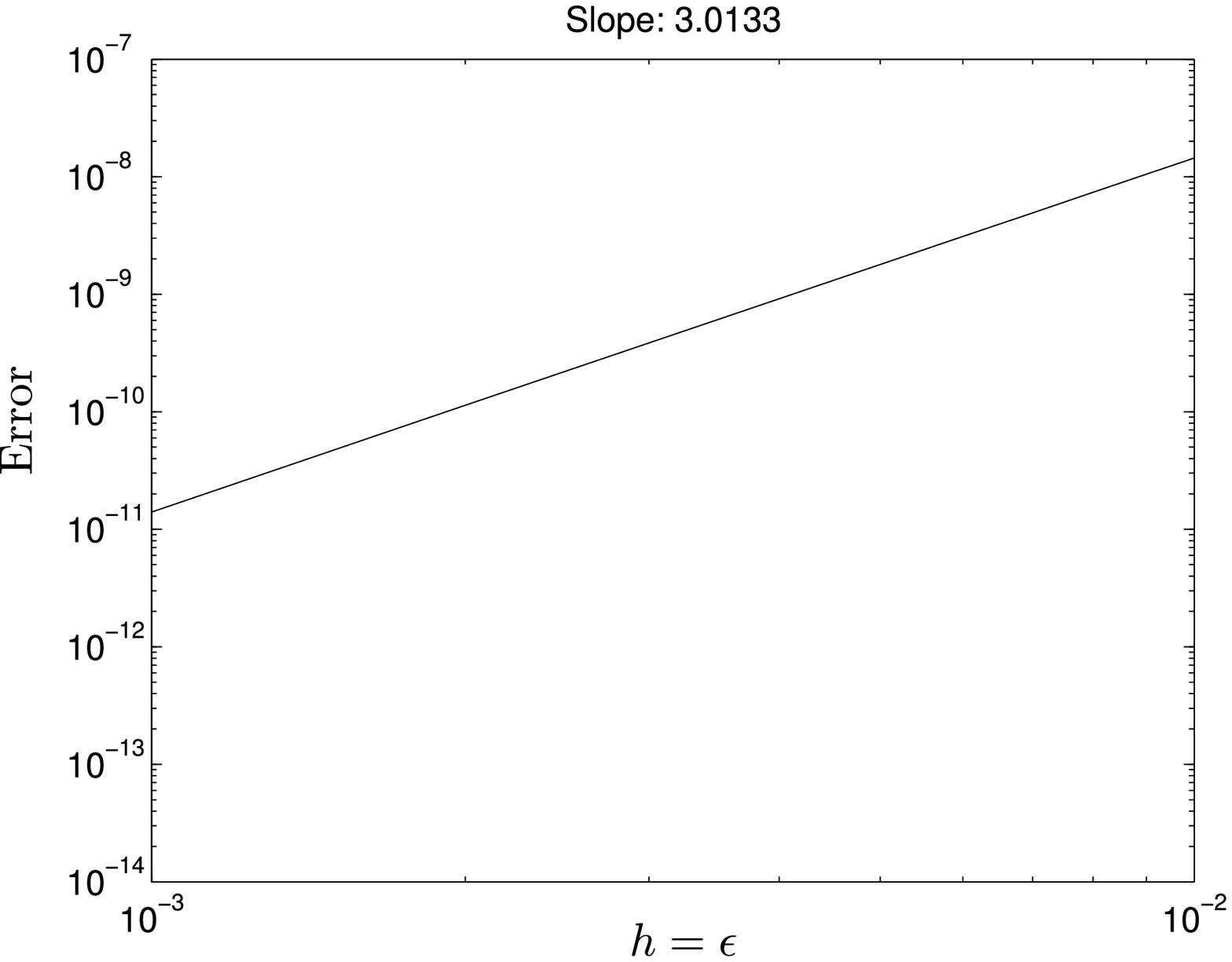}}
\subfigure[]{\includegraphics[width=.475\textwidth]{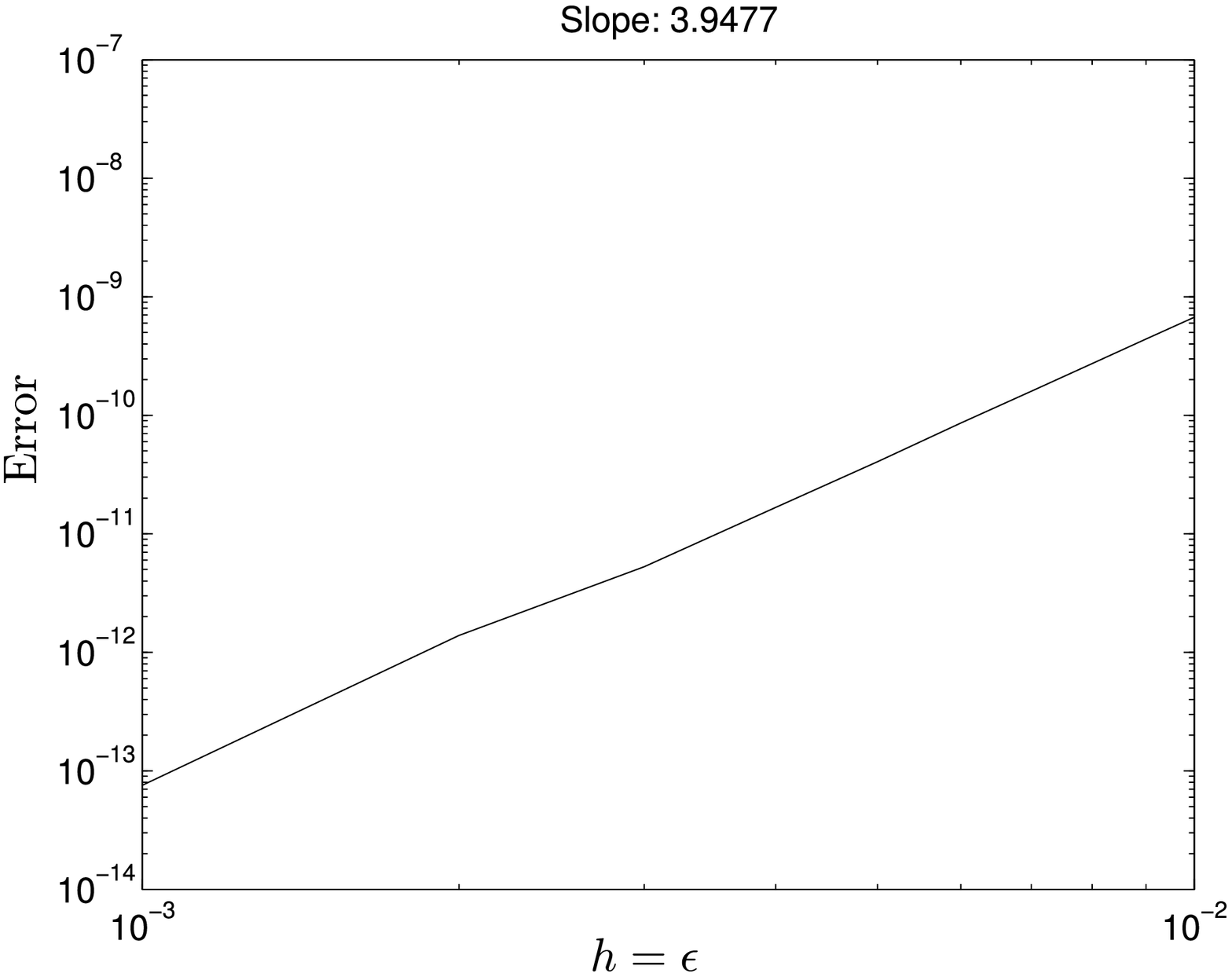}}
\end{center}
\caption{The error $\Vert \eta^h-\eta\Vert$ for the problem \eqref{toy} at $x=(-0.5,-0.7)$ for $h=\epsilon$ and as a function of $\epsilon$.  The approximation $\eta^h$ in (a) is obtained using \propref{etah} while $\eta^h$ in (b) is is based on \corref{etah2} and the explicit use of $\partial_x \eta_0$. The finite difference operator is second order ($p=2$) with respect to the grid size $h$. The reference solution $\eta$ is obtained using Maple (accurate up to $\mathcal O(\epsilon^9)$). The slopes of $\approx 3$ and $\approx 4$ correspond to orders of $\mathcal O(\epsilon^3)$ and $\mathcal O(\epsilon^4)$, respectively, which are in agreement with the analysis.}
% /home/krkri/dtu/AProf/research/matlab/Collocation method/linear_example_2304/test_etafunc.m
\figlab{etalinh}
\end{figure}
% $\eta$ and $\phi$ can be computed explicitly 
In \figref{rk4test} the results of applying the modified Runge-Kutta scheme to
\begin{align*}
  x' &=  \begin{pmatrix}
             \cos x_1 +\eta_1 +\eta_2\cos x_2 \\
             - \sin x_2+\eta_2 + \eta_1\sin x_1 
            \end{pmatrix}
\end{align*}
for different values of $\Delta \tau$ and $\epsilon = 10^{-3}$, is compared with a high-precision reference solution obtained using Matlab's ode45 applied to \eqref{reducedEqn}. The integration is initialized $x(0)=(                                          -0.5,                                                                                                              -0.7)$ and integrated up until $\tau = 10$.
The absolute and relative tolerances of ode45 were set to $10^{-12}$ and $\eta=(\eta_1,
            \eta_2
           )$ from the Maple computation, again including terms up to order $\epsilon^{8}$, was used in the ode45 solver to obtain an accurate reduction to the slow manifold. In the modified Runge-Kutta scheme the method described in \corref{etah2} was used with $\eta_0$ and $\partial_x \eta_0$ used explicitly, stopping the \textsc{SO} iteration when the error 
           \begin{align}
            \Vert -\delta_x^h \eta X^\epsilon(x,\eta)+Y(x,\eta)\Vert, \eqlab{error}
           \end{align}
           had reached below a tolerance $\mbox{tol}$, which was set to be $10^{-12}$. The grid size was set to be $$h=\min\{10^{-2},0.1 \Delta \tau^5/\epsilon^2\}.$$ 
           The factor of $0.1$ was introduced as a ``safety factor'' aiming to ensure that the error from the approximation of $\eta$ was subordinate to the error of the Runge-Kutta scheme, see also \eqref{hgrid}. %Moreover,  $h$ was always $\le 10^{-2}$. %A pseudo-code of the Runge-Kutta implementation is provided in \appref{rkapp}. 
%            Figure (a) shows the global $\mathcal O(\Delta \tau^4)$ convergence while (b) shows the reference ode45 solutions $x_1$ $(-)$ and $x_2$ ($--$) in comparison with the solutions  ($\diamond$) obtained by the modified Runge-Kutta scheme for $\Delta \tau = 0.5$. Finally (c) shows the result of direct simulation for $16$ different initial conditions that are obtained as displacements by $\pm 10^{-4},\,\pm 10^{-5},\ldots,\,\pm 10^{-11}$ from the slow manifold along its unstable directions. I use Matlab's ode15s with tolerances set to $10^{-12}$ for the propagation on the full space. Of all the pairs, only for the one with $\pm 10^{-11}$ do the trajectories jump in the same direction. This gives reason to believe that the slow manifold is 
% correct up to $\pm 10^{-10}$ but not more accurate than $\pm 10^{-11}$. 
%            
           \figref{rk4test} (a) compares the reference ode45 solutions $x_1$ $(-)$ and $x_2$ ($--$) with the solutions ($\diamond$) obtained by the modified Runge-Kutta scheme for $\Delta \tau = 0.5$. It is observed that the $\diamond$'s agree with the accurate reference solutions. The maximal deviation was $3\times 10^{-3}$ for this value of $\Delta \tau$. \figref{rk4test} (b) shows the result of direct simulation for $16$ different initial conditions that are obtained as displacements by $\pm 10^{-4},\,\pm 10^{-5},\ldots,\,\pm 10^{-11}$ from the slow manifold along its unstable directions. Matlab's ode15s was used with tolerances set to $10^{-12}$ for the propagation on the full space. Of all the pairs, only for the one with $\pm 10^{-11}$ do the trajectories jump in the same direction. This gives reason to believe that the slow manifold is 
correct up to $\pm 10^{-10}$ but not more accurate than $\pm 10^{-11}$.

           %Finally (c) shows the result of a direct simulation $(-
%.-)$ of the full system starting from a very accurate estimate of $\eta(x_{00})$. After $t\approx 28$ the solution is observed to start diverging from the slow manifold.
%            I started each of these iterations from $\eta_0$, or the previous value $\eta$, depending on which had the smallest error in the sense of \eqref{error}. 

\begin{figure}[h!]
\begin{center}
% % \subfigure[]{\includegraphics[width=.497\textwidth]{./rk4test2.eps}}
\subfigure[]{\includegraphics[width=.45\textwidth]{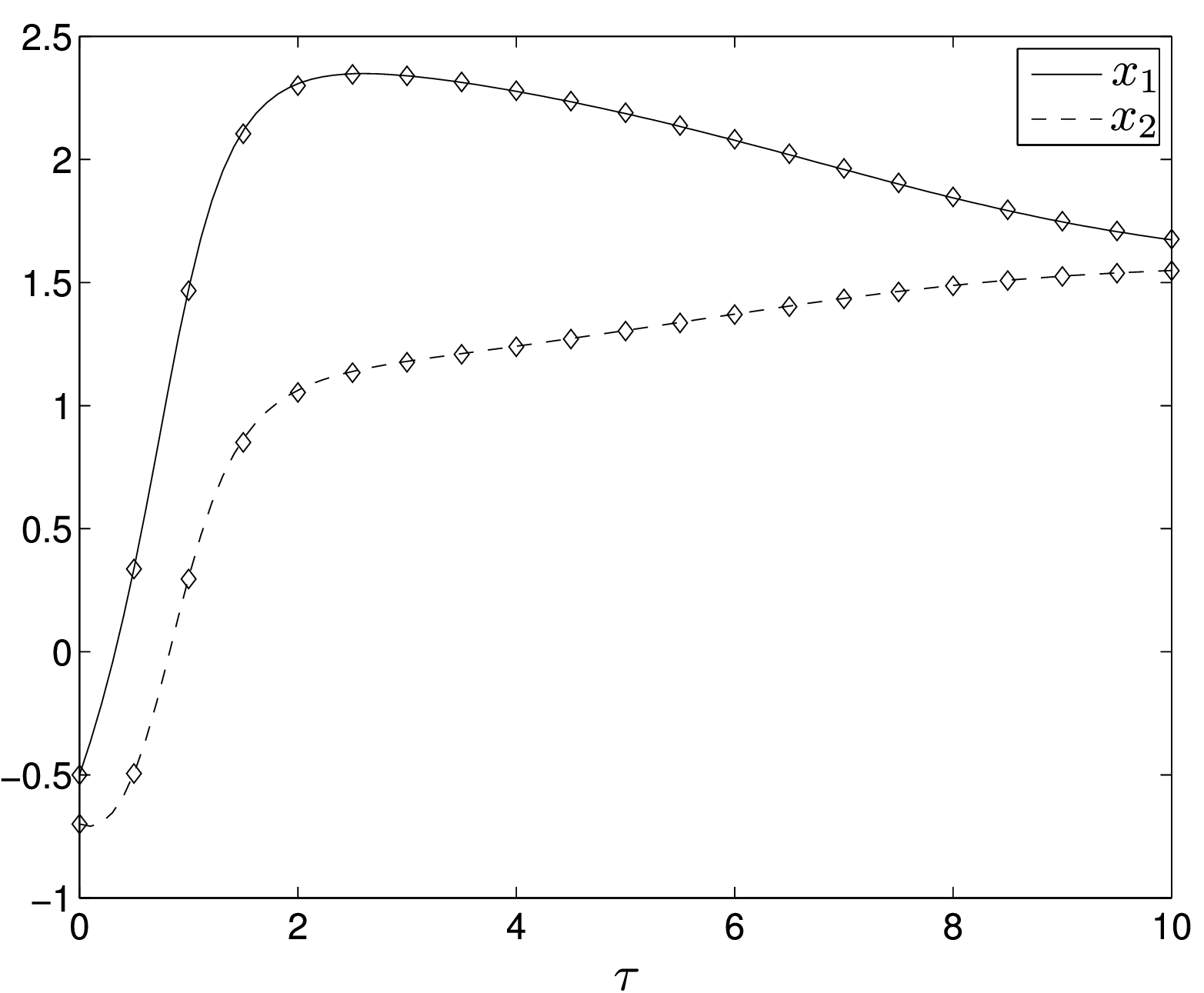}}
% \subfigure[]{\includegraphics[width=.475\textwidth]{./x1x2testzoom.eps}}
\subfigure[]{\includegraphics[width=.5\textwidth]{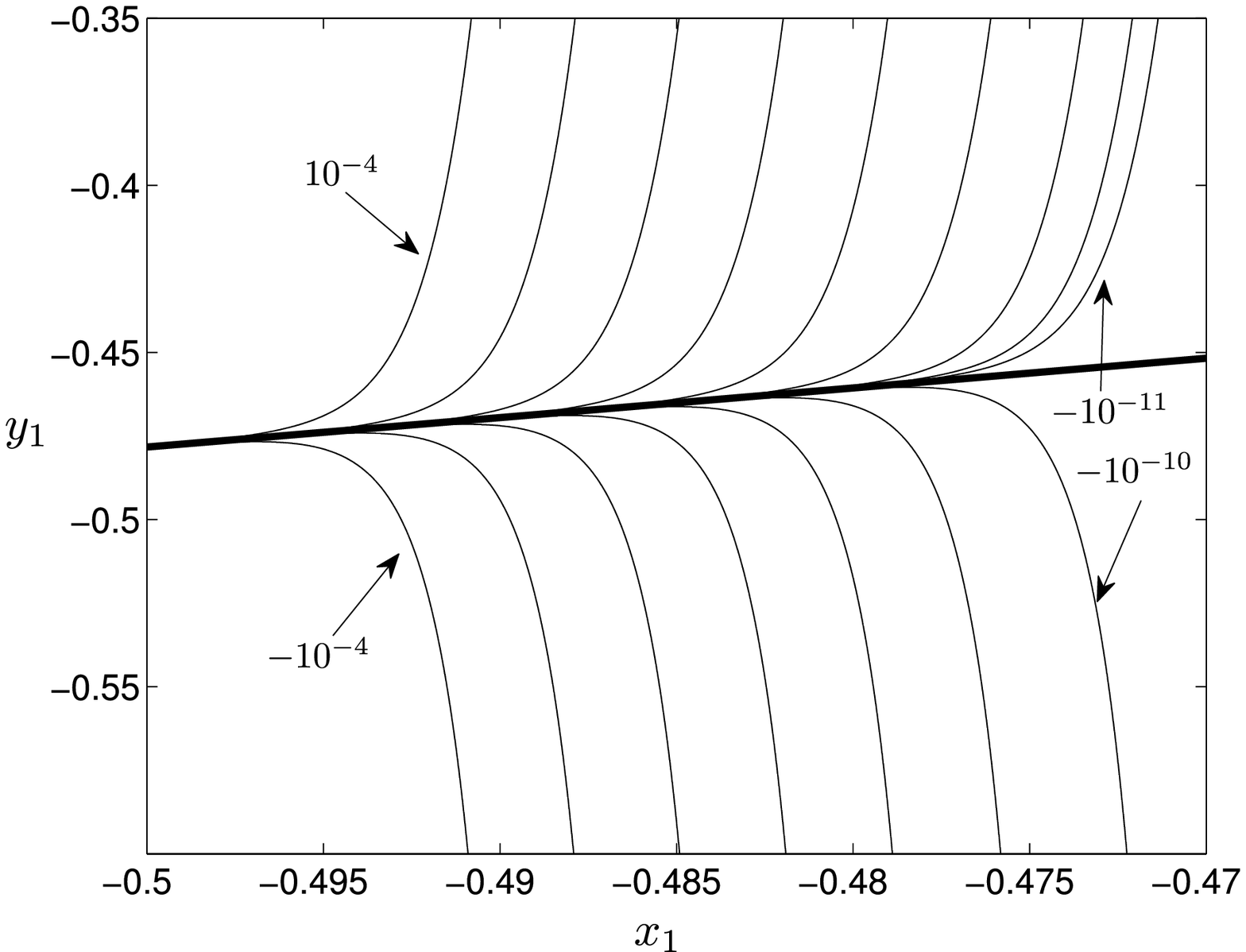}}
\end{center}
% \caption{The results of applying the modified Runge-Kutta scheme to the problem \eqref{toy}. Figure (a) shows the $\mathcal O(\Delta \tau^4)$ convergence of the Runge-Kutta scheme. Figure (b) compares the solution obtained by the modified Runge-Kutta scheme $(\diamond)$ with an accurate reference solution $(-,--)$. Finally (c) shows the result of an accurate direct simulation on the full system for 16 initial conditions displaced by distances of $\pm 10^{-4},\,\pm 10^{-5},\ldots,\,\pm 10^{-11}$ from the slow manifold along its unstable directions. Of all the pairs, only for the one with $\pm 10^{-11}$ do the trajectories jump in the same direction. }%Of all the pairs, only for the one with $\pm 10^{-11}$ do the trajectories jump in the same direction.  }
\caption{The results of applying the modified Runge-Kutta scheme to the problem \eqref{toy}. Figure (a) compares the solution obtained by the modified Runge-Kutta scheme $(\diamond)$ with an accurate reference solution $(-,--)$. Figure (b) shows the result of an accurate direct simulation on the full system for 16 initial conditions displaced by distances of $\pm 10^{-4},\,\pm 10^{-5},\ldots,\,\pm 10^{-11}$ from the slow manifold along its unstable directions. Of all the pairs, only for the one with $\pm 10^{-11}$ do the trajectories jump in the same direction. }%Of all the pairs, only for the one with $\pm 10^{-11}$ do the trajectories jump in the same direction.  }
\figlab{rk4test}
% /home/krkri/dtu/AProf/research/matlab/Collocation method/2_slow_2_fast_2404/main files/run_rk3sf.m
\end{figure}

 The function $\phi^\epsilon$ can also be computed explicitly for the toy problem \eqref{toy}:
 \begin{align*}
  \phi^\epsilon = \epsilon\begin{pmatrix}
                           -1 & \cos x_2\\
                           -\sin x_1 & 1
                          \end{pmatrix}+\mathcal O(\epsilon^2).
 \end{align*}
 Again Maple is used with terms up to order $\epsilon^8$. In \figref{philinh} this is compared with a numerical solution $\phi^{\epsilon,h}$ at $x=(-0.5,-0.7)$ taking again $h=\epsilon$. In agreement with the analysis, cf. \eqref{errorPhinh}$_{h=\epsilon}$ with $p=2$, the slope in the log-log scale is $\approx 4$. %In \figref{connectionToy} we have used the \textsc{SO-SMST} method to obtain a connection to the base trajectory. This will be investigated further in \secref{num} and \secref{fhn} on real-world problems so the comparison with \textsc{SMST} is delayed until till then.
\begin{figure}[h!]
\begin{center}
{\includegraphics[width=.575\textwidth]{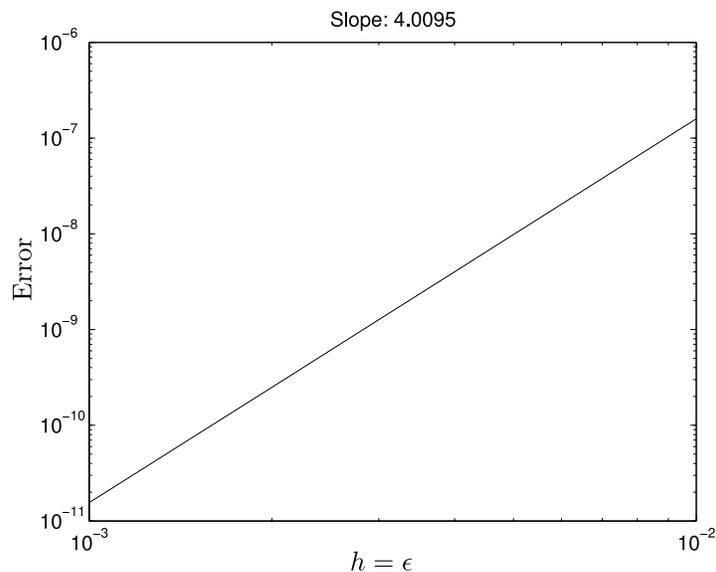}}
% \subfigure[]{\includegraphics[width=.475\textwidth]{./etavsh_linear.eps}}
\end{center}
\caption{Comparison of $\phi^{\epsilon,h}$ (obtained using \propref{sofmod} with $\delta_x^h$ a second order finite difference) with $\phi^\epsilon$ (obtained using Maple, accurate up to order $\mathcal O(\epsilon^9)$) for the toy problem \eqref{toy}. The grid size $h$ is set equal to $\epsilon$ and the difference $\Vert \phi^{\epsilon,h}-\phi^\epsilon\Vert$ is computed for different values of $\epsilon$. The computations are based on $x=(-0.5,-0.7)$. The slope $\approx 4$ corresponds to an order of $\mathcal O(\epsilon^4)$ which is in agreement with \propref{sofmod} and \eqref{errorPhinh}$_{h=\epsilon}$.}
\figlab{philinh}
\end{figure}

\subsection{An example where SO-SMST gives the result up to exponentially small terms}\seclab{SOSMSTExact}
 Consider the following linear, singular perturbed boundary value problem:
 \begin{align}
  \epsilon u''(\tau) + u'(\tau) = 1,\quad u(0)=1=u(1),\eqlab{shi1}
 \end{align}
 taken from \cite{shi}.
 Setting $$x=u,\quad \text{and} \quad y=u',$$ gives the following slow-fast system
 \begin{align*}
  \dot x & = \epsilon  y\\
 \dot y &=1-y,
 \end{align*}
 with respect to the fast time $t=\epsilon^{-1}\tau$. Here $y=1$ is a normally hyperbolic invariant manifold. The \textsc{SOF} method gives $\eta=1$ and
% \begin{align*}
 $\phi^\epsilon = -\epsilon$ both exact in one step. Since the problem is linear and $\phi^\epsilon$ is independent of $x$ this also implies that the \textsc{SO}-\textsc{SMST} method gives 
 \begin{align*}
 x(\tau) = \tau+e^{-\tau/\epsilon},
 \end{align*}
 which agrees with the true solution of \eqref{shi1}
 \begin{align*}
x(\tau) = \tau+\frac{e^{-\tau/\epsilon}-e^{-1/\epsilon}}{1-e^{-1/\epsilon}},
 \end{align*}
 up to exponentially small terms. 
%  Indeed, applying $(x,y)=(x_0-\epsilon y_0,1+y_0)$ gives
%  \begin{align*}
%  \dot x_0 &= \epsilon,\\
%  \dot y_0 &= - y_0.
%  \end{align*}
%  Solving the first equation gives $x_0=\epsilon \tau$ using that $x_0(1)=1$. Then we can perform collocation on the fast space using the fast time scale and specify the boundary condition $x=1$ as the boundary condition at $t=0$. This gives
%  
This is not the case for the classical \textsc{SMST} method. See \cite{shi}. % unless a Shishkin mesh is employed, see e.g. \cite{shi}. 
% \end{align*}
% also exactly.
% \end{example}

% \section{Examples: Testing the SO-SMST method}
% % \begin{example}
% \subsection{An example where SO-SMST gives the exact result}
%  Consider the following linear singular perturbed boundary value problem:
%  \begin{align*}
%   \epsilon u''(\tau) + u'(\tau) = 1,\quad u(0)=1=u(1),
%  \end{align*}
%  taken from \cite{shi}.
%  Setting $x=u$ and $y=\epsilon u'$ gives the following slow-fast system
%  \begin{align*}
%   \dot x & = y\\
%  \epsilon \dot y &=-y+\epsilon,
%  \end{align*}
%  with respect to the fast time $t=\epsilon^{-1}\tau$. 
%  Here $y=0$ with $\epsilon=0$ is a normally hyperbolic critical manifold. The \textsc{SOF} method gives $\eta=\epsilon$ and
% % \begin{align*}
%  $\epsilon \phi = -1$ both exact in one single step. Since the problem is linear with $\epsilon \phi$ independent of $x$ this also implies that the \textsc{SO}-\textsc{SMST} method gives the true solution
%  \begin{align*}
%  u(\tau) = \tau+\frac{1+e^{-\tau/\epsilon}}{1-e^{-1/\epsilon}},
%  \end{align*}
%  up to any desired accuracy uniformly in $\epsilon\ll 1$. This is not the case for the \textsc{SMST} method unless a Shishkin mesh is employed, see e.g. \cite{shi}. 
% % \end{align*}
% % also exactly.
% % \end{example}

\subsection{Numerical results for a model for reciprocal inhibition}\seclab{num}
To demonstrate the \textsc{SO-SMST} method further we consider a model for a pair of neurons coupled by reciprocal inhibition \cite{row1}:
\begin{align*}
 \dot q_1 &= \epsilon (-q_1 + sv_1),\\
 \dot q_2 &= \epsilon (-q_2 + sv_2),\\
 \dot v_1 &=-\left(v_1 -a\tanh \left(\frac{\sigma_1 v_1}{a}\right)+q_1 +\omega f(v_2)(v_1-r)\right),\\
 \dot v_2 &=-\left(v_2 -a\tanh \left(\frac{\sigma_2 v_2}{a}\right)+q_2 +\omega f(v_1)(v_2-r)\right),
\end{align*}
with 
\begin{align*}
 f(x) = \left(1+\exp(-4\gamma (x-\theta))\right)^{-1}.
\end{align*}
Here the fast variables $v_1$ and $v_2$ are interpreted as the membrane potential of two neurons coupled synaptically through the terms involving $f$. The slow variables $q_1$ and $q_2$ represent the gating of membrane channels in the neurons. The model does not incorporate the fast membrane currents, and in that sense it is slightly caricatural. However, further reduced models have been used to study reciprocal inhibition of a pair of neurons \cite{ski1,rin2}. The model was also considered in \cite{guc3}, the paper presenting the \textsc{SMST} algorithm. The following parameter values:
\begin{align*}
 \omega = 0.03,\,\gamma = 10,\,r=-4,\,\theta = 0.01333,\,a=1,\,s=1,\,\sigma_1 = 3\quad \mbox{and}\quad \sigma_2 = 1.2652372051,
\end{align*}
also used in \cite{guck4,guc3}, will be used henceforth.

\textbf{Computation of base trajectory}.
\figref{risegB} shows two projections (thick lines) in (a) and (b) of a trajectory segment on the slow manifold, which includes the segment B' in Fig. 6 (c) in \cite{guck4}, which was computed using the modified Runge-Kutta scheme with the discretized \textsc{SO} method based on \corref{etah2}. The time $T$ was set to $0.5$. In forward time the flow is from the lower left to the upper right. Here $\Delta \tau = 10^{-2}$ and $h=10^{-4}<\epsilon^{-1}\Delta \tau^{5/2} =0.01$. {To compute such trajectories using the \textsc{SMST} algorithm it is expected that one has to introduce some sort of continuation to \textit{pull out the transitions at the ends} \cite{langfield2014}.} The segment computed here is much longer than the one in \cite{guck4}. To realize this one can for example compare \figref{risegB} (b) with 
Fig. 6 (c) in \cite{guck4}. It took $0.01$ seconds to compute the 
trajectory in Matlab on an Intel Core i7-3520M 2.90 GHz processor. This time includes the computation of $\phi^\epsilon$ which will be used in the following subsection.
%he critical manifold can easily be written as $q=\eta_0^{-1}(v)$ from which I also obtain $\partial_q \eta_0$.
%As opposed to Example \exref{toy} the vector-field $U$ is now nonlinear in the fast variables, and therefore there is now a real difference between the modified and the traditional \textsc{SO} methods (recall \remref{modSO}). For this case I found it useful to first use the traditional \textsc{SO} method a few times to ensure that the error was below $\epsilon$ and then applied the modified \textsc{SO} method, presented in \propref{modSOprop}, for the consecutive steps. %A pseudo-code is available in \appref{etafunc}. ???? 
% In $\epsilon$-free methods one could replace $\epsilon$ by a suitable norm. %However, this is only done for the computation of $\kappa_1$. For $\
% kappa_
% 2$ I then use the $\eta$-value obtained when computing $\kappa_1$. 
% The result for $\epsilon=10^{-3}$ is shown in \figref{risegB}in terms of projections onto the $(q_1,v_1,v_2)$-space (a) and the $(q_1,v_1)$-plane (b). 
% The thick line is the computed trajectory on the slow manifold. In forward time the flow is from the lower left to the upper right. 
Trajectories, with initial conditions that are displayed from the slow manifold by distances of $\pm 10^{-10}$ along the stable and unstable manifold, are displayed using thinner lines at the tip of this base trajectory. These were obtained from direct integration using Matlab's ode15s with tolerances set to $\text{tol}=10^{-10}$. 

\begin{figure}[h!]
\begin{center}
\subfigure[]{\includegraphics[width=.495\textwidth]{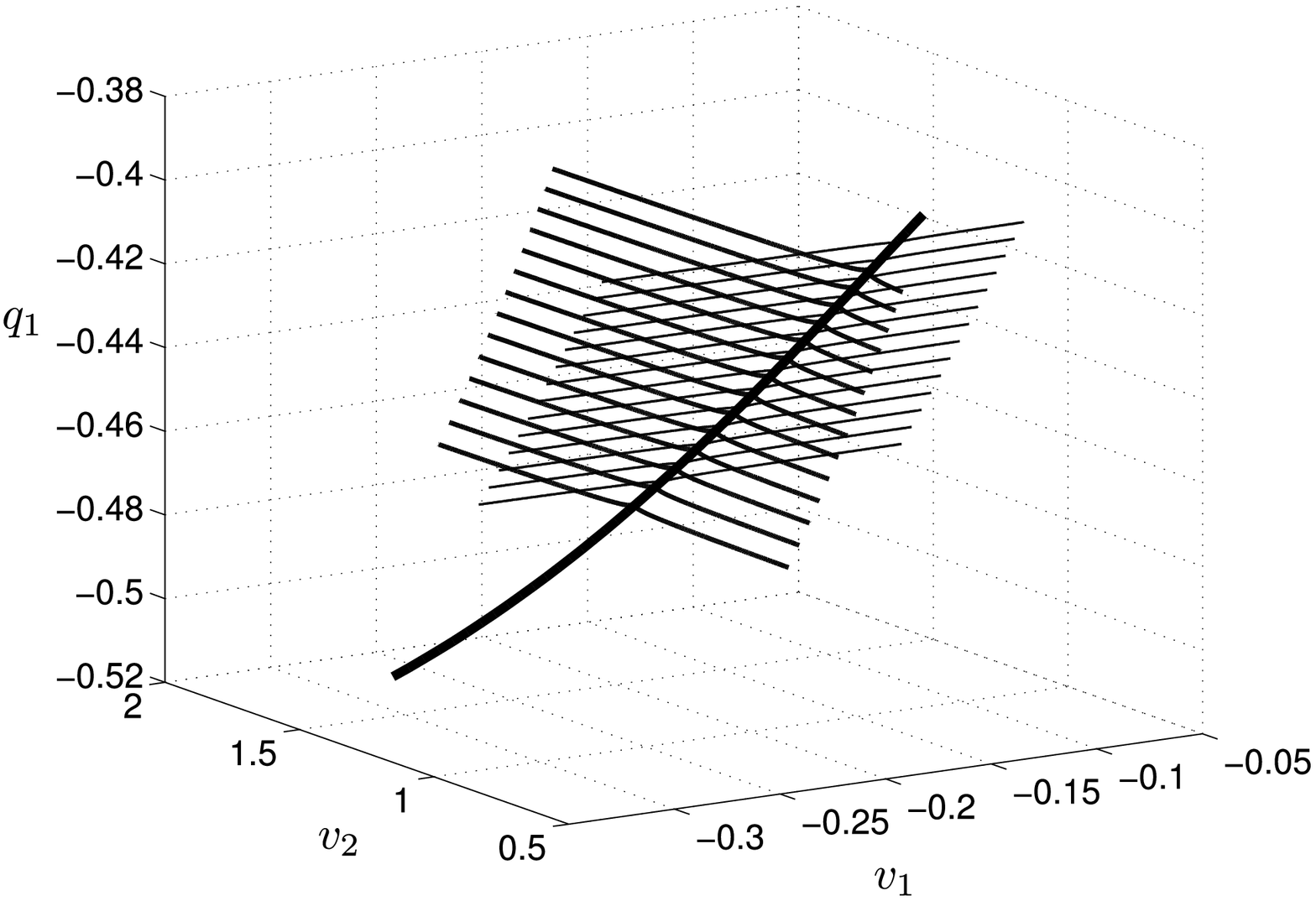}}
\subfigure[]{\includegraphics[width=.475\textwidth]{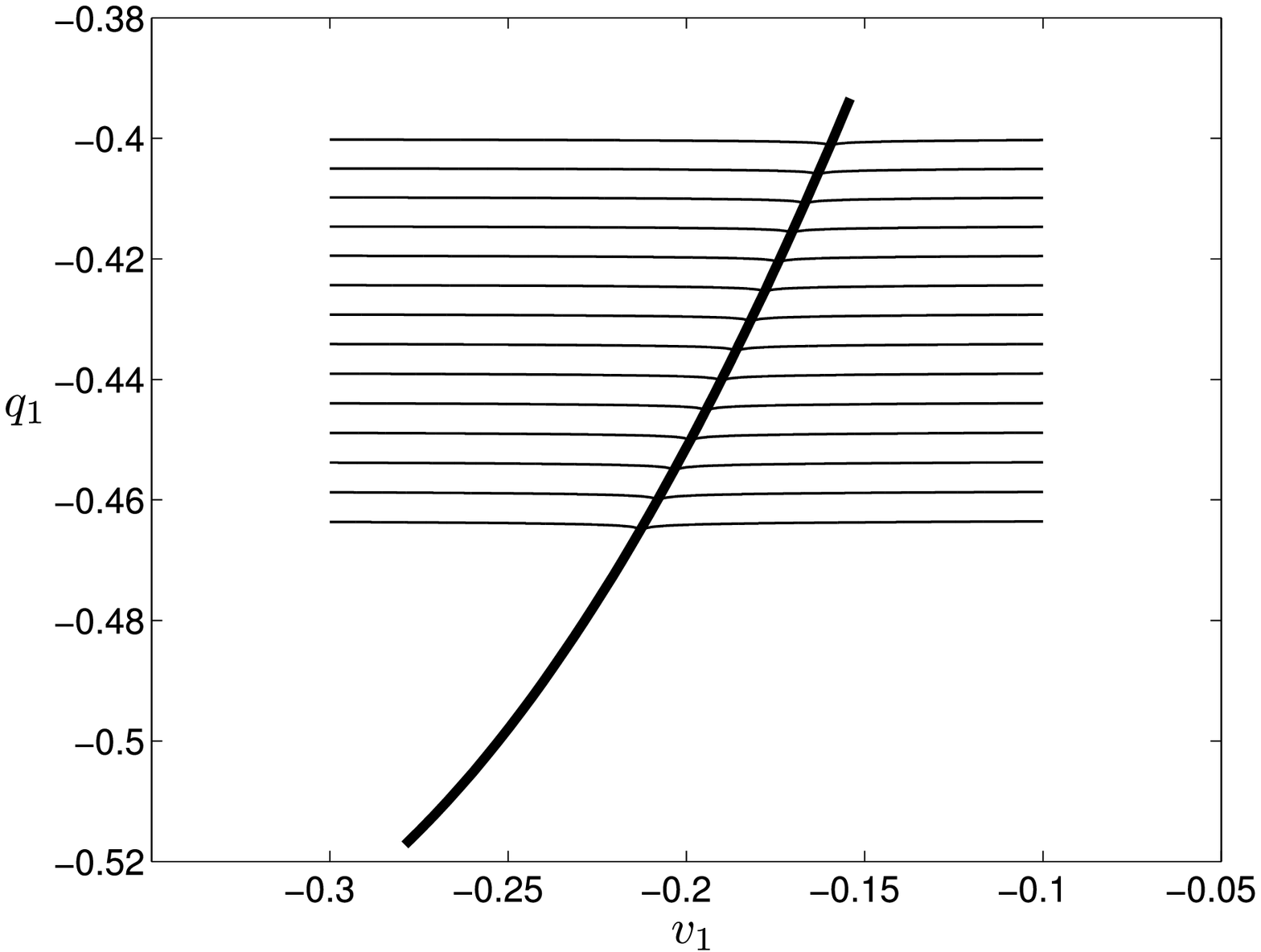}}
\caption{The thick lines in (a) and (b) represent two projections of a trajectory on the slow manifold for the model of reciprocal inhibition. These trajectories were obtained using the modified Runge-Kutta scheme with the \textsc{SO} method based on \corref{etah2}. Trajectories, with initial conditions that are displaced from the slow manifold by distances of $\pm 10^{-10}$ along the stable and unstable manifold, are displayed using thinner lines at the tip of this base trajectory. The thicker of the two sets of thinner lines, going in the $v_2$-direction, corresponds to trajectories on the stable manifold. The set of lines going in the $v_1$ direction corresponds to trajectories on the unstable manifold. These are obtained by direct backward and forward integration respectively. \figlab{risegB}}
\end{center}
% /home/krkri/dtu/AProf/research/matlab/Collocation method/2_slow_2_fast_2404/main files/run_new_eta_ri.m

\end{figure}

\textbf{Computation of transients}. 
Next, trajectories that connect to the trajectory $z=z(\tau)$ in \figref{risegB} (b) near its starting point 
\begin{align*}
(q_{1},q_{2})(0)&= (-0.51723351869,-0.73434299772),\\(v_{1},v_{2})(0) &= (-0.27894449516,1.71095643157),
\end{align*}
and leave it near its end point
\begin{align*}
(q_{1},q_{2})(T) &= (-0.39340933174,0.00310289762),\\(v_{1},v_{2})(T)& = (-0.15410414452,0.72034762953),
\end{align*}
 were computed using the \textsc{SO}-\textsc{SMST} method described in \secref{nonstiff}. 
An example is shown in \figref{risegB2} (a) as a projection onto the $(v_1,v_2,q_1)$-space. The trajectory was obtained using the \textsc{SO}-\textsc{SMST} method with $\Delta t= \Delta \tau = 0.01$. The value of $v_2$ is fixed at $\tau=0$ to $1.71095643157$ while $v_1$ is fixed to be $-0.025410414452$ at $\tau=T$. This gives a distance of $r\approx 0.1$ from the slow manifold at both ends. In (b) this is compared with an accurate reference solution obtained using the \textsc{SMST} algorithm by plotting the Euclidean norm of the difference of the two solutions as a function of time. There is a good agreement between the two trajectories, the maximal error being $\approx 7.5\times 10^{-6}$ at $\tau=T$. This value is also consistent with \propref{est}: Here $\epsilon=10^{-3}$ and $r= 0.1$ so $\epsilon r^2 = 10^{-5}$. The computation of the approximation using the principle in \secref{nonstiff}, which is visualized using the projections in \figref{risegB2} (a) took 
\begin{align}
1.6\,\text{seconds}.\eqlab{time}
\end{align}
The $1.6$ 
seconds include the time required for the propagation of the base trajectory and the time for the collocation on the fast space. \figref{v1v2} displays $v_1=v_1(\tau)$ in (a) and $v_2=v_2(\tau)$ in (b). Here the fast transients are clearly visible. Finally, if the resulting time mesh from the \textsc{SO}-\textsc{SMST} method, $t$-fine at the ends, $\tau$-fine in-between, is used in the \textsc{SMST} collocation method then one obtains an accuracy of $3.0 \times 10^{-8}$ but it took about twice as long ($3.1$ seconds to be precise). If one continues in this way for smaller values of $\epsilon$ while fixing $r=0.1$, computing trajectories using the \textsc{SO}-\textsc{SMST} method, and then using the resulting time mesh in the \textsc{SMST} collocation method. The time used for the collocation method was still about twice as long, but more importantly the \textsc{SMST} method did not converge for smaller values of $\epsilon$ than $\epsilon=5\times 10^{-6}$. The two methods used the same Matlab collocation code. As 
opposed to the considerations in \cite{guc2}, the distance $r$ has been fixed from the slow manifold while decreasing $\epsilon$. It would be interested to perform a more detailed comparison of the two methods in future research.
\begin{figure}[h!]
\begin{center}
\subfigure[]{\includegraphics[width=.5\textwidth]{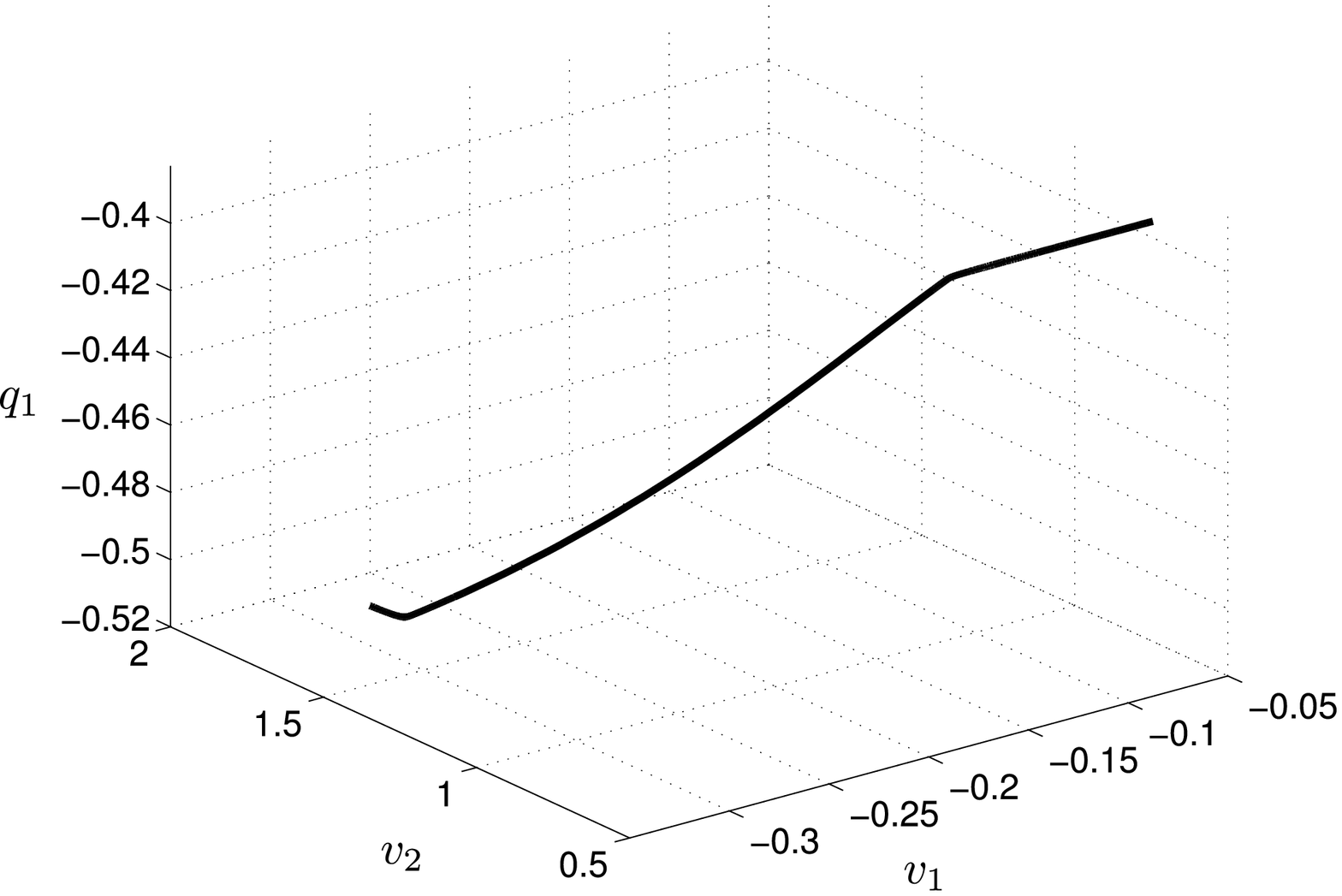}}
\subfigure[]{\includegraphics[width=.475\textwidth]{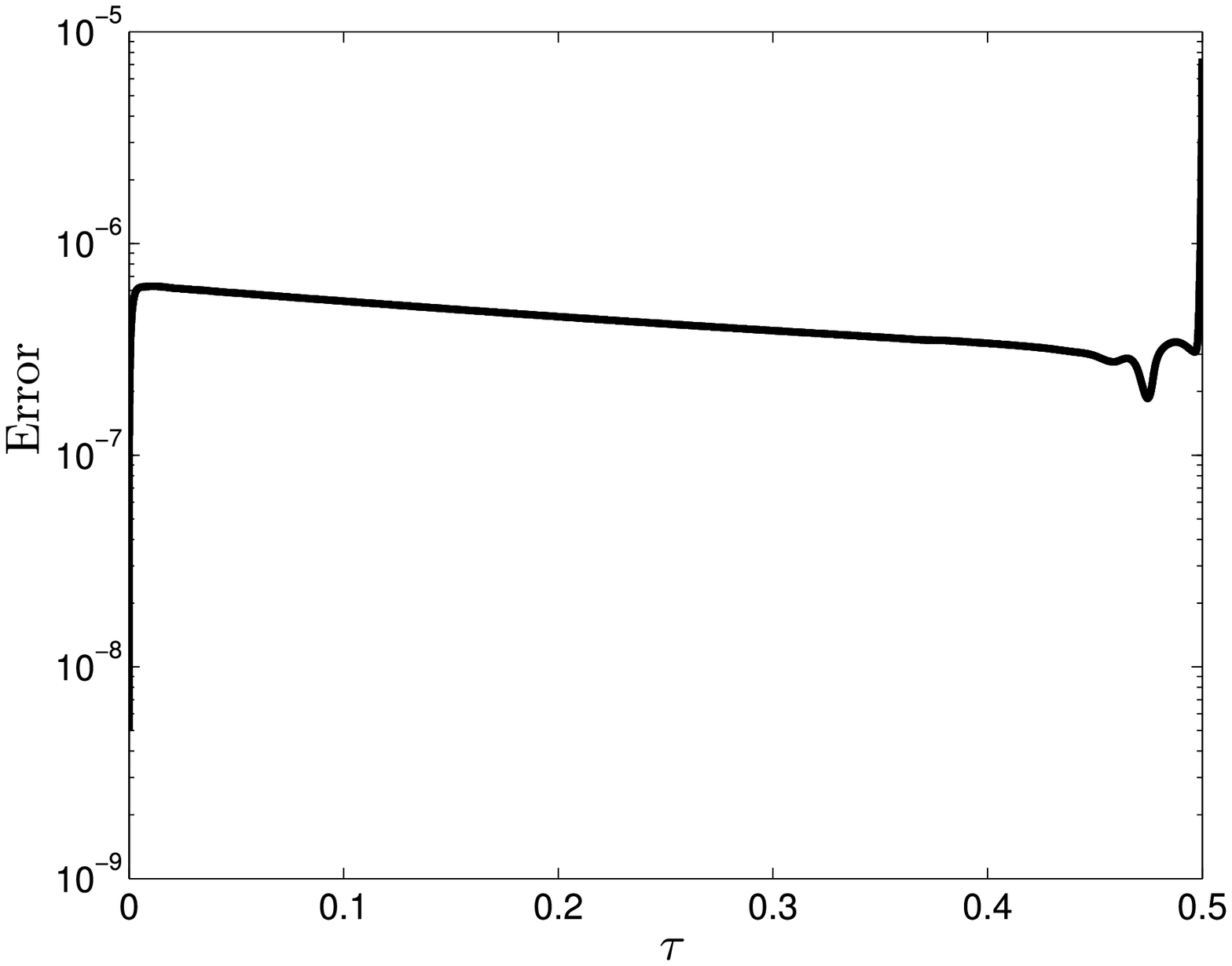}}
\caption{In (a): A trajectory computed using the \textsc{SO-SMST} principle in \secref{nonstiff}. The fast connections are clearly visible. In (b): The accuracy of the trajectory in (a) is analyzed by computing an accurate reference solution using the \textsc{SMST} algorithm. \figlab{risegB2}}
\end{center}
%/home/krkri/dtu/AProf/research/matlab/Collocation method/2_slow_2_fast_2404/main files/run_proj_new_epsilonv.m
\end{figure}

\begin{figure}[h!]
\begin{center}
\subfigure[]{\includegraphics[width=.475\textwidth]{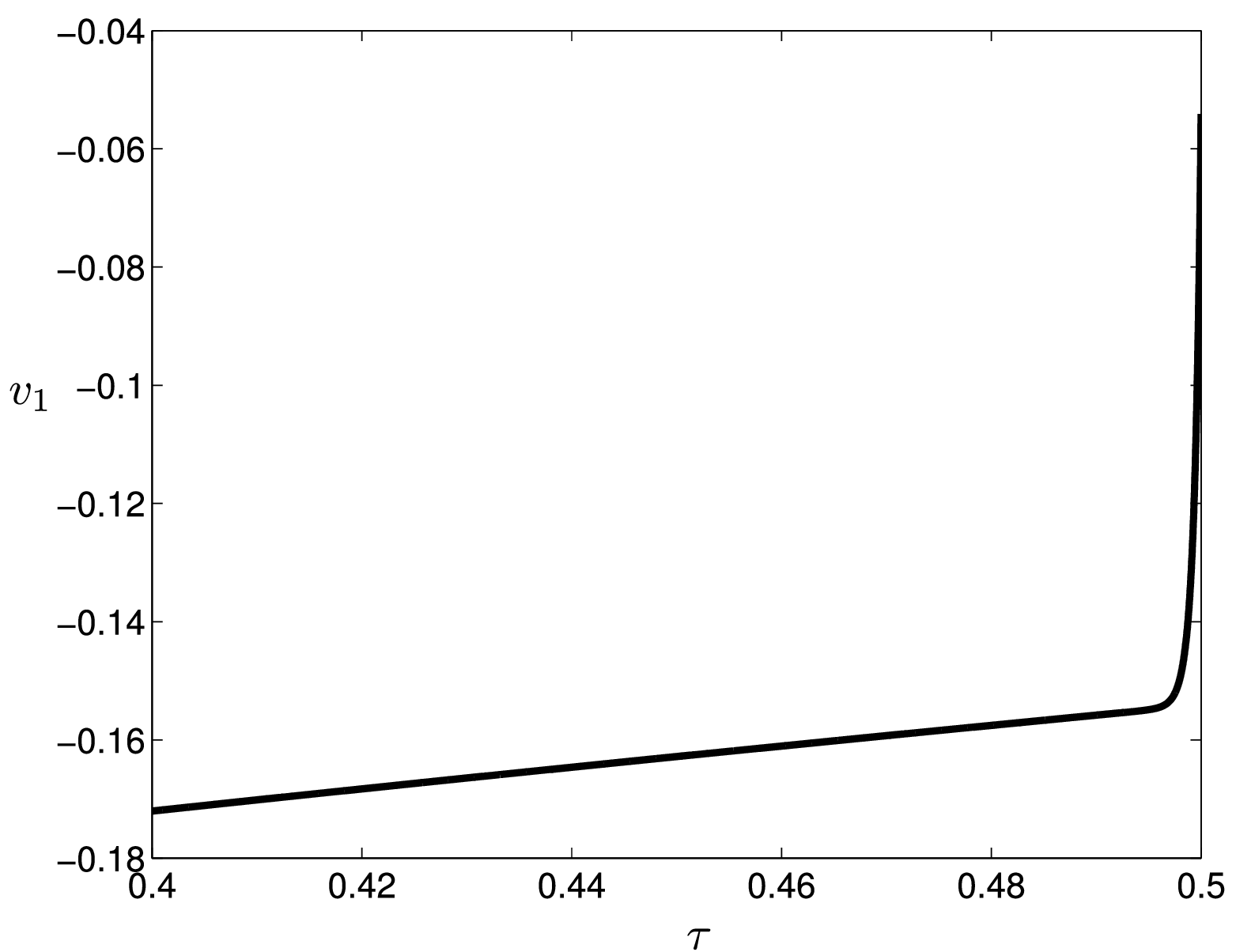}}
\subfigure[]{\includegraphics[width=.475\textwidth]{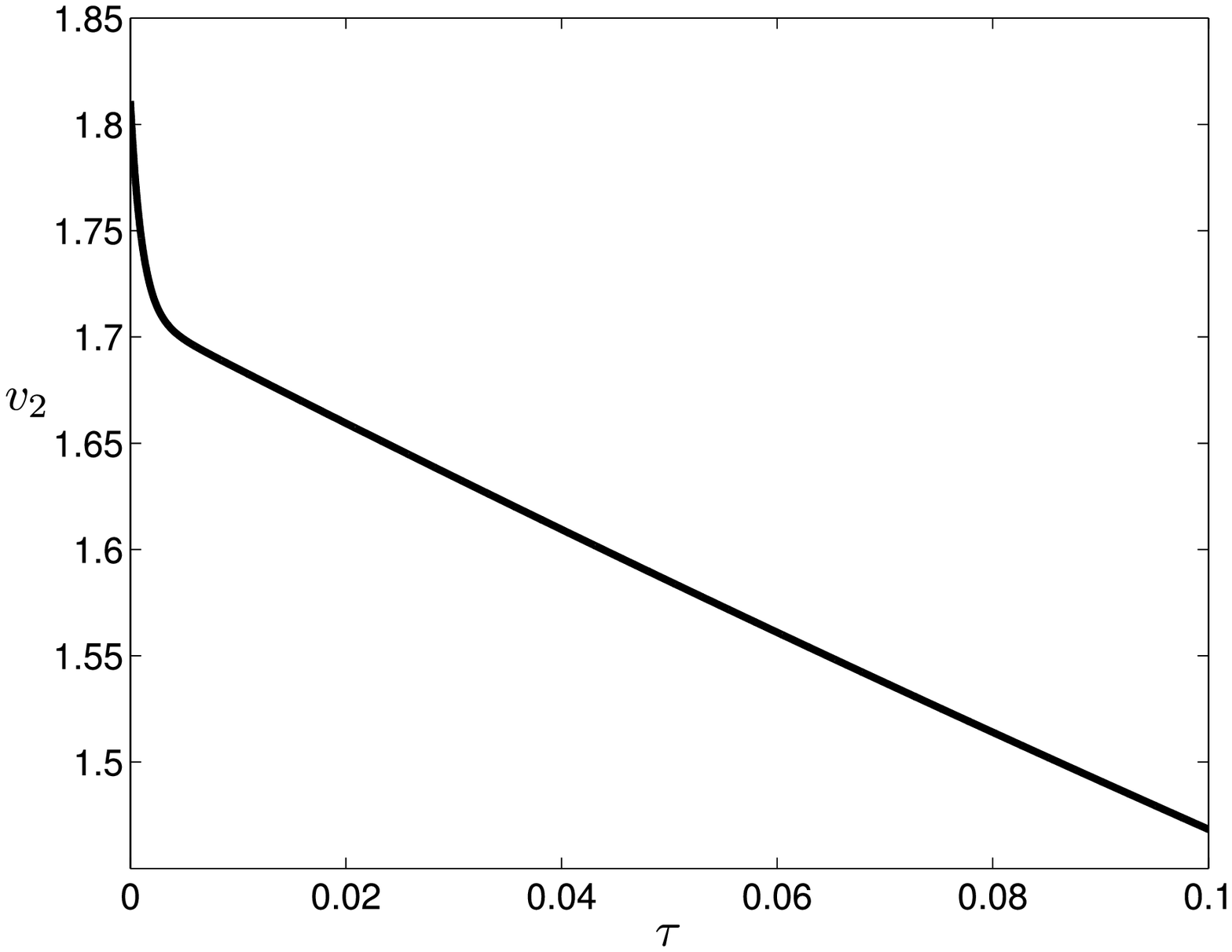}}
\caption{(a) $v_1=v_1(\tau)$ and (b) $v_2=v_2(\tau)$ for the trajectory shown in \figref{fhn} (a). The $v_2$-direction is the stable direction whereas the $v_1$-direction is the unstable direction. The fast connection to the slow manifold is clearly visible in these diagrams. The variable $v_2$ decreases quickly initially whereas $v_1$ grows fast near the end $\tau=0.5$. \figlab{v1v2}}
\end{center}
%/home/krkri/dtu/AProf/research/matlab/Collocation method/2_slow_2_fast_2404/main files/run_proj_new_epsilonv.m
\end{figure}

\figref{exit2} (a) shows a comparison of solutions obtained using the \text{SO-SMST} method with accurate solutions obtained using the \textsc{SMST} algorithm for three different values of $r=0.1,\,0.5$ and $1$. Only the last part of the trajectories are visualized using a projection onto the $(q_1,v_1)$-plane. The thick lines are the \text{SO-SMST} solutions while the thinner ones are those obtained using \textsc{SMST}  method. The error increases with increasing $r$. In (b) the square of $r$ in \eqref{epsr2} is verified by computing the slope $\approx 2$ of the maximal error as a function of $r$ on a logarithm scale. Here the maximal error is understood as the maximum over $\tau\in [0,T]$ of the Euclidean distances between the trajectories.

\begin{figure}[h!]
\begin{center}
\subfigure[]{\includegraphics[width=.51\textwidth]{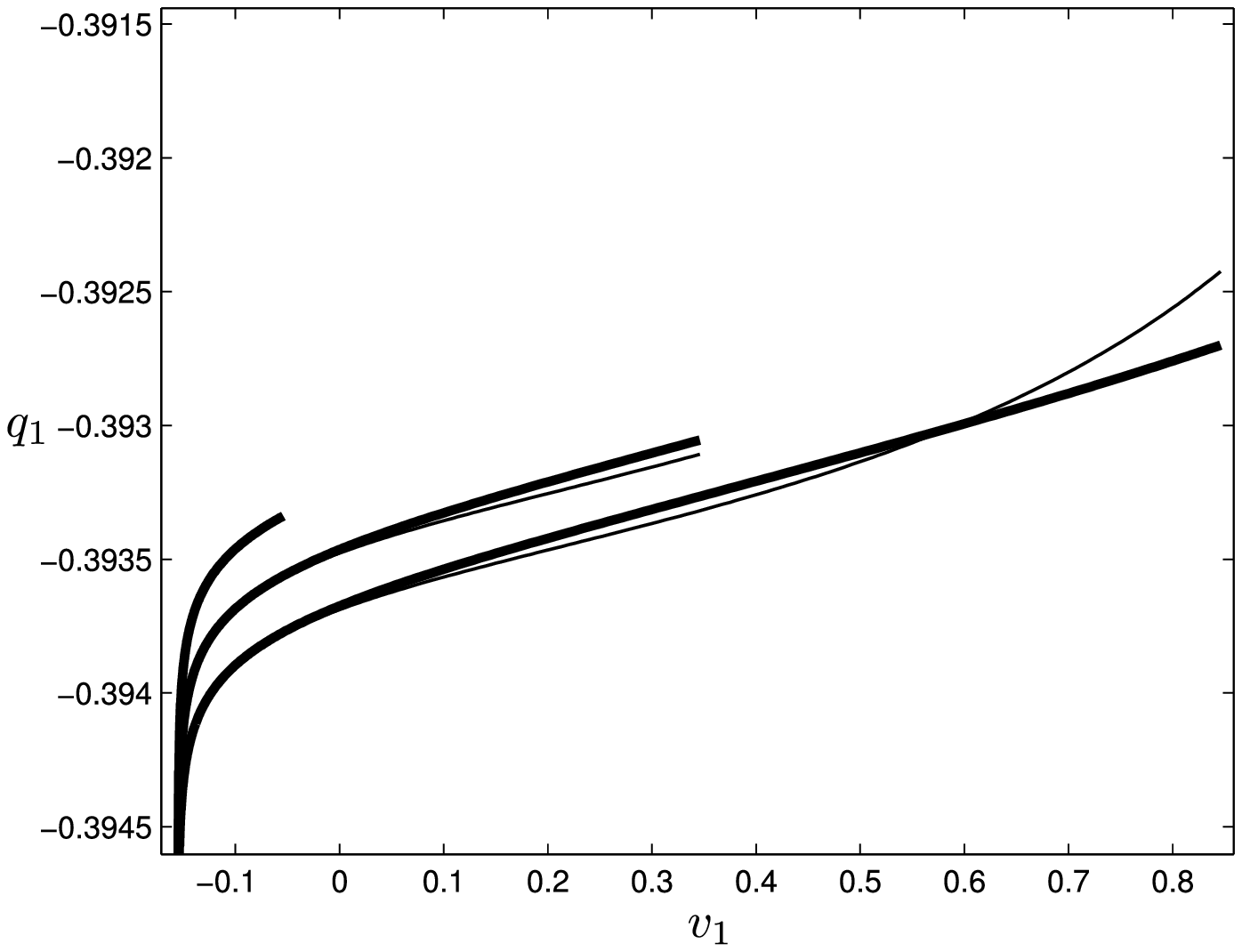}}
\subfigure[]{\includegraphics[width=.466\textwidth]{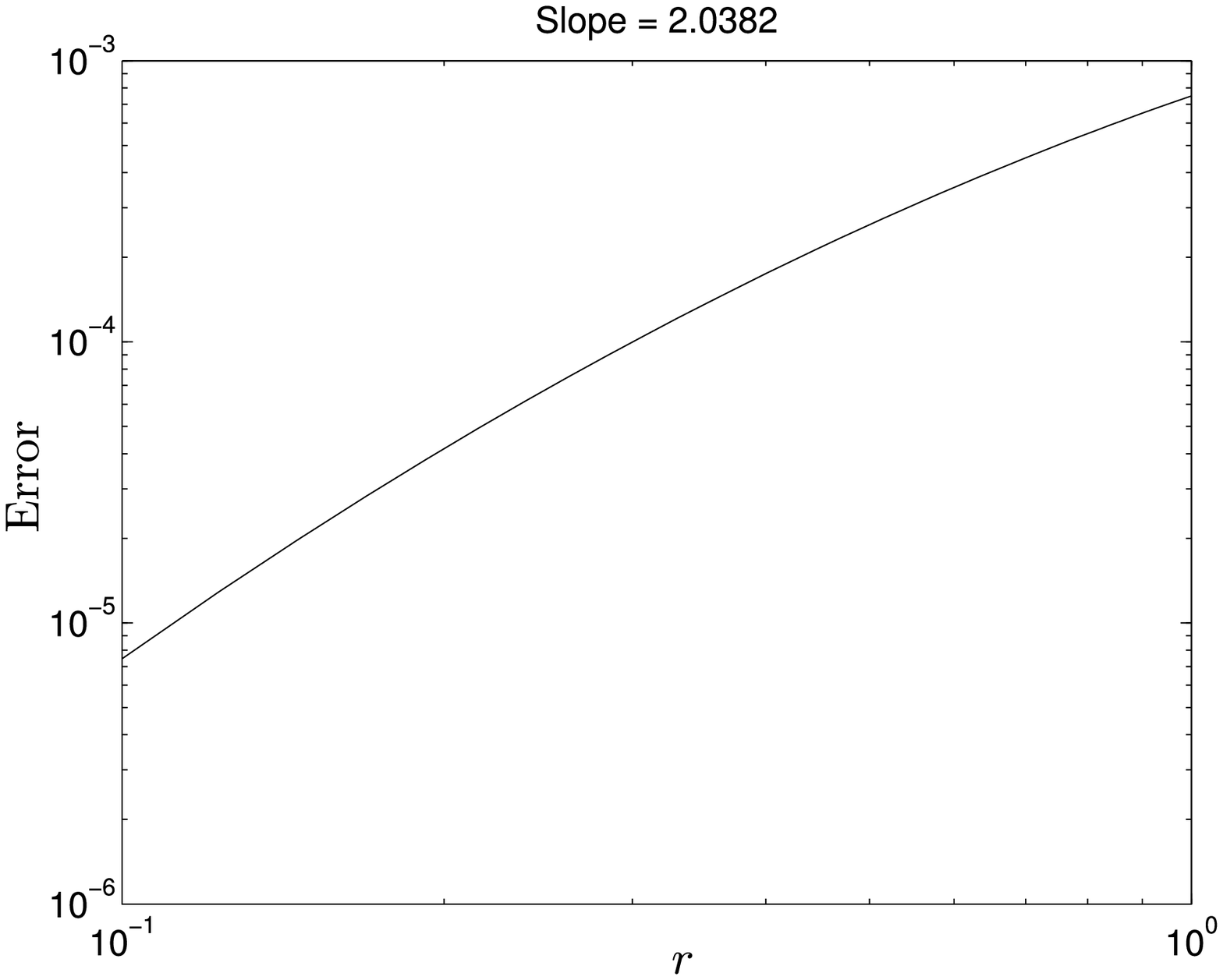}}
\caption{(a): Transients for three different values of $r=0.1,\,0.5$ and $1$. The thick lines are obtained using the \textsc{SO-SMST} principle described in \secref{nonstiff} while the thinner lines are due to the \textsc{SMST} algorithm. (b): The maximal error as a function of $r$. The slope on the logarithm scale is $\approx 2.0$. \figlab{exit2}}
\end{center}
%/home/krkri/dtu/AProf/research/matlab/Collocation method/2_slow_2_fast_2404/main files/run_proj_new_dv2.m
\end{figure}

By applying the \textsc{SO}-\textsc{SMST} procedure, the computation of trajectories near a saddle type slow manifold, has been split into two non-stiff subproblems and as such the singular nature of the original problem has been removed. Therefore no numerical issues appear when $\epsilon$ becomes extremely small. On the contrary, the solution becomes more accurate. \figref{epsv} shows the result of computing similar trajectories to the ones above for extremely small values of $\epsilon$. The distance to the slow manifold has been fixed to $r=0.1$ in both ends. In (a) the exit trajectories are shown in the $(v_1,q_1)$-plane while (b) shows the the time required for different values of $\epsilon$. Notice the $\approx 1.6$ seconds from above \eqref{time} for $\epsilon=10^{-3}$ and furthermore that all the time is taken up by the collocation on the fast space (compare $\diamond$'s with $\circ$'s). The time ($\approx 0.01$ seconds) required for the computation of the base trajectory ($\diamond$'s in \figref{epsv}) is not visible on this scale. 
\begin{figure}[h!]
\begin{center}
\subfigure[]{\includegraphics[width=.492\textwidth]{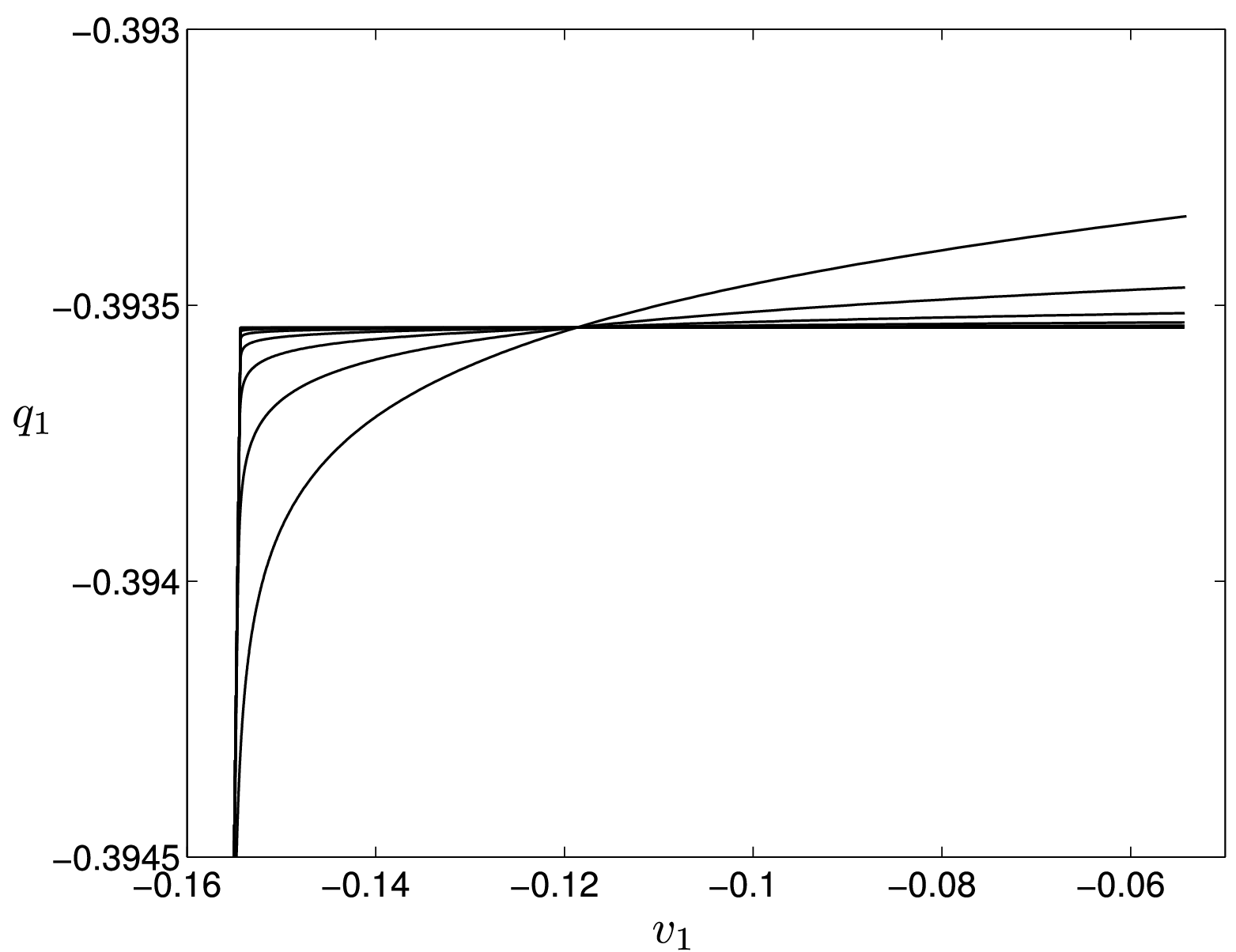}}
\subfigure[]{\includegraphics[width=.475\textwidth]{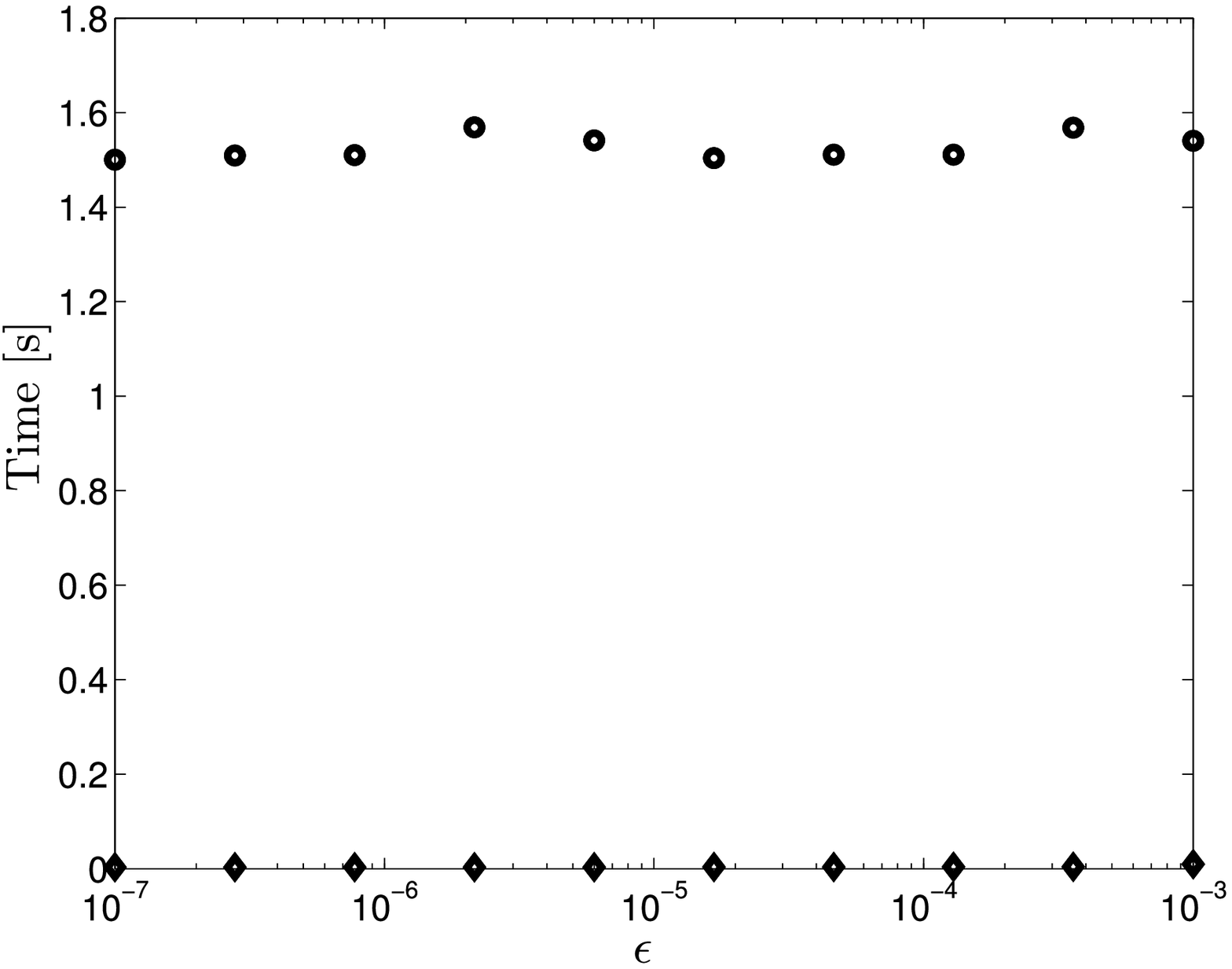}}
\caption{The results of computing canards and their transients by the \textsc{SO}-\textsc{SMST}  method for extremely small values of $\epsilon$. Figure (a) shows the different exit trajectories projected onto the $(v_1,q_1)$-plane. In (b) the time used by the \textsc{SO-SMST} method is illustrated for different values of $\epsilon$. The circles ($\circ$) give the total time, while the diamonds ($\diamond$) give the time required to compute the base trajectory. \figlab{epsv}}
\end{center}
\end{figure}

\subsection{Numerical results for the FitzHugh-Nagumo model}\seclab{fhn}
The FitzHugh-Nagumo model is a PDE model for the membrane potential of a nerve axon which is derived as a simplification of the Hodgin-Huxley model:
\begin{align*}
\partial_t u =\epsilon (v-\gamma u),\quad \partial_t v = d \partial_s^2 v+f_a(v)-u+p, \quad s\in \R^3.
\end{align*}
with $f_a(u) = u(u-a)(1-u)$ and parameters $p,\gamma,\,d$ and $a$. When looking for traveling wave solutions of the form $u(t,s) = x(s+ct)$, $v(t,s)=y_1(s+ct)$, $y_2=y_1'$ 
% and set
% \begin{align*}
%  x(t) &\equiv \tilde u(s+ct),\\
%  y_1(t) &\equiv \tilde v(s+ct),\\
%  y_2(t) &\equiv y_1'(t),\\
%  \end{align*}
one obtains the following finite dimensional slow-fast system
\begin{align*}
 \dot x &=\epsilon (y_1-\gamma x),\\
 \dot y_1 &=y_2,\\
 \dot y_2&=\frac{1}{d} (cy_2-f_a(y_1)+x-p).
\end{align*}
Here $c$ is the wave speed.
Geometric singular perturbation theory has been successfully used to analyze this system, see e.g. \cite{jon1,guc2,kru1} and references therein. In particular, the Exchange Lemma has been applied to prove the existence of homoclinic orbits including both fast and slow segments. Homoclinic orbits correspond, by the traveling wave ansatz, to traveling pulse solutions of the PDEs. Such trajectories will be computed in this section using the \textsc{SO-SMST} method. In this section it will be illustrated how the \textsc{SO-SMST} method can be combined with direct integration for computation of a full orbit.  As in \cite{guc3}, attention is restricted to $a=1/10$ and $d=5$, and $f\equiv f_{1/10}$ for simplicity. 

To explain an example of a homoclinic orbit it is first pointed out that the critical manifold is one-dimensional and of the form
\begin{align*}
 M_0 = \{y_2=0,\,x=f(y_1)+p\}=M_0^l\cup \{z_{lm}\}\cup M_0^m \cup \{z_{mr}\}\cup M_0^r.
\end{align*}
It has three different normal hyperbolic components $M_0^l$, $M_0^m$ and $M_0^r$ that are separated by two fold points $z_{lm}\approx (-0.0024+p,0.049,0)$ and $z_{mr}\approx (0.13+p,0.68,0)$. These objects are all contained within the plane $y_2=0$. An example for $p=0$ is shown in \figref{M0} (a). Both $M_0^l$ and $M_0^r$ are of saddle-type whereas $M_0^m$ is repelling. For $\epsilon$ sufficiently small Fenichel's theory imply that $M_0^l\backslash B_\rho(z_{lm})$, $M_0^m\backslash \{B_\rho(z_{lm})\cup B_\rho(z_{mr})\}$ and $M_0^r\backslash B_\rho(z_{mr})$ all perturb to some $M^l$, $M^m$ and $M^r$. Small neighborhoods $B_\rho(z_{lm})$ and $B_\rho (z_{rm})$ of the fold points $z_{lm}$ and $z_{mr}$, respectively, have been removed from have been removed from $M_0^l$ and $M_+^r$ since normal hyperbolicity is violated there. 

For $p=0$ the point $q=(0,0,0)$ is the unique equilibrium and the results of e.g. \cite{jon1, kru1} show that for $\epsilon$ sufficiently small there exists $c_*$, $z_r^{\text{approach}}$, $z_r^{\text{exit}}$ and $z_l^{\text{approach}}$ so that for $c=c_*$ there is a 
homoclinic 
connection to $q$ composed of four segments: 
\begin{itemize}
\item[(i)] a fast segment along the strong unstable manifold of $q$ connecting to $M^r$ close to $z_r^{\text{approach}}\in M^r$; 
\item[(ii)] a slow segment on $M^r$ initiated near $z_r^{\text{approach}}$ and terminated near $z_r^{\text{exit}}\in M^r$; 
\item[(iii)] a fast segment leaving $M^r$ near $z_r^{\text{exit}}$ and approaching $M^l$ near $z_l^{\text{approach}}\in M^l$; 
\item[(iv)] a slow segment on $M^l$ initiated near $z_l^{\text{approach}}$ and eventually terminating at $q=(0,0,0)$. 
\end{itemize}
This orbit is obtained by transversality (using the Exchange Lemma and Fenichel's theory) from a singular orbit whose projection onto the $(x,y_1)$-plane is shown in \figref{M0} (b). 
\begin{figure}[h!]
\begin{center}
\subfigure[]{\includegraphics[width=.475\textwidth]{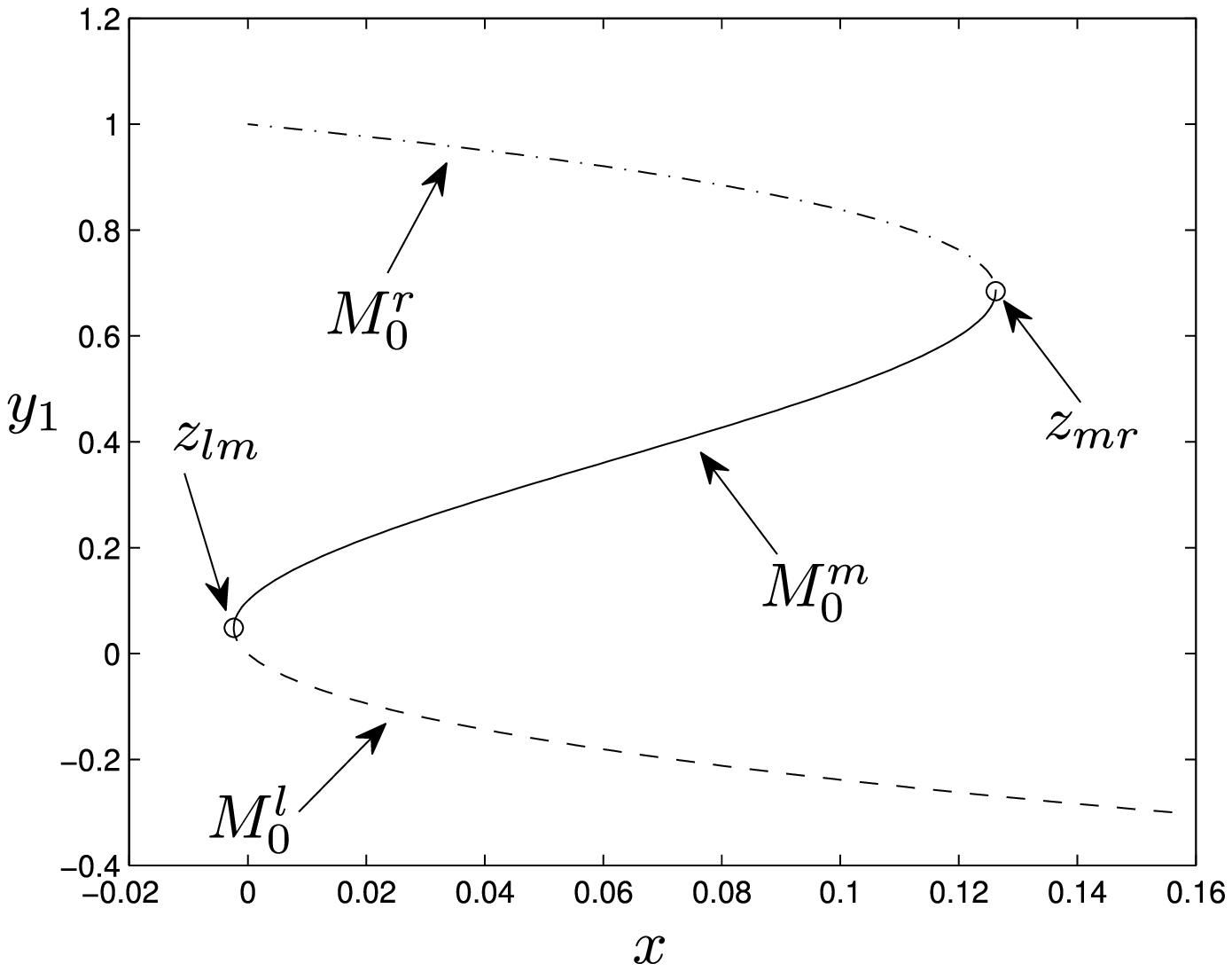}}
\subfigure[]{\includegraphics[width=.475\textwidth]{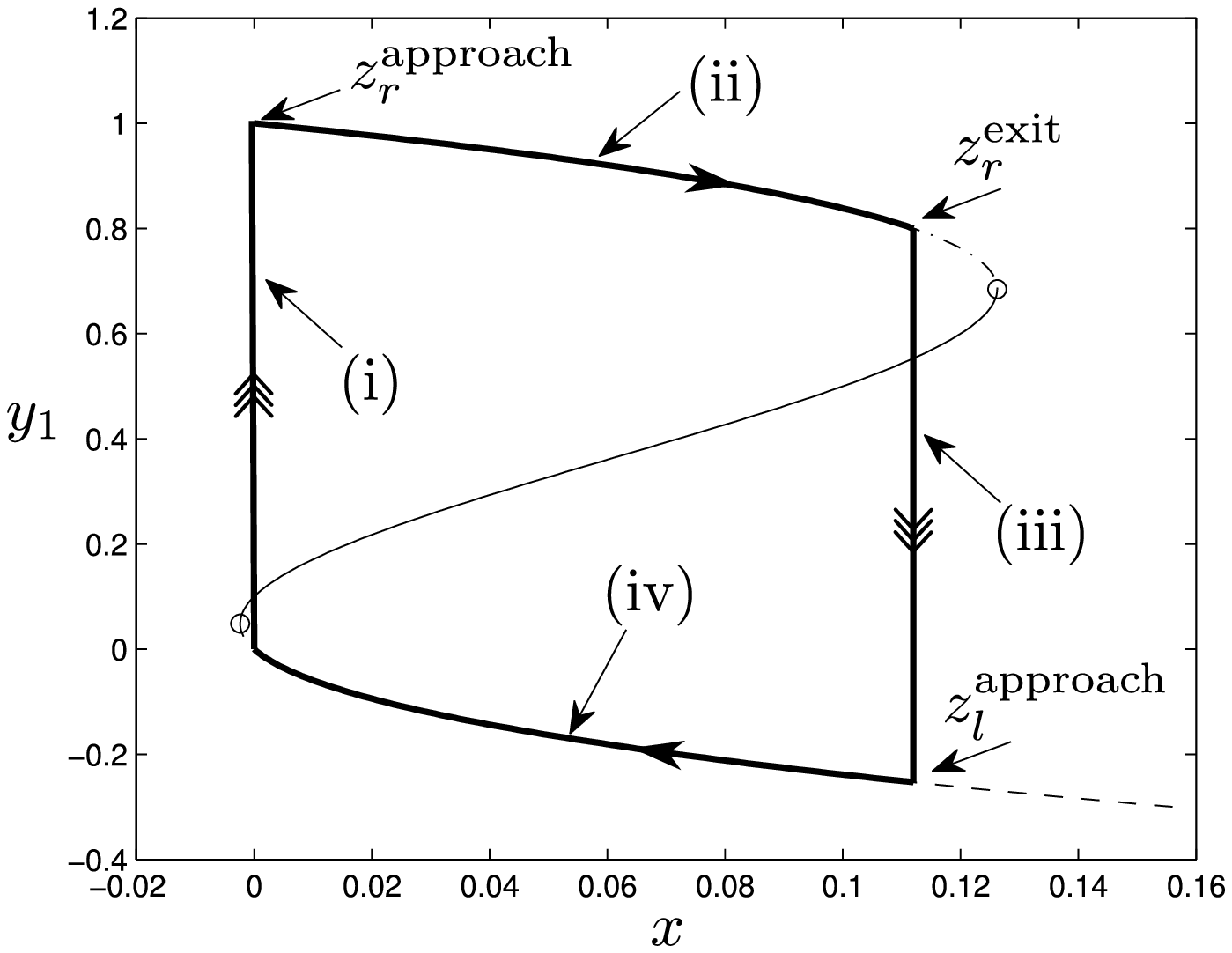}}
\caption{(a): The critical manifold $M_0\subset \{(x,y_1,y_2)\vert y_2=0\}$, and its three hyperbolic components $M_0^l$, $M_0^m$ and $M_0^r$ for $p=0$, within the plane $(x,y_1)$. The points $z_{lm}$ and $z_{rm}$ are fold points. (b): The singular homoclinic orbit is contained within the $(x,y_1)$-plane and composed of segments (i), (ii), (iii) and (iv). A true homoclinic connection can be established for $\epsilon>0$ but small by transversality from this singular homoclinic orbit. \figlab{M0}}
\end{center}
\end{figure}

We consider $p=0$ and compute the homoclinic connection to the equilibrium at $(0,0,0)$ as follows:
\begin{itemize}
 \item[$1^\circ$] Determination of $c$ and the strong unstable manifold of the equilibrium $(0,0,0)$: As in \cite{guc2} it is used that the stable manifold $W^s(M^r)$ acts a separatrix in phase space. The resulting trajectory is terminated at $y_2=0$. This fast segment is denoted by $\gamma_1$.
 \item[$2^\circ$] Computation of $M^r$ and the connection of $\gamma_1$ to $M^r$: Using the function $\phi^\epsilon$, the end-point of $\gamma_1$ is projected onto $M^r$ using \eqref{xx0invert} neglecting terms of order $\mathcal O(\epsilon^2 y_0^2)$, with $y_0$ measuring the deviation from $M^r$. The modified Runge-Kutta scheme is then used to compute $M^r=\{y=\eta_r(x)\}$. The connection from the end of $\gamma_1$ to $M^r$ is computed using the \textsc{SO-SMST} algorithm. This slow segment is denoted by $\gamma_2$.
 \item[$3^\circ$] Computation of $M^l$: It is obtained as a graph $y=\eta_l(x)$ by using the modified Runge-Kutta scheme in backwards integration of $q=(0,0,0)$.% to obtain $M^l=\{y=\eta_l(x)\}$.
 \item[$4^\circ$] Computation of $W^u(M^r)\cap W^s(M^l)$: For this the Newton's method is used to obtain a root of the function:
 $$F(x_b^r,x_b^l)=(x^r,y_2^r)(x_b^r)-(x^l,y_2^l)(x_b^l),$$
 $(x^r,y_2^r)$ being the intersection of a trajectory on $W^u(M^r)$, obtained by forward integration, that was initiated at a point that was displayed from $(x_b^r,\eta_r(x_b^r))$ on $M^r$ by an amount of $10^{-6}$ along the unstable direction, with the plane $y_1=1/2$. Similarly $(x^l,y_2^l)$ is the intersection of a trajectory on $W^s(M^l)$, obtained by backward integration, that was initiated at a point that was displayed from $(x_b^l,\eta_l(x_b^l))$ on $M^l$ by an amount of $10^{-6}$ along the stable direction, with the plane $y_1=1/2$. The Jacobian is computed through the variational equations. The derivatives $\partial_x \eta_r$ and $\partial_x \eta_l$ are obtained from the \textsc{SO} method. The resulting trajectory segment is denoted by $\gamma_3$. It connects $\gamma_2$ from the point of departure $(x_b^r,\eta(x_b^r))$ with $M^l$ through the entrance $(x_b^l,\eta_l(x_b^l))$. %It includes a slow segment on $M^r$ the end of $\gamma_2$ to the point of departure $(x_b^r,\eta(x_b^r))$. 
 \item[$5^\circ$] The final slow segment $\gamma_4$ is taken from $M^l$ from the entrance point $(x_b^l,\eta(x_b^l))$ to $(0,0,0)$.
\end{itemize}
The union of the segments $\gamma_1, \gamma_2, \gamma_3,$ and $\gamma_4$ forms a homoclinic orbit. The result is shown in \figref{fhn}. From here it is also clear that the homoclinic has segments near the end of segment (iv) and near the end of segment (ii) that are relatively close to the fold points $z_{lm}$ and $z_{mr}$. 
%\figref{fhn}. 
We obtain $c=1.2462875$ for $p=0$.  
%\figref{fhnxy1_closeup} 
The result of step $2^\circ$ for $\epsilon=10^{-3}$ is shown in \figref{fhnxy1_closeup} using a close-up. There is an error in the connection with $\gamma_1$ to $\gamma_2$ due to \eqref{xx0invert} that is not visible in this diagram. It is too small: $5\times 10^{-9}$.

A simpler alternative to the projection method used here, that is based on the determination of the function $\phi^\epsilon$, would be to use the ``naive'' fiber projection: $(x,y)\mapsto (x,\eta(x))$. See also \eqref{Naive}. In general this projection is $\mathcal O(\epsilon r)$-close to the correct one. The number $r$ again measures the deviation from the slow manifold. If this naive projection is applied here then one obtains a slightly larger error of $2\times 10^{-6}$ in the connection. There is an improvement by factor of $10^{-3}$ using the more accurate \textsc{SOF} projection without any detectable increase in time.

\begin{figure}[h!]
\begin{center}
\subfigure[]{\includegraphics[width=.475\textwidth]{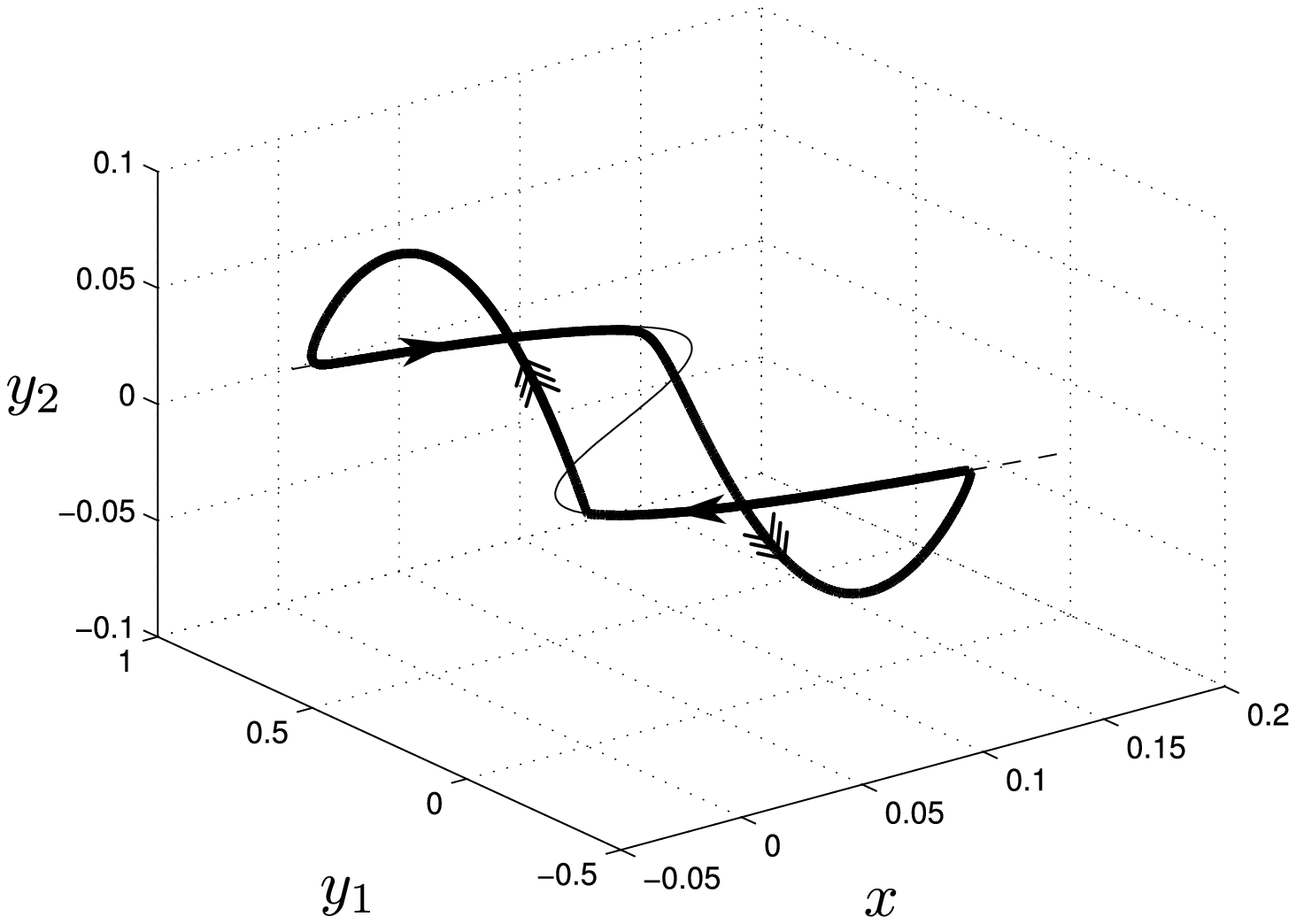}}
\subfigure[]{\includegraphics[width=.475\textwidth]{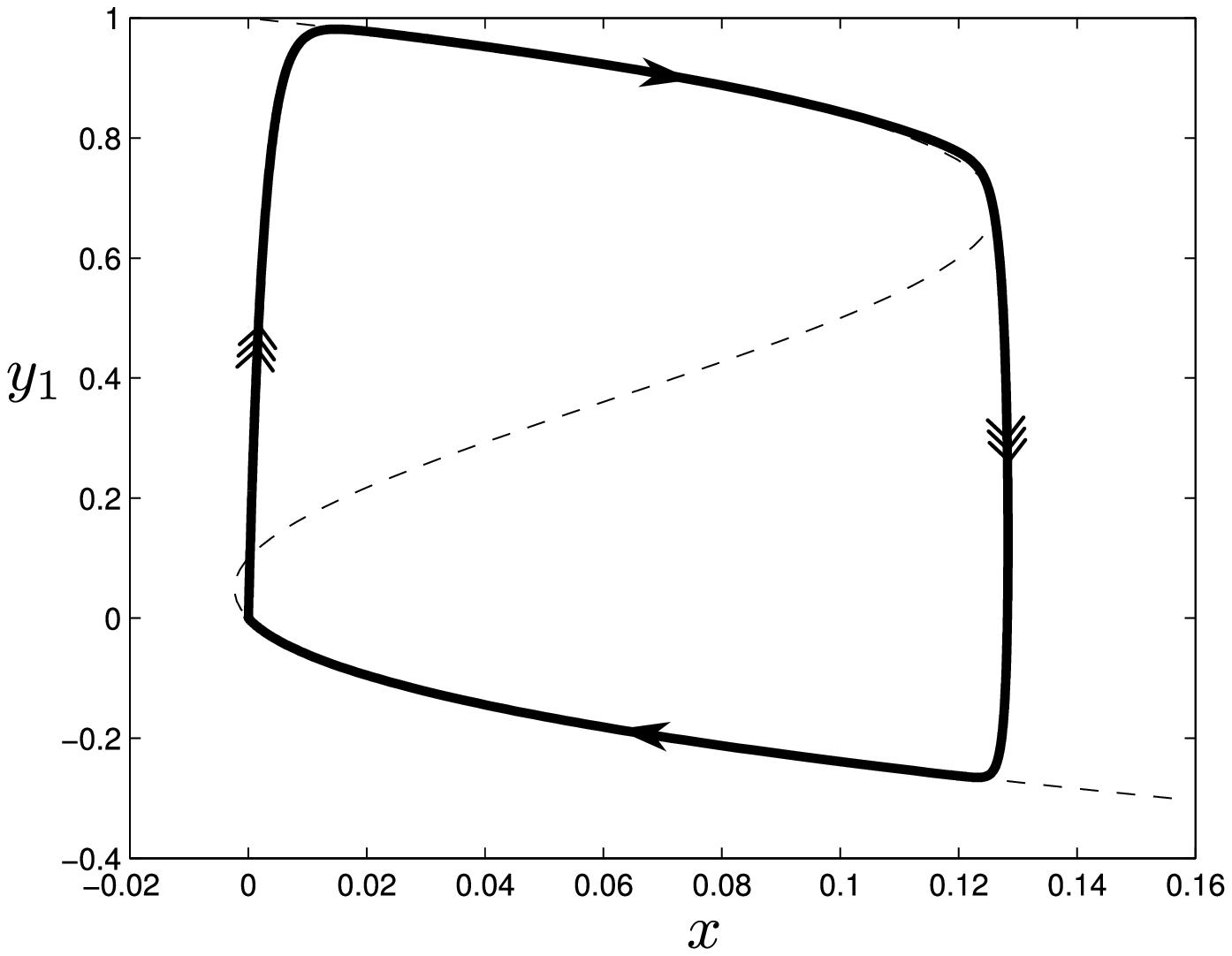}}
\caption{The homoclinic orbit for $p=0$, $c=1.2462875$ and $\epsilon=10^{-3}$ computed using the method described by $1^\circ,2^\circ,\ldots,5^\circ$.  Figure (b) shows the projection of the homoclinic onto the $(x,y_1)$-plane. The thin dotted lines indicate the critical manifold. \figlab{fhn}}
\end{center}
\end{figure}

\begin{figure}[h!]
\begin{center}
\subfigure[]{\includegraphics[width=.675\textwidth]{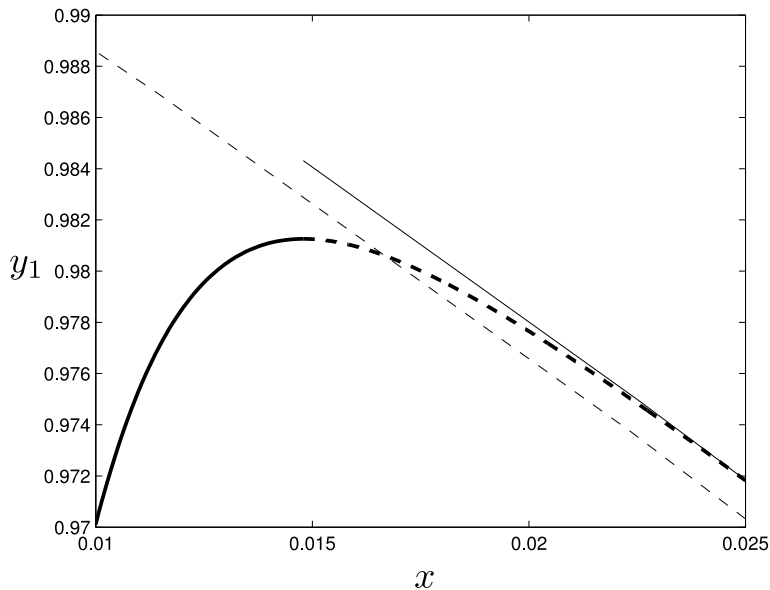}}
\caption{The result of the step $2^\circ$ for $\epsilon=10^{-3}$ and $p=0$. The full thick line shows the incoming trajectory $\gamma_1$ from step $1^\circ$. The thin line shows the base trajectory on $M^r$ obtained by projection through $\phi$. The thick dotted line is $\gamma_2$, the result of applying the \textsc{SO-SMST} algorithm to connect $\gamma_1$ to $M^r$. There is a discrepancy at the connection of $\gamma_1$ with $\gamma_2$ due to the errors associated with the projection described in \eqref{xx0invert}. It is however small and not visible being only $5\times 10^{-9}$. \figlab{fhnxy1_closeup}}
\end{center}
\end{figure}

\subsection{The Lindemann mechanism: An example not in the canonical slow-fast form}\seclab{Lindemann}
In this section we finally consider the Lindemann mechanism 
\begin{eqnarray}
 \dot x &=&X^\epsilon (x,y)=-x(x-y),\eqlab{lm}\\
 \dot y &=&Y(x,y)=x(x-y)-\epsilon y,\nonumber
\end{eqnarray}
also considered in \cite{gou1,kri4}. Here $x,\,y\ge 0$. It is an example of a slow-fast system where the slow and fast variables have not been properly identified and it is used as a caricature of an $\epsilon$-free system. Setting $(w,z)=(x+y,2y)$ gives a system in the canonical slow-fast form: 
\begin{eqnarray}
 \dot w &=&\epsilon W(w,z)=- \frac12 \epsilon z,\eqlab{wzeqn}\\
 \dot z &=& Z(w,z)=2w^2-(3w+\epsilon)z+z^2,\nonumber%(2w- z) (w-z) -\epsilon z,\nonumber
\end{eqnarray}
The graph $z=w$, which corresponds to $y=x$ in the original variables, is then an normally attracting critical manifold. Using the original variables in \eqref{lm} it is easy to realise the existence of a unique equilibrium at $(x,y)=0$. This equilibrium is non-hyperbolic even for $\epsilon>0$: the eigenvalues are $0$ and $-\epsilon$. In \cite{calder2011properties} it is, nevertheless, shown that the origin attracts all of the first quadrant $x,\,y\ge 0$ for all $\epsilon>0$. %The reference \cite{calder2011properties} does not use slow-fast theory to show this result and it applies for all $\epsilon>0$. 
% It should however be a relatively easy task to establish this using the blow-up method based on the following scaling $x=r_2 x_2,\,y=r_2 y_2,\,\epsilon = r_2$.

In \cite{kri4} it was shown that the two iterative methods, \textsc{SO} and \textsc{SOF}, are both successful in approximating the slow manifold and the tangent spaces of the fibers. What proves crucial to this, is 
\begin{itemize}
\item[(a)] \textsc{SO} and \textsc{SOF} makes no explicit reference to $\epsilon$. These methods only involve the vector-fields $X^\epsilon$ and $Y$ and their first partial derivatives. 
\item[(b)] The variable $x$ can still parametrize the critical manifold. As opposed to \eqref{fs}, where normally hyperbolicity always implies that the critical manifold can be written as graph over the slow variables, this does not need to hold true if the slow and fast variables have not been properly identified. %Note that I still write $\epsilon X$ in \eqref{lm} although it is not always small. This is just to further point to the fact that $x$ is treated as slow.
\end{itemize}
It was demonstrated in \cite{kri4} that \textsc{SO} and \textsc{SOF} performed better than the alternative CSP method when applied to \eqref{lm}. In particular, Fig. 10 in \cite{kri4} shows that $n$ applications of the CSP method and the \textsc{SOF} method give approximations of the tangent spaces of the fibers accurate to order $\mathcal O(\epsilon^{n-1})$ and $\mathcal O(\epsilon^{n+1})$, respectively. 

% Although \eqref{lm} is not an example with a saddle-type slow manifold, I will still demonstrate that the 
The system \eqref{lm} is not an example with a saddle-type slow manifold. Nevertheless, the \textsc{SO-SMST} method will still be applied in order to demonstrate its use on $\epsilon$-free systems. \figref{LindemannOrder} shows a comparison of accurate closed-form solutions for $\eta=\eta(x)$ and $\phi^\epsilon=\phi^\epsilon(x)$ obtained using Maple with solutions $\eta^{h}$ and $\phi^{\epsilon,h}$ obtained using the discretized iterative methods in \corref{etah2} and \propref{sofmod}. The comparison was made for $x=1$, grid size $h=\epsilon$ and varying values of $\epsilon$. The operator $\delta_x^h$ was again based on classical second order finite differences. The errors are seen to give approximately straight lines in the log-log scale. The slopes were $\approx 5$ and $\approx 4$ for the approximations of $\eta$ and $\phi^\epsilon$, respectively. Since $p=2$ one would expect from \eqref{dsomr}$_{h=\epsilon}$ a slope of $\approx 4$ for the determination of $\eta$. The improved slope of $\approx 5$ is due to the fact that the $\mathcal O(\epsilon)$-term in the asymptotic expansion for $\eta$:
\begin{align*}
 \eta(x) = x-\frac12 \epsilon+\mathcal O(\epsilon^2),
\end{align*}
is constant. The error $(\partial_x -\delta_x^h) (\eta^h-\eta_0)$ is therefore $\mathcal O(\epsilon^2)$ and the error in \eqref{dsomr}$_{h=\epsilon}$ should in this case with $p=2$ be $\mathcal O(\epsilon^3 h^2)$, ignoring the exponentially small terms. Moreover, the order of $\approx 4$ for the determination of $\phi^\epsilon$ is not in agreement with \eqref{errorPhinh}$_{h=\epsilon}$ since 
\begin{align}
\phi^\epsilon=-\frac12 +\mathcal O(\epsilon),\eqlab{phiEpsLindemann}
\end{align}
is not small. See also \remref{errorPhinhRemark}. However, since the zeroth order term in the expansion of $\phi^\epsilon$ in \eqref{phiEpsLindemann} is constant, the error from replacing $\partial_x$ by $\delta_x^h$ is therefore $(\partial_x-\delta_x^h)\phi^{\epsilon,h}=\mathcal O(\epsilon)$ and the total error is therefore $\mathcal O(\epsilon^2 h^2)$, again ignoring the exponentially small terms, which is in agreement with \figref{LindemannOrder} (b).

\begin{figure}[h!]
\begin{center}
\subfigure[]{\includegraphics[width=.475\textwidth]{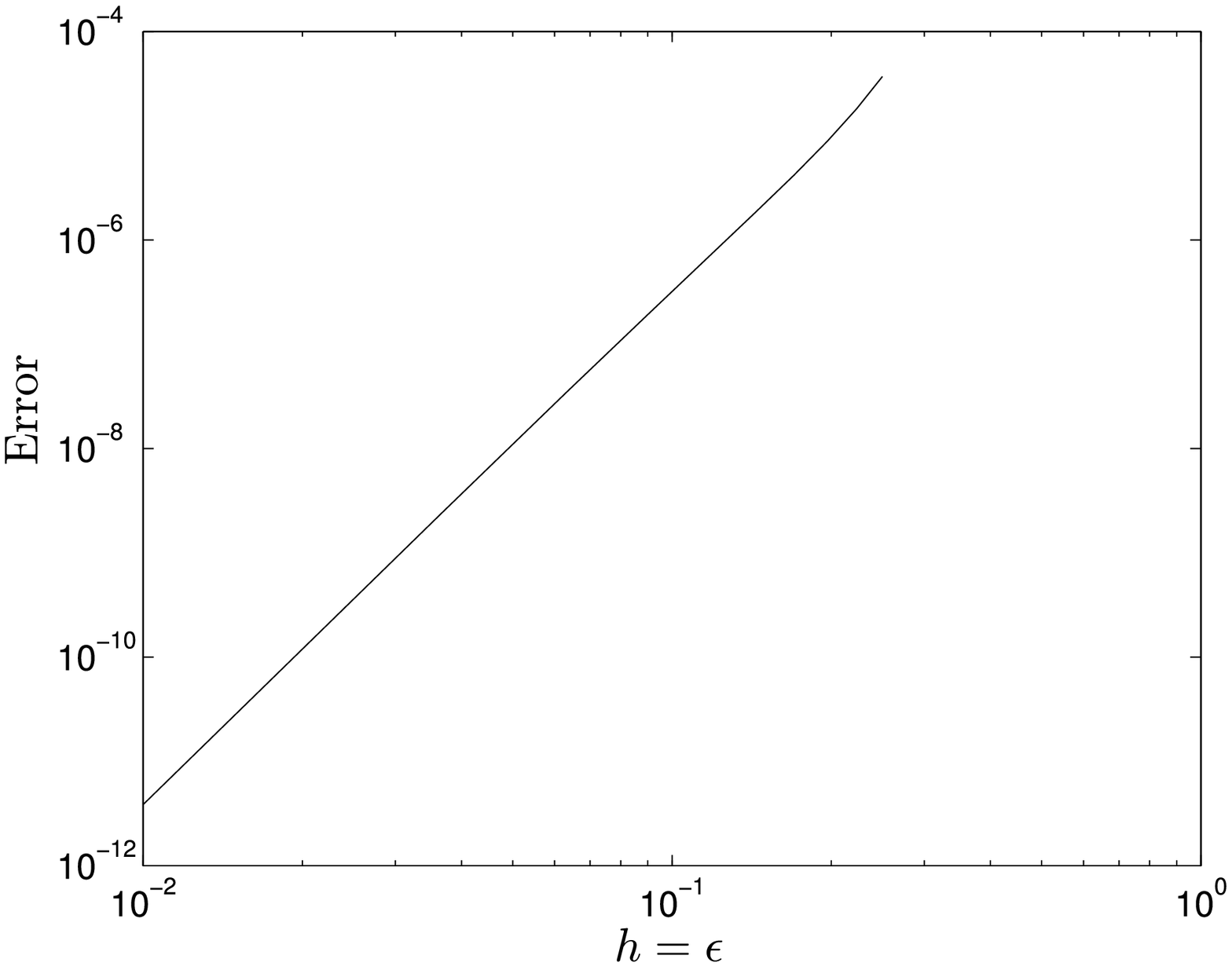}}
\subfigure[]{\includegraphics[width=.475\textwidth]{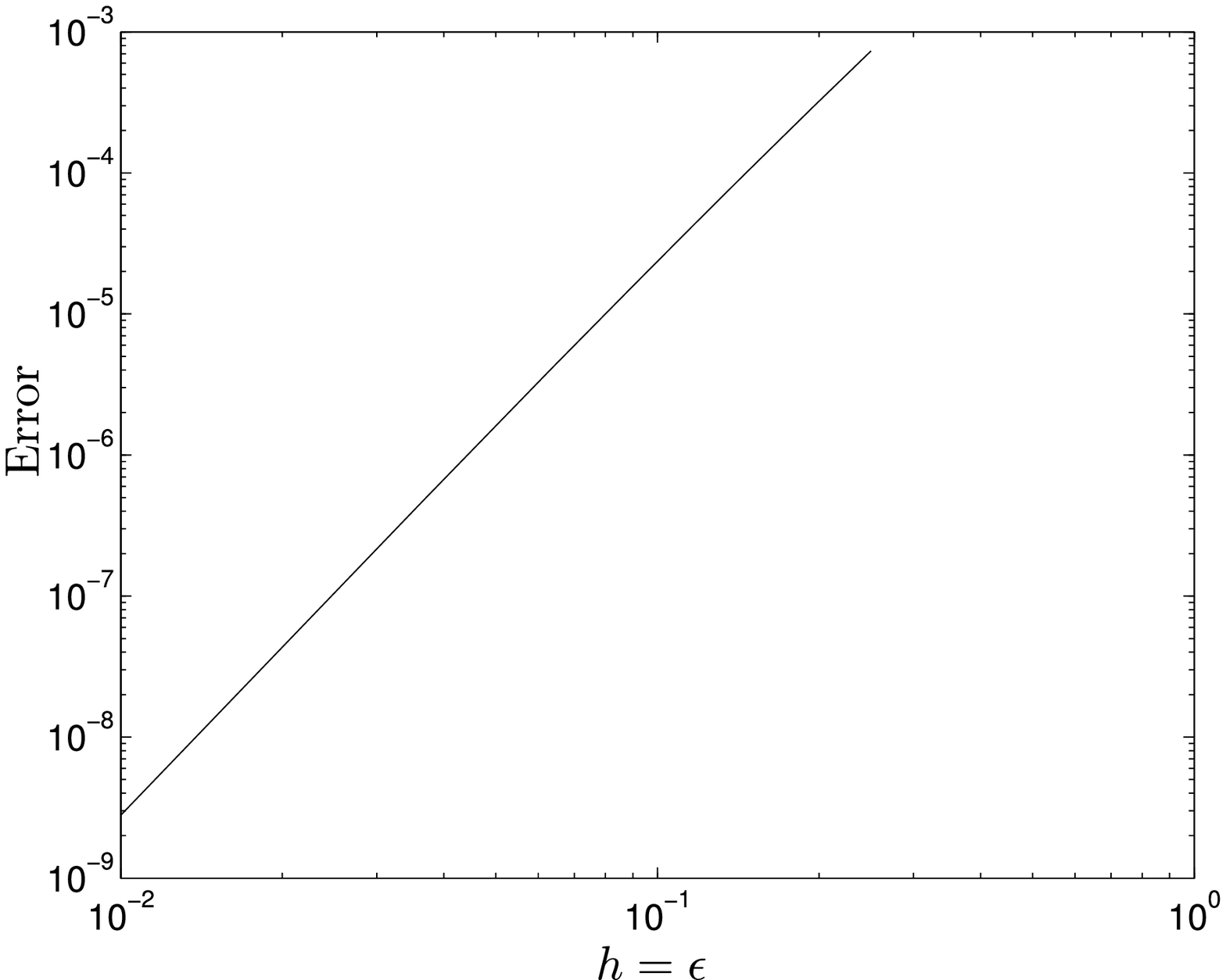}}
\end{center}
\caption{The errors $\Vert \eta^h-\eta\Vert$ (a) and $\Vert \phi^{\epsilon,h} - \phi^\epsilon\Vert$ (b) for the Lindemann mechanism \eqref{lm} at $x=1$ for $h=\epsilon$ and as a function of $\epsilon$. The slopes are $\approx 5$ and $\approx 4$. There is an improvement with respect to the estimates in \eqref{dsomr}$_{h=\epsilon}$ and \eqref{errorPhinh}$_{h=\epsilon}$ for this example. This is due to the fact that for this example the finite difference operator $\delta_x^h$ resolve the derivatives of the first terms in the asymptotic expansions of $\eta$ and $\phi^\epsilon$ exactly. \figlab{LindemannOrder}}
% /home/krkri/dtu/AProf/research/matlab/Collocation method/linear_example_2304/test_etafunc.m
\end{figure}

\figref{LindemannTest} (a) shows a trajectory computed using the modified Runge-Kutta scheme (thick line) for $\Delta \tau=0.01$, $\epsilon=0.1$ and $h=10^{-3}$. The thinner lines show the result of accurate backwards integration of initial conditions that were displayed by an amount of $\pm 10^{-4},\,\pm 10^{-5},\,\ldots,\,\pm 10^{-9}$ from the slow manifold along the stable direction. Of all the pairs only for the one with $\pm 10^{-9}$ do the trajectories jump in the same direction. The slow manifold is therefore expected to be correct up to $\pm 10^{-8}$ but not more accurate than $\pm 10^{-9}$. \figref{LindemannTest} (b) shows a connection (full thick line) to the base trajectory in \figref{LindemannTest} (a) (dotted line in \figref{LindemannTest} (b)) obtained using the \textsc{SO-SMST} method with $\Delta t=\Delta \tau=0.01$. Using $\diamond$'s this solution is compared with a solution obtained by direct forward integration. There is a good agreement between the two solutions. \figref{LindemannTest} (c) shows the maximal error between  accurate reference solutions, obtained using accurate forward integration, and trajectories computed using the \textsc{SO-SMST} method as function of the distance $r$ from the slow manifold. The slope of the straight line in the logarithmic scale is $\approx 2$ in agreement with \eqsref{epsr2}{Dy0}. The ``naive'' projection described in \eqref{Naive} assumes that the fast fibers are vertical. Applying this principle to this example, will therefore lead to an $\mathcal O(1)$-error with no improvement for $\epsilon\rightarrow 0$. % of the form $\mathcal O(y_0)$ with $y_0=y-\eta(x)$ describing the vertical distance. 

\begin{figure}[h!]
\begin{center}
\subfigure[]{\includegraphics[width=.475\textwidth]{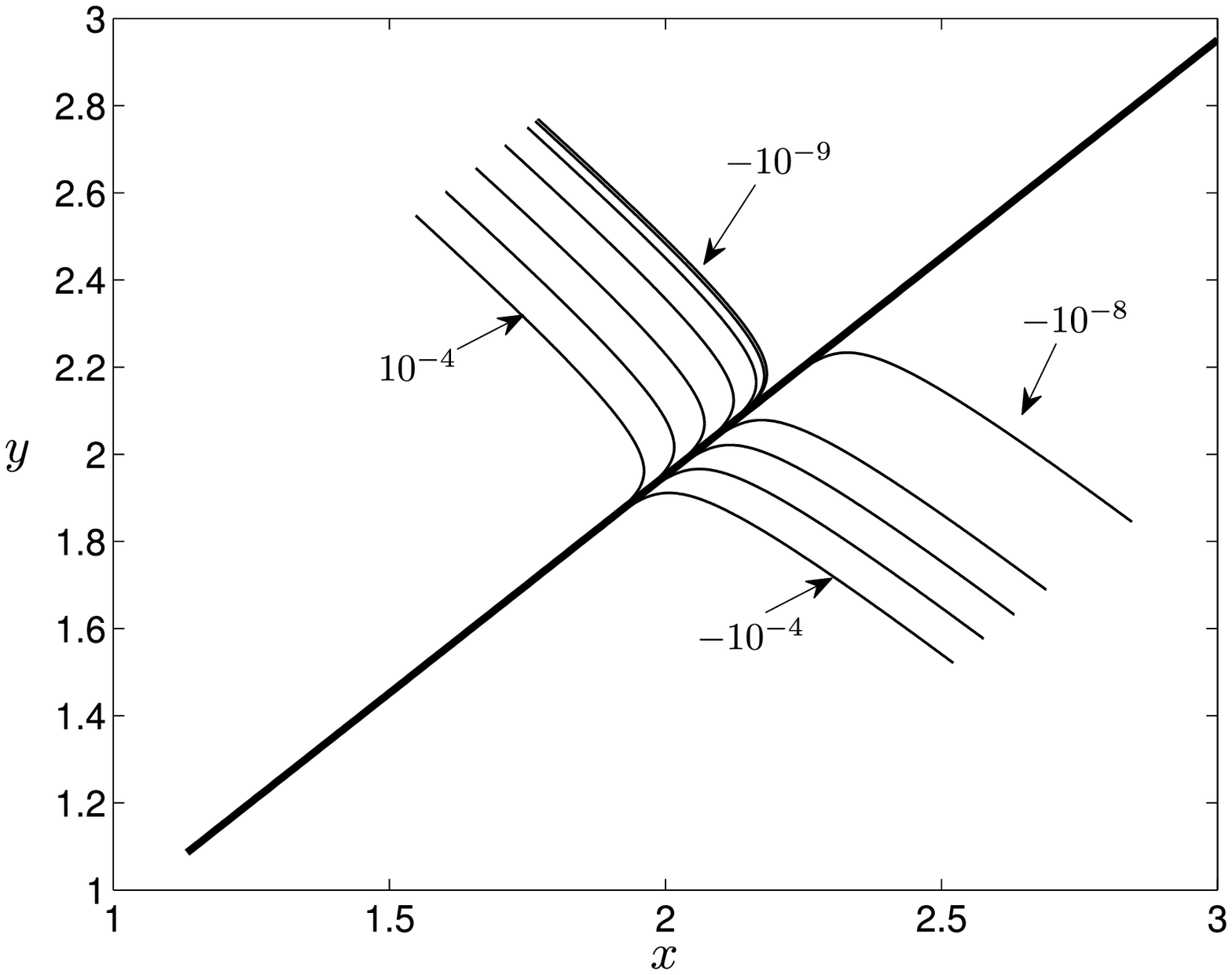}}
\subfigure[]{\includegraphics[width=.475\textwidth]{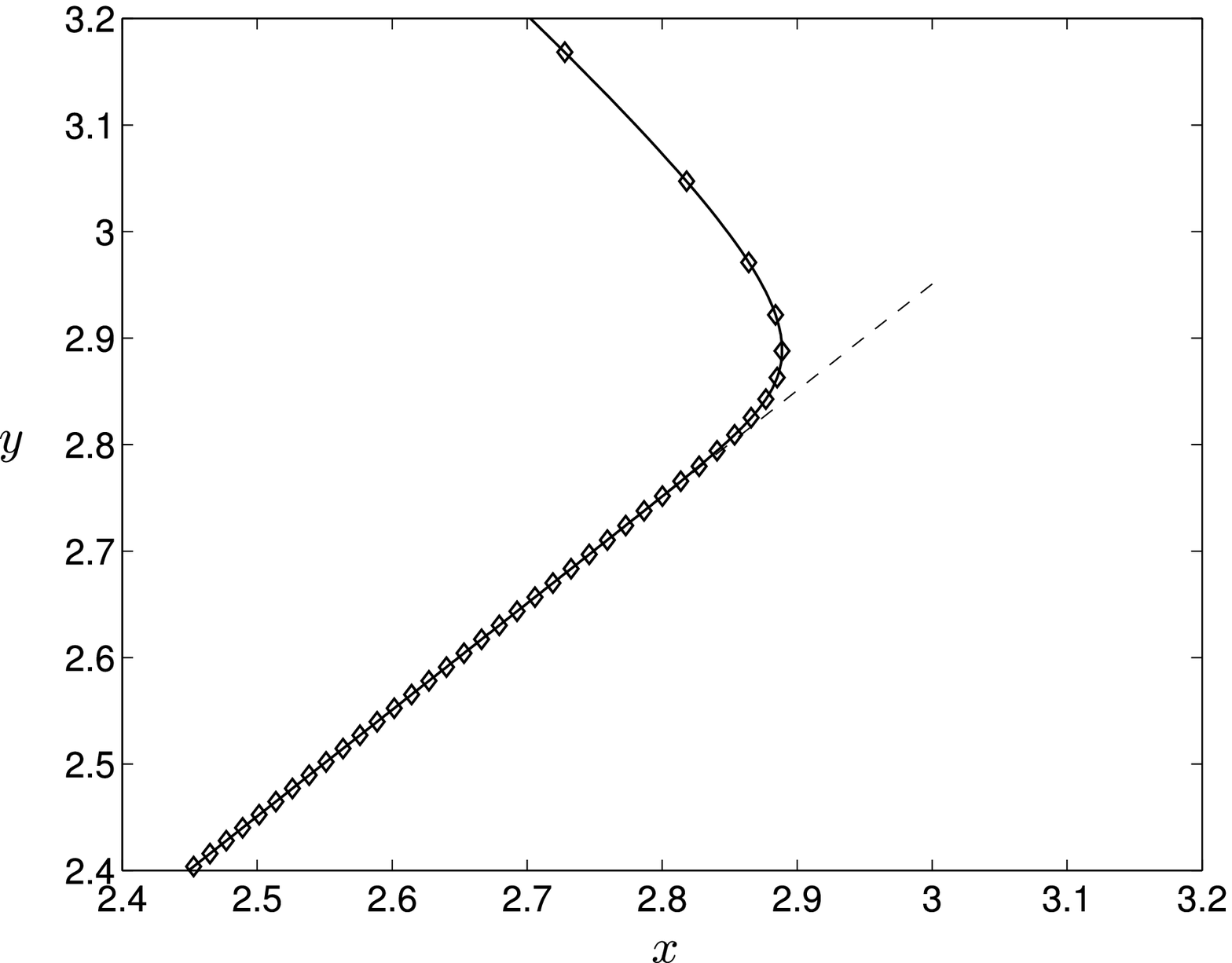}}
\subfigure[]{\includegraphics[width=.475\textwidth]{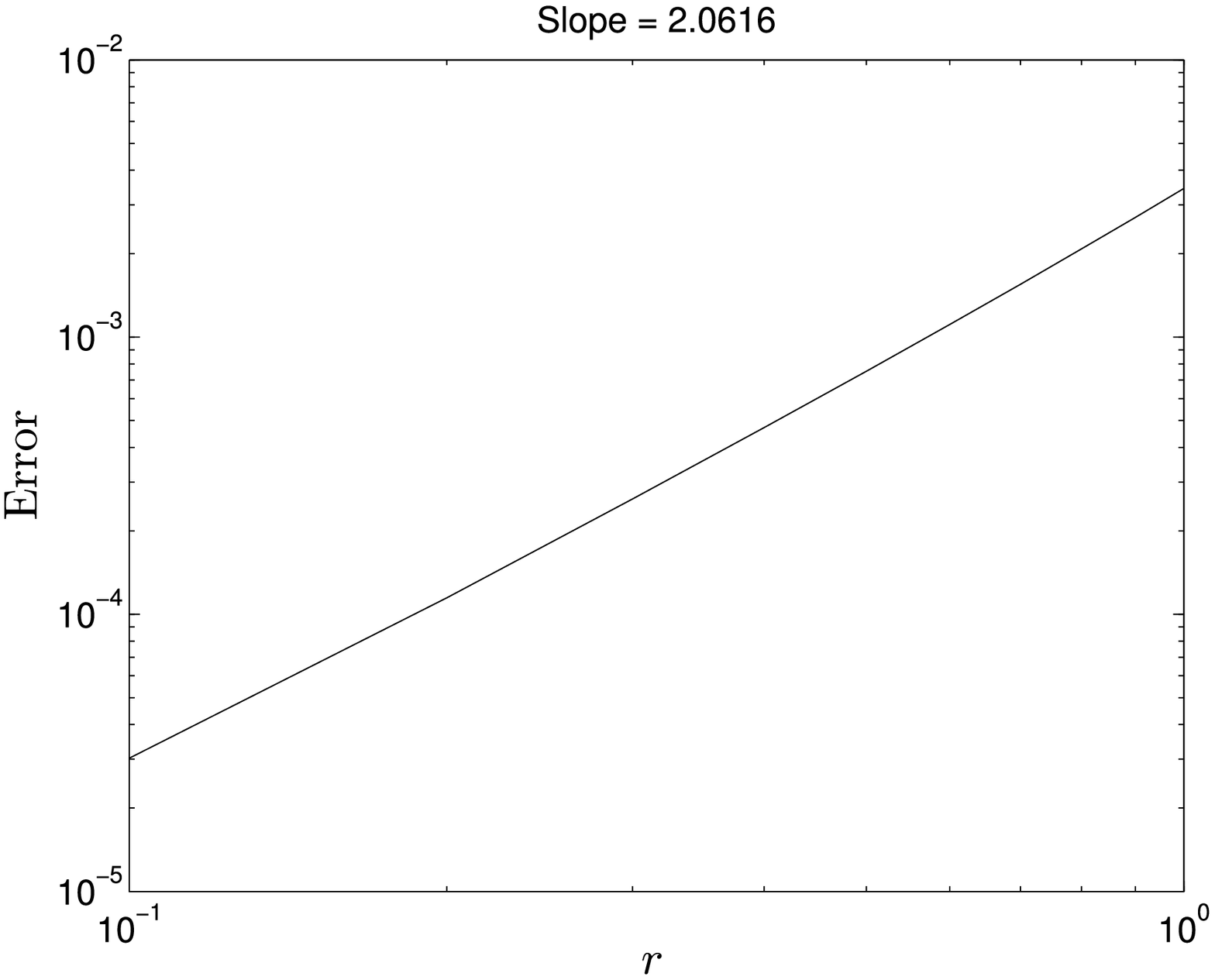}}
\caption{Figure (a) shows the result of accurate backwards integration of \eqref{lm} with $\epsilon=0.1$ for pairs of initial conditions displayed by an amount $\pm 10^{-4},\,\pm 10^{-5},\ldots,\pm 10^{-9}$ from the slow manifold, computed using \corref{etah2}, along the unstable direction. Of all the pairs only the one with $\pm 10^{-9}$ jump in the same direction. The slow-manifold is therefore expected to be correct up to $\pm 10^{-8}$. Figure (b) shows the result of applying the \textsc{SO-SMST} method (thick line) for the computation of a trajectory asymptotic to the slow manifold. The dotted line shows the slow manifold and the $\diamond$'s show the result of applying accurate direct forward integration. There is a good agreement with the $\diamond$'s and the thick line. Figure (c) shows the result of comparing direct integration with the result of applying the \textsc{SO-SMST} method for different initial distances $r$ from the slow manifold. In agreement with \eqsref{epsr2}{Dy0} the slope is $\approx 2$. \figlab{LindemannTest}}
\end{center}
% /home/krkri/dtu/AProf/research/matlab/Collocation method/linear_example_2304/test_etafunc.m
\end{figure}
\begin{remark}\remlab{LindemannPropertySOF}
{The system \eqref{lm} also exemplifies the importance of the property described in \remref{modSO}. To explain this, first note that the critical manifold $y=x$ of \eqref{lm} (or $z=w$ of \eqref{wzeqn}) is non hyperbolic at $x=0$. This manifests itself in the fact that if one approximates the slow manifold using e.g. asymptotic expansions, then the accuracy of the approximation will deteriorate for $x\rightarrow 0^+$. But since $(x,y)=0$ is actually an equilibrium of the system, and the \textsc{SO} approximation always includes equilibria of the system, the \textsc{SO} approximation for the slow manifold of \eqref{lm} goes through $(x,y)=0$ and the error of the approximation therefore improves near $x=0$. The \textsc{SOF} approximation is also well-defined up to $(x,y)=0$. See \cite[Section 8]{kri4}. }
%  Surprisingly, even though $y_1$ is not fast near $x_0=0$, the graph $y_1=0$ is still close to being invariant there as $\rho_1(0)=0$. The slow manifold obtained through the SO method always includes nearby equilibria. In particular, there is an improvement in the error as we approach $x_0=0$. The solutions shown in \cite{gou1} do not have this property as the truncation of the expansion about $\epsilon=0$:
% \begin{eqnarray*}
%  \rho_1(x_0)=\frac{3\epsilon^3}{x_0}+\mathcal O(\epsilon^4),
%  \end{eqnarray*}
%  does not preserve $\rho_1(0)=0$. 
\end{remark}

%% file: conclusion.tex
\section{Conclusion}
This paper has presented an alternative method for the computation of trajectories on saddle-type slow manifolds using iterative methods to approximate the slow manifold and its fiber projections. This included a numerical implementation of a modified \textsc{SO} method (also known as the iterative method of Fraser and Roussel) in a classical Runge-Kutta quadrature scheme for the computation of these unstable trajectories on the slow manifold. This part applies to other types of slow manifolds, even normally elliptic ones. 
For the computation of transients the \textsc{SOF} method was augmented to this quadrature scheme and a basic principle of splitting the problem into two non-stiff sub-problems was outlined and demonstrated on several examples, including a model of reciprocal inhibition and the FitzHugh-Nagumo model. This principle, which was named \textsc{SO-SMST}, benefits from the fact that the singular nature of the problem has been removed. {On the other hand, the \textsc{SO-SMST} method is disadvantaged by the fact that its accuracy is determined by $\epsilon$ alone. }%To achieve improved accuracy one will therefore have to decrease $\epsilon$.}  %There are therefore no problems encountered when letting $\epsilon$ become extremely small.

{Future research should further explore the use of the proposed method in applications. A promising area is believed to be $\epsilon$-free systems. In a ``real-life'' slow-fast systems one will typically not expect there to be an explicit small parameter (such as the Olsen model \cite{olsen1983}) and it may be very difficult (if not impossible) to write the system in the canonical form \eqref{fs}, see e.g. \cite{bro1,bro2,brokri1,eqn1}. The method presented here applies to such systems, as demonstrated in \secref{Lindemann}, and hence it could potentially provide a useful tool for numerical exploration of such systems. }% interesting area are for further exploration  }
% MORE HERE.
\section{Acknowledgement}
I would like to thank M. Br{\o}ns and S. J. Hogan for helpful discussions and providing valuable feedback in the preparation of this document. I also thank an anonymous referee for suggestions leading to an improved manuscript. 

%% file: errorApp.tex
\section{Error estimates for \textsc{SO-SMST}}\applab{Error}
%  Having described the trajectories considered, 
In this appendix, the error introduced by replacing \eqref{x0eqn}:
\begin{align*}
%  \begin{align}
 \dot x_0 &=\Lambda^\epsilon (x_0)+\mathcal O(\epsilon y_0^2),\\
 \dot y_0 &=A(x_0) y_0+\mathcal O(y_0^2).\nonumber
% \end{align}
\end{align*}
with \eqref{x0eqn2}:
\begin{align*}
 \dot x_0= \Lambda^\epsilon (x_0),
\end{align*}
in the \textsc{SO-SMST} is quantified. The set $y_0=0$ is here a saddle-type slow manifold.
\begin{proposition}\proplab{est}
 Suppose that $y_0=y_0(t)$ decays exponentially fast to the slow manifold $y_0=0$ in one end and escapes it exponentially fast at the other end. That is, assume that there exists a positive constant $\lambda$ so that
 \begin{align*}
  \Vert y_0(t)\Vert \le {r}\left(\exp(-\lambda t)+\exp(-\lambda (T/\epsilon-t))\right),
 \end{align*}
 for $t\in [0,T/\epsilon]$, where $r=\max\{\Vert y_{0}(0)\Vert, \Vert y_{0}(T/\epsilon)\Vert\}$. Then the error $\Delta x_0=\Delta x_0(\tau)$, taking $\Delta x_0(0)=0$, from replacing \eqref{x0eqn} by \eqref{x0eqn2} is 
 \begin{align}
 \mathcal O(\lambda^{-1} \epsilon r^2) \quad \mbox{for all $\tau \in [0,T]$}.\eqlab{epsr2}
 \end{align}
\end{proposition}
\begin{proof}
 Given that $\Delta x_0(0)=0$ and \eqref{x0eqn} then $\Delta x_0$ satisfies
 \begin{align*}
  \Delta x_0(\tau) & \le  L \int_0^{\tau} \Delta x_0(s) ds +C r^2 \int_0^{\tau} \left(\exp(-2\lambda \epsilon^{-1} s)+\exp(-2\lambda \epsilon^{-1}  (T-s))\right)ds\\
%   & \le L \int_0^{\tau} \Delta x_0(s) ds + \frac12 C \epsilon r^2 \left( \lambda_s^{-1} + \lambda_u^{-1}\right),  
&\le L \int_0^{\tau} \Delta x_0(s) ds + C \lambda^{-1} \epsilon r^2,
\end{align*}
with $L=\sup_x \Vert \partial_x \Lambda\Vert$ and some $C>0$, for $\epsilon$ and $r$ sufficiently small. Applying Gronwall's inequality in integral form \cite{Bellman} then gives 
\begin{align*}
%  \Delta x_0(\tau) \le \frac12 C \epsilon r^2 \left( \lambda_s^{-1} + \lambda_u^{-1}\right) \exp(L T)\quad \mbox{for all}\quad \tau \in [0,T],
 \Delta x_0(\tau) \le C \lambda^{-1} \epsilon r^2 \exp(L T)\quad \mbox{for all}\quad \tau \in [0,T],
\end{align*}
from which the result follows. 
\end{proof}

The error $\Delta x_0$ from replacing \eqref{x0eqn} by \eqref{x0eqn2} gives rise to an error $\Delta y_0$ in $y_0$. This error is described in the following proposition:
\begin{proposition}\proplab{est2}
  Let $\tilde y_0$ be the solution obtained of \eqref{y0eqn2} using \eqsref{x0eqn3}{xx0} and set $\Delta y_0=\tilde y_0 - y_0$. %and assume that there exists a positive constant $\lambda$ so that
%  \begin{align*}
%  \Vert  \tilde y_0 e^{\lambda t}\Vert, \Vert y_0 e^{\lambda t}\Vert \le r,
%  \end{align*}
% with $r=\text{max}\,\{\Vert y_0(0)\Vert ,\Vert \tilde y_0(0)\Vert \}$ for all $t\in [0,t_0]$. Then
Then there exist constants $C_1$ and $C_2$ independent of $\epsilon$ and $r$ so that:
\begin{align}
\Vert \Delta y_0(t)  \Vert\le \left(\Vert \Delta y_0(0)\Vert +C_1 \epsilon r^2 t_0\right)\exp(C_2 t_0),\eqlab{Dy0}
\end{align}
for all $t\in [0,t_0]$.
\end{proposition}
\begin{proof}
It directly follows that
\begin{align*}
 \Vert \Delta y_0(t) \Vert &\le \Vert \Delta y_0(0)\Vert + \int_0^t \int_0^1 \Vert \partial_x Y(x+s \Delta x_0,y_0+s \Delta y_0)\Vert \Vert  \Delta x_0 \Vert ds dt \\
 &+ \int_0^t \int_0^1 \Vert \partial_y Y(x+s \Delta x_0,y_0+s \Delta y_0) \Vert \Vert \Delta y_0 \Vert ds dt.
\end{align*}
Now use \eqref{epsr2} and Gronwall's inequality in integral form to obtain \eqref{Dy0}. %Set $C_1 = \sup_{x,y}\Vert \partial_x Y\Vert$ and $C_2 = \sup_{x,y}\Vert \partial_y Y\Vert$ and apply Gronwall's inequality in integral form to obtain \eqref{Dy0}.
\end{proof}

\begin{remark}\remlab{eps3app}
 For the trajectories computed in \cite{guc3} where the stable and unstable components are taken from the critical manifold, setting \eqref{scomp}$_{\epsilon=0}$ resp. \eqref{ucomp}$_{\epsilon=0}$ to $0$, this gives $r=\mathcal O(\epsilon)$ and errors in \eqsref{epsr2}{Dy0} (supposing in the latter case that $\Delta y_0(0)=0$) of order $\mathcal O(\epsilon^3)$. This is the proof of the statement in \remref{eps3}. 
\end{remark}